\documentclass[11pt]{amsart}
\usepackage{graphicx} 
\usepackage{amsfonts,amsmath,amssymb,amsthm,mathrsfs,mathdots}
\usepackage[left=3cm, top=2cm, bottom=2cm, right=3cm]{geometry}
\usepackage{tikz}
\usetikzlibrary{arrows.meta,bending,positioning,calc}

\usepackage[english]{hyperref}

\DeclareMathOperator{\trop}{trop}
\DeclareMathOperator{\Gr}{Gr}
\DeclareMathOperator{\pr}{pr}
\DeclareMathOperator{\val}{val}
\DeclareMathOperator{\cone}{cone}
\DeclareMathOperator{\sign}{sign}

\newcommand\barRR{\overline{\RR}}
\newcommand\flag{\mathcal{F}l}
\newcommand\shortFlag{\mathcal{F}l_{\text{short}} }
\newcommand\TT{{\mathbb T}}
\newcommand\PP{{\mathbb P}}

\newcommand\ZZ{{\mathbb Z}}
\newcommand\RR{{\mathbb R}}
\newcommand\CC{{\mathbb C}}
\newcommand\puiseux{\CC\{\!\!\{t\}\!\!\} }
\newcommand*{\1}{\text{\usefont{U}{bbold}{m}{n}1}}

\theoremstyle{plain}
    \newtheorem{theorem}{Theorem}
    \newtheorem{theoremalph}{Theorem}
    
    \newtheorem{corollary}[theorem]{Corollary}
    \newtheorem{lemma}[theorem]{Lemma}
    \newtheorem{proposition}[theorem]{Proposition}
    \newtheorem*{convention*}{Convention}
\theoremstyle{definition}
    \newtheorem{remark}[theorem]{Remark}
    \newtheorem{example}[theorem]{Example}
    \newtheorem{definition}[theorem]{Definition}

\title{Tropicalization and degeneration of short flag varieties and their fibers}
\author{Hannah Markwig, Michael Schl\"o\ss er}
\address[Hannah Markwig]{
    Universit\"at T\"ubingen,
    Fachbereich Mathematik,
    Auf der Morgenstelle 10,
    72076 T\"ubingen}
\email{hannah@math.uni-tuebingen.de}
\address[Michael Schl\"o\ss er]{
    RWTH Aachen, 
    Lehrstuhl f\"ur Algebra und Darstellungstheorie, 
    Pontdriesch 10-12, 52062 Aachen}
\email{schloesser@art.rwth-aachen.de}

\begin{document}
\begin{abstract}
    Tropicalizations of Grassmannians resp.\ flag varieties parametrizing (flags of) linear spaces have been studied intensely and reveal rich combinatorial structure. Linear degenerations of flag varieties provide ways of interpolating between products of Grassmannians and flag varieties. Tropical versions of these have been introduced recently \cite{BS23}, but their combinatorial description and computation remains a challenge. We consider a particular case which allows the use of matroidal methods, namely linear degenerate short flag varieties parametrizing tuples of linear spaces such that a projection of one is contained in the other. Their fibers turn out to be linear spaces themselves under mild assumptions, which allows a comprehensive combinatorial description of their tropicalizations, and accordingly, of tropicalizations of linear degenerate short flag varieties.
\end{abstract}


\maketitle

\section{Introduction}\label{section: 1 - intro}
Tropicalization can be viewed as a degeneration process which enables an exchange of methods between combinatorics and algebraic geometry.

Tropicalizations of linear spaces are the building blocks to understand smooth tropicalized varieties. They can be studied in terms of their underlying (valuated) matroid. Accordingly, the study of tropicalized linear spaces and their parameter spaces, the tropical Grassmannians, form a rich and intense area of study over recent years (see e.g.\ \cite{SS04a, Spe08, Rin13, Mun, BKKUV, SW21}).

Tropical Grassmannians, and, more generally, tropical flag varieties parametrizing flags of tropicalized linear spaces are in general very hard to compute, and understanding their combinatorial properties remains a major challenge for the coming years \cite{BBRS}.

Linear degenerate flag varieties arise from ordinary flag varieties by replacing inclusions for arbitrary linear maps, and have been studied extensively for example in \cite{Feigin2012}, \cite{CFFFM17} and \cite{CFFFM19}.
%
%
%
%
From our point of view they offer an interpolation between flag varieties and products of Grassmannians (whose tropicalizations are better understood compared to flag varieties), thus it is only natural to consider their tropicalization. From a theoretical perspective, this has been done previously in \cite{BS23}. Nevertheless, explicitly computing such tropicalizations and understanding their combinatorics remains difficult.

In this article, we approach this problem in the context of short linear degenerate flag varieties, by first considering their fibers. 
Given $S \subseteq [n]$, the short linear degenerate flag variety $\shortFlag^S$ parametrizes tuples $Y, X \subseteq \PP^{n-1}$ of $d-1$, resp.\ $d$-dimensional linear subvarieties such that $Y$ is contained in $X$ after projection onto the coordinates outside $S$.

Fixing $X$, and collecting all $Y$ from above, leads to the moduli space $L^S(X)$ of $S$-degenerate codimension-one subspaces of $X$. 
When $S$ is empty, it is shown in \cite{JMRS} that $L(X) = L^\emptyset(X)$ turns out to be linear itself, viewed as subvariety of $\PP^{\binom{n}{d}-1}$. This is desirable, because then tropicalization can again be studied in terms of tropicalized linear spaces.
In Theorem \ref{theorem: L^S(X) = L(X^S)} we show that, under some mild assumptions, this statement extends.

\begin{theoremalph}\label{theorem: A}
    Suppose $X \leq K^n$ is a linear subspace of dimension $d+1$ and $S \subset [n]$ is such that there exists a non vanishing Plücker coordinate $q_C$ of X, whose index set $C$ contains $S$. Then $L^S(X) = L(X^S)$ for a suitable vector space $X^S$ implying that $L^S(X)$ is a linear variety of dimension $d$. 
\end{theoremalph}

We furthermore extend the explicit descriptions of $L(X)$ as rowspace respectively kernel of suitable matrices found in \cite{JMRS} (see chapter \ref{section: 2 - L^S(X)}).

Whenever one investigates a vector space it is natural to study its associated linear matroid as well.
In \cite{JMRS}, the linear matroid of $L(X)$ is studied, especially for generic $X$.
We continue this study in Proposition \ref{proposition: matroid of L^S(X)}, where we show that linear degeneration leads to a direct sum decomposition of associated linear matroid.
For generic $X$, this leads to a complete description of the matroid associated to $L^S(X)$. 

Ultimately, our study of the matroid of $L^S(X)$ sets up the investigation of $\trop L^S(X)$, whose points describe tropicalized linear spaces such that their $S$-projection, defined by setting the coordinates indexed by $S$ to $\mathcal{\infty}$, is contained in $\mathcal{X} = \trop(X)$.

When $L^S(X)$ is linear, $\trop L^S(X)$ depends solely on its associated (valuated) matroid and if $X$ happens to have constant coefficients, the tropicalization of $L^S(X)$ arises as the (extended) Bergman fan of its matroid. 
For the general, non-constant coefficient case we describe the Plücker coordinates of $L^S(X)$ (and thus its valuated matroid) in terms of oriented hypergraphs. 
Note that since $L^S(X)$ is a linear subspace of $\PP^{\binom{n}{d}-1}$, its Plücker coordinates are indexed by a set of $d$-subsets. 
Thus it is only natural to interpret labels as hypergraphs.
Roughly speaking, a hypergraph can be oriented by picking a target inside each of its hyperedges. 
The remaining vertices of an edge are called sources, and any vertex appearing as a source of some hyperedge, but never as a target, is called a root. 
The Theorem below summarizes our description of Plücker coordinates (see chapter \ref{section: 2 - L^S(X)} for more details and notation).
\begin{theoremalph}\label{theorem: B}
    Let $X \leq K^n$ and $S \subseteq [n]$ as in Theorem \ref{theorem: A}.
    Let $J_1, \ldots, J_{d+1} \in \binom{[n]}{d}$ be distinct and sorted in lexicographic order. Let $H$ be the hypergraph on $[n]$ with hyperedges $h_i = [n] \setminus J_i$. Then the Plücker coordinate $p_{J_1 \ldots J_{d+1}}$ of $L^S(X)$ equals
    %
    \begin{equation*}
        p_{J_1 \ldots J_{d+1}} = 
        \begin{cases}
            \sum_{\mathcal{O}} \sign(I_0, \mathcal{O}) \prod_{i \neq i_0} q_{[n] \setminus s_{\mathcal{O}}(h_i)} & \text{if $S$ consists of leaves} \\
            0 & \text{otherwise}.
        \end{cases}
    \end{equation*}
    The sum runs over all orientations $\mathcal{O}$ of $H$ with roots given by $I_0$,
    where ${i_0} \in [d+1]$ and $I_0 \subset h_{i_0}$ of size $n-d-1$ can be chosen freely.
\end{theoremalph}
In the special case $d+2=n$, the hypergraphs above reduce to oriented graphs. This is to be expected, since in this case, for sufficiently generic $X$ the matroid associated to $L(X)$ is isomorphic to the graphical matroid of the complete graph on $n$ vertices.
Lastly, we also investigate how to obtain its valuated circuits in terms of oriented hypergraphs through Proposition \ref{proposition: valuated circuits of L^S(X)} and the subsequent discussion.

To come back to short linear degenerate flag varieties again, we need to vary $X$. 
We place a particular focus on the case where $d+2=n$, where we can give a complete description of the fan structure (in the different strata) of $\trop( \shortFlag^R)$. Here $R \subseteq [n]$.
 We show in Theorem \ref{theorem: trop (degen) short flag} that $\trop \shortFlag^R$ is built from tropicalizations of $L^S$-spaces as well as Grassmannians. Here our earlier observation $L^S(X) = L(X^S)$ plays a central role, as it allows for the case distinction necessary to derive our result.
In Theorem \ref{theorem: bergman fan for trop short flags} and the subsequent discussion we thoroughly describe the fan structure in each cell of $\trop \shortFlag^R$. Outside of cells containing a Grassmannian, the cones making up these fans  come from tuples of chains of flats. Depending on the precise cell these flats live in certain summands of the matroid associated to a suitable $L^S$-space.
Our results regarding $\trop( \shortFlag^R)$ can be summed up as follows (see Section \ref{section: 4 - short flags}):

\begin{theoremalph}\label{theorem: C}
    Suppose $d+2 = n$ and $R \subseteq [n]$. 
We can describe $\trop( \shortFlag^R)$ completely in terms of tropicalized $L^S(X)$-spaces and tropical Grassmannians. More precisely:
    
    Define $w^S$ by setting $w^S_I = 0$ whenever $S \subseteq I$, and otherwise set $w^S_I = \infty$.
    Then $\shortFlag^R$ tropicalizes to the disjoint union 
    \begin{align*}
        \trop( \shortFlag^R) = 
        &\bigcup_{R^C \subseteq S \subset [n]} ~ \trop (\Gr(n-2,n)) \times (\RR^{\binom{[n]}{n-1}}/\RR + w^S) ~~ \cup \\
        &\bigcup_{R^C \nsubseteq S \subset [n]} ~ A_S + \trop(L^{R \cup S}(X)) \times w^S.
    \end{align*}
    Here, $A_S$ is the subspace of $\RR^{\binom{[n]}{n-2}}/\RR \times \RR^{\binom{[n]}{n-1}}/\RR$
    generated by $\langle a_i \times e_{[n] \setminus i} ~:~ i \notin S \rangle_\RR, $ 
    with $(a_i)_J = |\{i\} \setminus J|$.
    Moreover, each nonempty stratum of $\trop (\shortFlag^R)$ has the structure of a fan arising either from phylogenetic trees or as the Bergman fan of a suitable matroid.
\end{theoremalph}




Notice that, following \cite{BS23}, both the tropicalization of a fiber $\trop(L^R(X))$ as well as the tropicalization of a short linear degenerate flag variety parametrize tuples of tropicalized linear spaces $(\trop(Y), \trop(X))$ such that $\pr_R(\trop(Y)) \subset \trop(X)$. We discuss several examples of tropicalizations and their property as parameter spaces.

We remark that understanding the tropicalizations of fibers of a short (linear degenerate) flag variety can also be viewed as a solution to the so-called codimension one relative realizability problem (see e.g.\ \cite{BK, BS, BGS, Gei}): there, we fix $X$ of dimension $d+1$ and ask: for which tropical linear spaces $\mathcal{Y}$ contained (resp.\ after projection) in $\trop(X)$ does there exist a $d$-dimensional linear space $Y$ with $\pr_R(Y)\subset X$ and $\trop(Y)=\mathcal{Y}$?
The answer is: all $\mathcal{Y}$ which appear in the tropicalization of $L^R(X)$. 
In this sense, our study can be viewed as an answer to the relative realizability problem in this situation. \\

The structure of the paper can be summarized as follows. The main body consists of three parts, Section \ref{section: 2 - L^S(X)}, \ref{section: 3 - Tropicalizations} and \ref{section: 4 - short flags}. 
The first part, Section \ref{section: 2 - L^S(X)}, focuses on the investigation of the locus $L^S(X)$ of $S$-degenerate codimension-one subspaces. 
After showing that this space is indeed linear, provided certain Plücker coordinates of $X$ do not vanish, we go on to describe its (valuated) matroid. 
In Subsection \ref{section: 2 - preliminaries} we recall Grassmannians, Plücker relations as well as incidence relations.
In Subsection \ref{section: 2 - Defs and empty S} we properly define the space $L^S(X)$, and collect the needed results from \cite{JMRS} on the case of $S = \emptyset$.
Subsection \ref{section: 2 - cocodim 1} deals with the special case of $X$ having codimension-one in its ambient space. 
Not only can we furnish a complete description in this case, which we dub cocodim $1$, but we shall find ourselves continuously coming back here to look for inspiration on how to deal with the higher codimension case. A special feature of cocodim $1$ is we require essentially only linear algebra to understand it.
The case of $X$ having arbitrary codimension is discussed in Subsection \ref{section: 2 - higher codim}. 
Contrary to cocodim $1$ some algebraic arguments are needed here, which are a direct generalization of their counterparts from \cite{JMRS}, and can be found in the Appendix.


The second part, Section \ref{section: 3 - Tropicalizations}, deals with the tropicalization of the moduli space $L^S(X)$. 
We introduce the necessary background on tropical geometry in Subsection \ref{section: 3 - Preliminaries on tropical geometry}.
In Subsection \ref{section: 3 - trop L^S(X)} we exploit Propositions \ref{proposition: matroid of L^S(X)} and \ref{theorem: Plücker coords of L^S(X)} to study the tropicalization of $L^S(X)$. 

The last part, Section \ref{section: 4 - short flags}, introduces the short linear degenerate flag variety $\shortFlag^R$ and discusses its tropicalization provided the flags have low enough codimension in their ambient space.
In Subsection \ref{section: 4 - covering short flags} we briefly explain how to cover such flag varieties by (torus orbits of) the moduli spaces $L^S(X)$ investigated in earlier sections.
The final Subsection \ref{section: 4 - trop short flags} contains the study of $\trop \shortFlag^R$.

\subsection{Acknowledgments}\label{section: 1 - acknowledgements}
We acknowlegde support by the Deutsche Forschungsgemeinschaft (DFG, German Research Foundation), Project-ID 286237555, TRR 195. 
We would like to thank Xin Fang and Ghislain Fourier for multiple useful discussions. 

\subsection{Notation and conventions}\label{section: 1 - notation}

We provide a table of notation extending \cite{JMRS}.

\begin{center}
    $\begin{array}{cc}
        [n] & \text{The positive integers up to } n \vspace{.4ex} \\
        \binom{[n]}{m} & \text{The subsets in } [n] \text{ of cardinality } m \vspace{.4ex} \\
        X & \text{The }e\text{-dimensional linear space in which we consider subspaces (mostly $e = d+1$). } \vspace{.4ex} \\
        q_I & \text{The Plücker coordinates of the linear space X} \vspace{.4ex} \\
        R_{A,B} & \text{The Plücker relation for } A \in \binom{[n]}{m-1}, B \in \binom{[n]}{m+2} \vspace{.4ex} \\
        \Gr(m,n) & \text{The Grassmannian parametrizing m-dimensional linear spaces in } K^n \vspace{.4ex} \\
        S & \text{A subset of } [n] \text{ indexing the coordinates we project away} \vspace{.4ex} \\
        L^S(X) & \text{The moduli space of d-dimensional } S \text{-degenerate subspaces of X, see Definition } \ref{definition: L^S spaces} \vspace{.4ex} \\
        P_J & \text{Variables for the Plücker coordinates of a } S \text{-degenerate subspace of X} \vspace{.4ex} \\
        I^S_{A,B} & \text{The } S \text{-incidence relation for } A \in \binom{[n]}{d-1}, B \in \binom{[n]}{d+2}, \text{ see Equation } \eqref{equation: S-incidence relations} \vspace{.4ex} \\
        W & \text{A matrix whose row space is X} \vspace{.4ex} \\
        V & \text{A matrix with certain minors of W as entries, see Equation } \ref{equation: the matrix V} \vspace{.4ex} \\
        U & \text{The matrix of incidence relations, see Equation } \ref{equation: the matrix U^S} \vspace{.4ex} \\
        X^S & \text{The } S \text{-degeneration of } X \vspace{.4ex} \\
        W^S & \text{A matrix whose row space is } X^S \vspace{.4ex} \\
        V^S & S \text{-degenerate analogue of } V \vspace{.4ex} \\
        U^S & S \text{-degenerate analogue of } U \vspace{.4ex} \\
        \mathcal{X} & \text{The tropicalization of } X \text{, see Example } \ref{example: trop L^S(X) for n=4}
    \end{array}$
\end{center} 

\vspace{1em}
\textbf{Conventions}:
The symbol $\subset$ is reserved for strict inclusions while we use $\subseteq$ whenever we also want to allow equalities.

If $Y$ is a subobject of $X$ (subgroup, linear subspace, subvariety etc.) we often just write $Y \leq X$, where again $<$ is reserved for proper subobjects.
There will always be enough context to distinguish this from the ordering on $\ZZ$. \\
Finally we take a few intuitive measures to increase readability. For example we suppress the notation for unions and omit parentheses when we only want to add a few elements. Similar conventions are used for removing elements. In such cases the union or difference of sets is meant to bind stronger than other set operations occurring in the same term.
For example $ Aij \cap B \setminus k = (A \cup \{i,j\}) \cap (B \setminus \{k\})$.
The notation $A^C$ denotes the complement of $A$ in its superset, which most of the time is $[n]$, but will always be clear from the context

\section{The locus of $S$-linear degenerate codimension-one subspaces and its associated matroid}\label{section: 2 - L^S(X)}


\subsection{Preliminaries}\label{section: 2 - preliminaries}
We begin by recalling basic facts about Grassmannians, Plücker coordinates, incidence relations as well as their linear degenerations. 

Let $K$ be an algebraically closed field. As as a set the \emph{Grassmannian} $\Gr(d,n)$ is the collection of all $d$-dimensional linear subspaces of $K^n$, or equivalently the set of all $(d-1)$-dimensional projective linear subspaces of $\PP^{n-1}$. It can be equipped with the structure of an projective algebraic variety by embedding it into $\PP^{\binom{n}{d}-1}$ via the \emph{Plücker embedding}. Write a $d$-dimensional subspace $L \subseteq K^n$ as row space of a full rank matrix $M \in K^{d \times n}$, then its vector of \emph{Plücker coordinates} 
$$ \det(M^J)_J \in \PP^{\binom{n}{d}-1} $$
is a well defined element in projective space, i.e.\ it does not depend on our choice of $M$. Here, $J$ runs over all subsets of $[n]$ of cardinality $d$, while $M^J$ denotes the quadratic submatrix of $M$ with columns indexed by $J$. 

Let $K[P_J: J \in \binom{[n]}{d}]$ be the homogeneous coordinate ring of $\PP^{\binom{n}{d}-1}$. Then the image of the Plücker embedding is the subvariety of $\PP^{\binom{n}{d}-1}$ cut out by the following relations:

For any $A \in \binom{[n]}{d-1}$ and $B \in \binom{[n]}{d+1}$, the \emph{Plücker relation} $R_{A,B}$ is
\begin{equation*}
    R_{A,B} = \sum_{i \in B \setminus A} (-1)^{|[i]\cap A| + |[i] \cap B|} P_{A \cup i} \, P_{B \setminus i} ~ . 
\end{equation*}

Now, let $d \leq e$ be nonnegative integers, and suppose $X$ is a fixed $e$-dimensional linear subspace of $K^n$, with Plücker coordinates $q_I$ for $I \in \tbinom{[n]}{e}$. \\

Let $f \in \text{End}(K^n)$ be a linear map. A $d$-dimensional linear subspace $L$ is \emph{$f$-linear degenerately contained in $X$} if $f(L) \subseteq X$. 
If $S \subseteq [n]$ and $f = \pr_S$ is the projection vanishing on the standard unit vectors $e_i \in K^n$ with $i \in S$, we say \emph{$L$ is $S$-linear degenerately contained in $X$} or more briefly a \emph{$S$-degenerate subspace of $X$}.

Denote the Plücker coordinates of $L$ by $P_J$ for $J \in \tbinom{[n]}{d}$. Then $\pr_S(L) \subseteq X$ holds if and only if the following linear incidence relations are equal to zero:
For any $A \in \tbinom{[n]}{d-1}$ and $B \in \tbinom{[n]}{e+1}$, 
the \emph{$S$-incidence relation} \cite{BS23,LW19} $I^S_{A,B}$ is
\begin{equation}\label{equation: S-incidence relations}
 I^S_{A,B} = \sum_{i \in B \setminus (A\cup S)} (-1)^{|[i]\cap A| + |[i] \cap B|} q_{B \setminus i} \, P_{A \cup i}\enspace .
\end{equation}

For the remainder of this paper we always assume $e = d+1$.
Moreover, we denote the \emph{Plücker ideal} (generated by all $R_{A,B}$) with $\mathcal{I}_{pl}$, and the \emph{$S$-incidence ideal} (generated by all $I_{A,B}^S$) with $\mathcal{I}_{in}^S$.

\subsection{$L^S$-spaces and the case $S = \emptyset$}\label{section: 2 - Defs and empty S}
Fix a $(d+1)$-dimensional subspace $X$ of $K^n$ with Plücker coordinates $q = (q_I)_{I \in \binom{[n]}{d+1}}$. We want to study the moduli space of $d$-dimensional $S$-degenerate subspaces of $X$. 

\begin{definition}\label{definition: L^S spaces}
    Let $X \leq K^n$ be of dimension $d+1$ and $S \subseteq [n]$. The $L^S$-space $L^S(X)$ of $X$ is the moduli space of $d$-dimensional $S$-degenerate subspaces of $X$. That is $$L^S(X) = \{Y \in \Gr(d,n) : \pr_S(Y) \leq X \}.$$
\end{definition}

Clearly, $L^S(X)$ is the subvariety of $\PP^{\binom{n}{d}-1}$ cut out by the $S$-incidence relations together with the Plücker relations. We do not claim, that in general these relations define a radical ideal, although especially in the cases we focus on later, this will be evident. \\

For the reader's convenience and easier reference, we collect some known results from \cite{JMRS}, dealing with the case where $S$ is empty. 
As it turns out in this case $L(X) = L^\emptyset(X)$ is a $d$-dimensional linear subvariety of $\PP^{\binom{n}{d}-1}$. Moreover it can be explicitly described as (the projectivization of) the row space of a matrix \cite[Theorem 1]{JMRS}.

First write $X$ as the row space of a matrix $W \in K^{(d+1) \times n}$.
Then $L(X)$ is the row space of the matrix $V$ (depending on $W$) with rows indexed by $[d+1]$ and columns indexed by $\binom{[n]}{d}$ defined via
\begin{equation}\label{equation: the matrix V}
    V_{i,J} = \det( e_i ~|~ W^J),
\end{equation}
where $( e_i ~|~ W^J)$ is the $(d+1) \times (d+1)$-matrix obtained by prepending the $i$-th unit vector to the submatix $W^J$ of $W$ with columns indexed by $J$.

$V$ can also be expressed in terms of the Plücker coordinates of $X$. Indeed, at least one Plücker coordinate, say $q_C$, of $X$ is nonzero. Identifying $[d+1]$ with $C$ in order preserving manner, we view the rows of $V$ as indexed by $C$.
Then we have 
\begin{equation}\label{equation: formula for V}
    V_{i,J} = 
    \begin{cases}
        (-1)^{|[i] \cap J|} \cdot q_{J i} & \text{if } i \notin J; \\
        0 & \text{otherwise}.
    \end{cases}
\end{equation}
(cf. \cite[proof of Proposition 9]{JMRS}). \\

There is a second way to describe the underlying vector space of $L(X)$. 
Since the incidence relations are linear, we may collect them in a matrix $U = U(X) = U(q)$ of dimensions $\binom{n}{d-1} \binom{n}{d+2} \times \binom{n}{d}$ such that each row contains the coefficients of a relation. 
More precisely, for $A \in \binom{[n]}{d-1}, B \in \binom{[n]}{d+2}$ and $J \in \binom{[n]}{d}$ we have

\begin{equation}\label{equation: the matrix U}
    U_{(A,B),J} = 
    \begin{cases}
    (-1)^{|[i] \cap A| + |[i] \cap B|}q_{B \setminus j} & \text{if } J = Ai \text{ with } i \in B \setminus A, \\
    0 & \text{otherwise}
    \end{cases}
\end{equation}

Clearly, $L(X) = \ker U \cap \Gr(d,n)$. By \cite[Corollary 2]{JMRS} we can omit the intersection with $\Gr(d,n)$ so that $L(X) = \ker U$. In particular, the Plücker relations are not needed to cut out $L(X) \subseteq \PP^{\binom{n}{d}}$. Up to radical, they are already implied by the incidence relations. \\
 
Putting everything together, we have the following Theorem.

\begin{theorem}[\cite{JMRS}]\label{theorem: L(X) is linear + V & U}
    For any $X \in \Gr(d+1,n)$, $L(X)$ is a $d$-dimensional linear projective subvariety of $\PP^{\binom{n}{d}-1}$. It can be described as the row space of $V$ defined in Equation \eqref{equation: the matrix V} as well as the kernel of $U$ as defined in Equation \eqref{equation: the matrix U}.
\end{theorem}

As a linear space, $L(X)$ defines a matroid $\mathcal{M}(L(X))$ given as the matroid of dependencies among the columns of the matrix $V$. Alternatively, it can also be described as the dual matroid to the matroid of dependencies among the columns of the matrix $U$, or as the matroid of lines in a certain hyperplane arrangement induced by $X$. See \cite[Theorem 6]{JMRS} for details.

It is interesting that for \emph{generic} $X$ taken from an open dense subset of $\Gr(d+1,n)$, the associated matroid $\mathcal{M}(L(X))$ is always the same. 

\begin{definition}\label{definition: generic spaces and matroids}
    Given $X \in \Gr(d+1,n)$ the Plücker coordinates of $L(X)$ are homogeneous polynomials in the Plücker coordinates $q_I$ of $X$. Some of these polynomials vanish for any $X$ due to Plücker relations. Requiring that all remaining polynomials as well as all $q_I$  be non-zero defines the open subset of \emph{generic spaces} in $\Gr(d+1,n)$. 
\end{definition}

It follows immediately from the definition that for $X, X' \in \Gr(d+1,n)$ with $X$ generic and $X'$ arbitrary, all bases of $\mathcal{M}(L(X'))$ are also bases of $\mathcal{M}(L(X))$ (in fact by Proposition \ref{theorem: Plücker coords of L^S(X)} this characterizes genericity). This shows that $\mathcal{M}(L(X))$ is constant on generic spaces. 
%
We can describe this matroid more explicitly using Dilworth truncations, which we quickly recall now. The following definition due to Brylawski is taken from \cite[Chapter 7, Exercise 7.55]{White1986}.

\begin{definition}\label{definition: dilworth truncation}
    Suppose $M$ is a rank $e$ matroid on the ground set $[n]$. For $1 \leq k \leq e$, the $k$-the \emph{Dilworth truncation} $D_k(M)$ is defined via independent sets as follows. The ground set of $D_k(M)$ is given by the $k$-flats of $M$. The independent sets of $D_k(M)$ are the sets $\{J_1, \ldots, J_m\}$ such that for any $s>0$ and $\{j_1, \ldots, j_s\} \subseteq [m]$, the union $\bigcup_{l=1}^s J_{j_l}$ has at least rank $s+k-1$ in $M$. The matroid $D_k(M)$ has rank $e-k+1$.
\end{definition}

Consider the $k$-th Dilworth truncation of the free matroid $U_{n,n}$ of rank $n$ on $n$ elements. Relabeling the elements of $D_k(U_{n,n})$ by their complement gives rise to a rank $n-k+1$ matroid $\tilde{D}_k(U_{n,n})$, called the \emph{relabeled Dilworth truncation} of $U_{n,n}$. 
Concretely a collection of $(n-k)$-subsets $J_1, \ldots, J_m$ of $[n]$ is independent in $\tilde{D}_k(U_{n,n})$ if and only if $|\bigcap_{l=1}^s J_{j_l}| \leq n-k+1-s$ for any $\{j_1, \ldots, j_s\} \subseteq [m]$. \\
 
\begin{theorem}[\cite{JMRS}]\label{theorem: matroid of L(X)}
    Let $X \leq K^n$ be a generic subspace of dimension $d+1$. Then the matroid $\mathcal{M}(L(X))$ is given precisely by $\tilde{D}_{n-d}(U_{n,n})$.
\end{theorem}

If $S$ is nonempty, $L^S(X)$ will in general not be a linear space. The trivial example is taking $S = [n]$, in which case $L^S(X) = \Gr(d,n)$. We can get  more precise by analyzing which of the $q_I$ and $P_J$ have a chance to appear in a $S$-incidence relation:

\begin{lemma}\label{lemma: occuring plückercoords}
    A nonzero Plücker coordinate $q_C$ (with $|C|=d+1$) appears in some $S$-incidence relation if and only if $|S \cap C| \geq |S| - (n-d-2)$. 
    On the other hand, a Plücker coordinate $P_J$ occurs in a $S$-incidence relation if and only if there is some $q_C \neq 0$ such that $J \setminus (C \cup S) \neq \emptyset$.
\end{lemma}

\begin{proof}
    We have $q_C$ is a coefficient in $I_{A,B}^S$ if and only if there is $j \in B \setminus (A \cup S)$ such that $C = B \setminus j$. This is the same as saying $B = Cj$ and $j \notin A$ for some $j \in [n] \setminus (S \cup C)$. In particular, $[n] \setminus (S \cup C) \neq \emptyset$. 
    
    On the other hand, given any $j \in [n] \setminus (S \cup C)$ we can always choose $A$ and $B$ such that $B = Cj$ and $j \notin A$. 
    We conclude $q_C$ appears in some $S$-incidence relation if and only if $[n] \setminus (S \cup C) \neq \emptyset$. This implies the first claim and the second statement is proven in similar fashion.
\end{proof}

As immediate consequence we get a bound on the cardinality of $S$ for which we can expect $L^S(X)$ to be a linear space:

\begin{lemma}\label{lemma: need pl.rels if |S| >= d+2}
    For $|S| \geq d+2$, the Plücker ideal $\mathcal{I}_{pl}$ is not contained in the radical of the $S$-incidence ideal $\mathcal{I}_{in}^S$. Thus, in general, the Plücker relations are needed to cut out $L^S(X)$.
\end{lemma}

\begin{proof}
    Any Plücker coordinate $P_J$ occurring in some $S$-incidence relation must satisfy $J \setminus S \neq \emptyset$. In other words, $\mathcal{I}_{in}^S$, and thus its radical, are completely independent of the Plücker coordinates $P_J$ with $J \subseteq S$.
    
    Now suppose $S = \{s_1, \ldots, s_{d+2}, \ldots\}$. We choose $A = \{s_1, \ldots, s_{d-1} \}$ and $B = \{s_2, \ldots, s_{d+2} \}$ and observe that the 3-term Plücker relation 
    $$R_{A,B} = \pm P_{A s_d} P_{B \setminus s_d} \pm P_{A s_{d+1}} P_{B \setminus s_{d+1}} \pm P_{A s_{d+2}} P_{B \setminus s_{d+2}}$$ only depends on Plücker coordinates $P_J$ with $J \subseteq S$. Hence $R_{A,B} \notin \text{rad}(\mathcal{I}_{in}^S)$.
\end{proof}

Just like in the case for $S = \emptyset$, we can collect the $S$-incidence relations in a matrix $U^S$ given by

\begin{equation}\label{equation: the matrix U^S}
    U^S_{(A,B),J} = 
    \begin{cases}
    (-1)^{|[i] \cap A| + |[i] \cap B|}q_{B \setminus j} & \text{if } J = Ai \text{ with } i \in B \setminus (A \cup S); \\
    0 & \text{otherwise.}
    \end{cases}
\end{equation}

Thus $U^S$ arises from $U$ by deleting entries. The above lemma tells us precisely which of the $q_C$ might still appear in $U^S$. \\

\begin{remark}\label{remark: q_C in U vs U^S}
    Caution must be taken: If $q_C \neq 0$ and $|S \cap C| \geq |S| - (n-d-2)$ then $U^S$ will in general contain $q_C$ less often than $U$. 
    
    Take for example $n = 5, d = 2, S = 12$ and $C = 134$. Then $q_{134}$ will appear in $U^S$, as $|S \cap C| = 1 = |S|-(n-d-2)$. However, we have 
    \begin{align*}
        I_{1,1234} &= q_{123}P_{14} - q_{124}P_{13} + q_{134}P_{12}, \\
        I_{1,1234}^S &= q_{123}P_{14} - q_{124}P_{13},
    \end{align*}
    and so $U^S_{(1,1234),12} = 0$, while $U_{(1,1234),12} = q_{134}$. 
    
    This shows that $U^S$ can usually not be obtained from $U$ by just deleting all $q_C$ with $|S \cap C| < |S| - (n-d-2)$.
\end{remark}

\subsection{Cocodim 1}\label{section: 2 - cocodim 1}

We now turn our attention to the case where $d+1 = \dim X = n-1$, and let $S \subseteq [n]$ arbitrary. Thus we study the moduli space of $S$-degenerate codimension one spaces in a codimension one space of $K^n$. This situation shall henceforth be dubbed as cocodim 1. \\

This situation is our prime case since our tools allow a comprehensive study. First, let us reconsider Lemma \ref{lemma: occuring plückercoords}. Under the given circumstances $|S \cap C| \geq |S| - (n-d-2)$ reduces to the condition $S \subseteq C$. Now suppose for such $C$, the Plücker coordinate $q_C$ appears in an incidence relation $I_{A,B}$. This means there is $j \in B \setminus A$ s.t. $C = B\setminus j$. Thus $j \notin S \subseteq C$ and so $j \in B \setminus (A \cup S)$. This shows that $q_C$ will still appear in the $S$-incidence relation $I^S_{A,B}$ and the problem described in Remark \ref{remark: q_C in U vs U^S} does not occur. Hence, $U^S$ arises from $U$ precisely by deleting each entry containing a $q_C$ not satisfying $S \subseteq C$. \\

Let us now define $q^S$ by setting to zero all entries $q_C$ of $q$ for which $S \subseteq C$ fails. Recall that $\text{Gr}(n-1,n) \cong \PP^{n-1}$. Thus, if $q^S \neq 0$, there is a codimension 1 vector space $X^S$ having $q^S$ as a Plücker vector. We call $X^S$ the \emph{$S$-degeneration} of $X$. In this case the previous paragraph can be rephrased as $U^S(q) = U(q^S)$.

\begin{theorem}\label{theorem: L^S(X) = L(X^S) in cocodim 1}
    Suppose $q$ and $S$ are chosen such that there is $q_C \neq 0$ with $S \subseteq C$. Then we have $L^S(X) = L(X^S)$ and $L^S(X)$ is a linear projective variety of dimension $n-2$, cut out solely by the $S$-incidence relations \eqref{equation: S-incidence relations}, i.e. $L^S(X) = \ker U^S$.
    On the other hand, if all coordinates $q_I$ with $S \subseteq I$ vanish, we have $L^S(X) = \textnormal{Gr}(n-2,n)$.
\end{theorem}

\begin{proof}
    If all $q_C$ with $S \subseteq C$ vanish, then by Lemma \ref{lemma: occuring plückercoords} so do all incidence relations. Hence we obtain $L^S(X) = \textnormal{Gr}(n-2,n)$.
    Now suppose we find $q_C \neq 0$ with $S \subseteq C$. In this case $U^S(q) = U(q^S)$ and we deduce
    \begin{align*}
        L^S(X)  \nonumber
        &= \ker U^S(q) \cap \text{Gr}(n-2,n) \\ 
        &= \ker U(q^S) \cap \text{Gr}(n-2,n) \\ 
        &= L(X^S). \nonumber
    \end{align*}
    This reduces our claims to Theorem 1 and Corollary 2 of \cite{JMRS}. 
\end{proof}

We shall (partially) extend this Theorem to the general setting later (see Theorem \ref{theorem: L^S(X) = L(X^S)}). Note that for generic $X$ we always know that $L^S(X) = L(X^S)$ is a linear space as all $q_I$ are nonzero. By Theorem \ref{theorem: L(X) is linear + V & U} it is the row space of $V^S := V(q^S)$, i.e. we just need to delete entries. In general, if $q_C \neq 0$ and $S \subseteq C$ then, by \eqref{equation: formula for V}, we can write 

\begin{equation}\label{equation: formula V^S in cocodim 1}
    V^S_{i,J} = \begin{cases}
    (-1)^{|[i] \cap J|} \cdot q_{J i} & \text{if } i \notin J \text{ and } S \subseteq J i \\
    0 & \text{otherwise}.
    \end{cases}
\end{equation}
 
\begin{example}
    Suppose $n = 4$, and $q_{123} \neq 0$. Then, by equation \eqref{equation: formula V^S in cocodim 1}, $L(X)$ is the row space of 
    $$
    V = 
    \begin{pmatrix}
        0 & 0 & 0 & q_{123} & q_{124} & q_{134}  \\
        0 & -q_{123} & -q_{124} & 0 & 0 & q_{234}  \\
        q_{123} & 0 & -q_{134} & 0 & -q_{234} & 0
    \end{pmatrix},
    $$
    where we ordered the columns lexicographically. To obtain the matrix $V^1$, which describes the space $L^1(X)$, we need only to set $q_{234}$ to zero, which yields
    $$
    V^{1} = 
    \begin{pmatrix}
        0 & 0 & 0 & q_{123} & q_{124} & q_{134}  \\
        0 & -q_{123} & -q_{124} & 0 & 0 & 0  \\
        q_{123} & 0 & -q_{134} & 0 & 0 & 0
    \end{pmatrix}.
    $$
    It is important to view $q$ in an appropriate affine open when degenerating. For example, say we additionally know $q_{124} \neq 0$. 
    We cannot get a matrix describing $L^{14}(X)$ from $V^1$ directly. Indeed, just deleting $q_{123}$ results in a lower rank matrix. Equation \eqref{equation: formula V^S in cocodim 1} tells us to view $q$ in the affine open determined by $q_{124} \neq 0$ instead: $\{1,4\} \not \subseteq \{1,2,3\}$, but clearly $\{1,4\} \subseteq \{1,2,4\}$. This gives
    $$
    V = 
    \begin{pmatrix}
        0 & 0 & 0 & q_{123} & q_{124} & q_{134}  \\
        0 & -q_{123} & -q_{124} & 0 & 0 & q_{234}  \\
        q_{124} & q_{134} & 0 & q_{234} & 0 & 0
    \end{pmatrix}
    $$
    and deleting $q_{234}$ and $q_{123}$ to obtain $V^{14}$ is no longer an issue. 
\end{example}

We notice that $V$ possesses a ``hook block structure'', and that passing to a linear degeneration means deleting the hooks from back to front. This is always the case if one orders the rows and columns of $V^S$ correctly. \\

Clearly, any element in $\Gr(n-1,n)$ is of the form $X^S$ for some generic $X$ and some $S \subset [n]$. Thus, if we want to describe the matroid of $L^S(X)$, Theorem \ref{theorem: L^S(X) = L(X^S) in cocodim 1} allows us to reduce to the generic case.
Note that for generic $X \leq K^n$ of codimension 1, the matroid $\mathcal{M}(L(X))$ is more easily understood as a graphical matroid cf. \cite[Example 15]{JMRS}. Indeed, by Theorem \ref{theorem: matroid of L(X)}, it is the relabeled Dilworth truncation $\tilde{D}_{2}(U_{n,n})$. If we forget the relabeling and instead consider $D_{2}(U_{n,n})$ directly,
a collection of $2$-subsets $J_1, \ldots, J_m$ in $[n]$ is independent if and only if $|\bigcup_{l=1}^s J_{j_l}| \geq s+1$ for any ${j_1, \ldots, j_s} \subseteq [m]$. Now interpret the elements of $[n]$ as vertices and $2$-subsets of $[n]$ as edges. The latter condition means precisely that a $J_1, \ldots, J_m$ are independent if and only if they do not contain a cycle. Hence $D_{2}(U_{n,n})$ is the graphical matroid $\mathcal{M}(K_n)$ of the complete graph on $n$ vertices. In particular, it is connected. \\

When $S \neq \emptyset$, the matrix $V^S$ has a block diagonal structure, so $\mathcal{M}(L^S(X))$ decomposes into connected components. 

\begin{proposition}\label{proposition: matroid of L^S(X) in cocodim 1}
    Suppose $X \leq K^n$ is generic of codimension $1$ and $S \subset [n]$. 
    Then the linear matroid $\mathcal{M}(L^S(X))$ decomposes into a direct sum 
    $$\mathcal{M}(L^S(X)) = \mathcal{M}_S \oplus \mathcal{M}_{S,0} \oplus \bigoplus_{a \in S} \mathcal{M}_{S,a} .$$
    Here, 
    \begin{itemize}
        \item $\mathcal{M}_S \cong \mathcal{M}(K_{[n]\setminus S})$ is the restriction of $\mathcal{M}(L(X))$ to those $(n-2)$-sets containing all of $S$;\item 
    $\mathcal{M}_{S,0}$ is the rank zero matroid on the $(n-2)$-sets missing two elements of $S$;
    \item
     for $a \in S$ we have a free rank one matroid $\mathcal{M}_{S,a}$ whose ground set consists of the $(n-2)$-sets containing $S \setminus a$ but not $a$. 
    \end{itemize}
\end{proposition}

\begin{proof}
    Let $E_S = \{ J : S \subseteq J \}, E_{S,0} = \{J : | S\setminus J| \geq 2\}$ and $E_{S,a} = \{J : S \setminus J = a \}$ for $a \in S$. Recall that $\mathcal{M}(L^S(X))$ is determined by the dependencies among columns of the matrix $V^S$.
    Without loss of generality, we may assume $S \subseteq [n-1]$ .
    Firstly, if $J \in E_{S,0}$, then $V^S_{i,J} = 0$ for any $i$. So $E_{S,0}$ contains only loops of $\mathcal{M}(L^S(X))$ as claimed. If $J \in E_{S,a}$ or some $a \in S$, then $V^S_{i,J} \neq 0$ if and only if $i = a$.
    Finally, let $J$ be in $E_S$ then $V^S_{i,J} \neq 0$ implies $i \notin S$. So up to reordering rows and columns we have a block diagonal matrix (with rectangular blocks). Moreover, the matrices $V^S$ and $V$ restricted to the columns in $E_S$ are the same. This proves the claimed decomposition. Viewed in terms of $\mathcal{M}(K_n)$ the restriction to $E_S$ leaves us precisely with the edges between vertices in $[n] \setminus S$.
\end{proof}

The last proposition implies that $\mathcal{M}(L^S(X))$ is graphical, whenever $L^S(X)$ is linear. We just need to add dipoles connected by parallel edges and singletons with loops attached to a suitable complete graph. \\

Provided $L^S(X)$ is linear, our next goal is to express its Plücker coordinates explicitly in terms of $q$. This gives rise to another interpretation of $\mathcal{M}(L^S(X))$.
First, assume $X$ to be generic and $S = \emptyset$.
The bases of $\mathcal{M}(L(X))$ correspond to spanning trees of $K_n$. Given such a spanning tree $T$, we have a unique path from $n$ to every other vertex, determining an orientation of $T$ ''away from $n$''. Thus any edge $a$ of $T$ has a \emph{source} $s(a)$ and a \emph{target} $t(a)$.
Let $a_1,\ldots,a_{n-1}$ be the edges of $T$ in reverse lexicographical order (where we interpret the edges as sets). Using the orientation of $T$, we can define the \emph{target permutation} $\pi_T = i \mapsto t(a_i)$ on the set $[n-1]$.
We furthermore denote the set of \emph{leaves} of $T$ by $l(T)$.

\begin{proposition}\label{proposition: Plücker coords of L^S(X) in cocodim 1}
    Let $X \leq K^n$ be of codimension $1$ and $S \subseteq [n]$ such that $L^S(X)$ is linear. Let $J_1, \ldots, J_{n-1} \in \binom{[n]}{n-2}$ be distinct and sorted in lexicographical order. Let $T$ be the subgraph of $K_n$ with edges $a_i = [n] \setminus J_i$, oriented away from $n$.
    Then the Plücker coordinate of $L^S(X)$ determined by the $J_i$ is
    \begin{equation}
        p_{J_1 \ldots J_{n-1}} = \begin{cases}
            \varepsilon_T \cdot \prod_{v \in [n]}q_{[n] \setminus v}^{\deg_T(v)-1} & \text{ if $T$ is a tree s.t. $S \subseteq l(T)$}, \\
            0 & \text{ otherwise,}
        \end{cases}
    \end{equation}
    where for a tree $T$, oriented away from $n$, the sign $\varepsilon_T$ is defined as 
    $$ \varepsilon_T = (-1)^{|\{i ~:~ s(a_i) ~<~ t(a_i)\}| }\sign(\pi_T). $$
    
\end{proposition}

\begin{proof}
    First assume $S = \emptyset$. It is enough to prove the claim on the open set of generic $X$, for which in particular all $q_I$ are assumed nonzero. Here we already know the Plücker coordinate $p_{J_1 \ldots J_{n-1}}$ is zero if $T$ is not a tree. Now assume $T$ is a tree.
    We use formula \eqref{equation: formula V^S in cocodim 1} to describe the entries of a matrix $V$ whose row space is $L(X)$. 
    Let $v_i$ denote the columns of $V$ indexed by $J_i$. Consider an edge incident to $n$, say $a_i = \{m,n\}$. Then the corresponding column $v_i$ is $(-1)^{m-1}q_{[n-1]}e_{m}$. This can be rewritten as
    $$ v_i = (-1)^{|[t(a_i)]\setminus a_i|}q_{[n] \setminus s(a_i)}e_{t(a_i)} .$$
    Note that we can bring any column $v_j$ to this form by a series of column reductions.
    If $a_j$ is adjacent to $a_i$, then we reduce $v_j$ with $v_i$ to bring it to the desired form. Now we just follow the paths starting at $n$ and reduce columns along the way. By multi-linearity of determinants this implies
    $$
    p_{J_1 \ldots J_{n-1}} = \sign(\pi_T)\prod_{i = 1}^{n-1} (-1)^{|[t(a_i)]\setminus a_i|}q_{[n] \setminus s(a_i)} .
    $$
    Any vertex $v < n$ has precisely $1$ incoming edge, so the number of edges $a_i$ with source $v$ is $\deg_T(v)-1$. For $v = n$ there are only outgoing edges. Hence the monomial in $p_{J_1 \ldots J_{n-1}}$ is 
    $$ q_{[n-1]}\cdot\prod_{v \in [n]}q_{[n] \setminus v}^{\deg_T(v)-1} .$$ 
    Since we can globally rescale by $q_{[n-1]}^{-1}$ this is what we need. As for the sign, note that 
    $$ (-1)^{|[t(a_i)]\setminus a_i|} = (-1)^{t(a_i) - 1 + [s(a_i) ~<~ t(a_i)] }, $$
    where we define $[\text{statement}] = 1$  if and only if the statement in the parenthesis is true. Any $v < n$ occurs precisely once as target of an edge, thus the sign in $p_{J_1 \ldots J_{n-1}}$ is 
    $$ (-1)^{\frac{(n-2)(n-1)}{2} + |\{i ~:~ s(a_i) ~<~ t(a_i)\}| }\sign(\pi_T) .$$
    Up to a global rescaling by $(-1)^{\frac{(n-2)(n-1)}{2}}$, this is what we claimed. To derive the claim for nonempty $S$, we use Theorem \ref{theorem: L^S(X) = L(X^S) in cocodim 1} and write $L^S(X) = L(X^S)$. Clearly a tree gives a vanishing Plücker coordinate for $L(X^S)$ if it contains a vertex $v \in S$ of degree greater than one.
\end{proof}

As immediate consequence it follows that in cocodim 1, the locus of generic spaces is precisely the common non vanishing set of the $q_I$. Moreover, if $X$ has trivial valuations, then so does $L^S(X)$, whenever it is a linear space.

\begin{example}\label{example: plücker coords for L(X) in cocodim 1}
    Let $n = 6$ and $d = 4$. Assume $X \leq K^6$ is of codimension $1$ and $S \subseteq [6]$ is such that $L^S(X)$ is linear. 
    We wish to calculate the Plücker coordinate $p_{J_1 \ldots J_5 }$ of $L^S(X)$ with indices $J_i$ given by: $1245, 1246, 1256,2345, 3456.$ The corresponding graph is a tree $T$ whose orientation away from $6$ is illustrated in Figure \ref{figure: plücker coord from tree} below. We obtain the permutation $\pi_T = (1 ~ 3 ~ 4)(5 ~ 2)$ with $\sign \pi_T = -1$. Moreover, we count $3$ increasing edges, so the overall sign will be $\varepsilon_T = 1$. As for the monomial only the vertices $1,3$ and $6$ contribute factors, since the others are leaves. Thus we get $p_{J_1 \ldots J_5 } = q_{12345}q_{12456}^2q_{23456}$. This Plücker coordinate vanishes if $S$ contains any of $1,3$ or $6$.
\end{example}

\begin{figure}[ht]
\begin{center}

\begin{tikzpicture}[scale = 1.5,
                    edge/.style={black, line width=0.9pt, line cap=round},
                    color = {black}]
  

\tikzstyle{node}=[text=black, inner sep=3pt, rectangle, rounded corners=3pt,fill=white, draw=none]
\tikzstyle{every path}=[shorten <=1mm, shorten >=2.5mm]

\coordinate (1) at (-.8,0.2);
\coordinate (2) at (-1.6,-.6);
\coordinate (3) at (.8,0.2);
\coordinate (4) at (1.6,-.6);
\coordinate (5) at (0,-.6);
\coordinate (6) at (0,1);

\draw[-Stealth] (6) -- (1) node [midway, above left=-2pt] {$a_4$};
\draw[-Stealth] (6) -- (3) node [midway, above right=-2pt] {$a_1$};
\draw[-Stealth] (1) -- (2) node [midway, above left=-2pt] {$a_5$};
\draw[-Stealth] (3) -- (4) node [midway, above right=-2pt] {$a_3$};
\draw[-Stealth] (3) -- (5) node [midway, above left=-2pt] {$a_2$};

\node[node] at (1) {\small $1$};
\node[node] at (2) {\small $2$};
\node[node] at (3) {\small $3$};
\node[node] at (4) {\small $4$};
\node[node] at (5) {\small $5$};
\node[node] at (6) {\small $6$};
   
\end{tikzpicture}

\caption{An oriented tree corresponding to a Plücker coordinate of $L(X)$.}\label{figure: plücker coord from tree}

\end{center}
\end{figure}
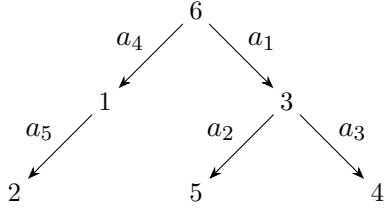

Proposition \ref{proposition: Plücker coords of L^S(X) in cocodim 1} says that the matroid associated to $L^S(X)$ arises from $\mathcal{M}(K_n)$ essentially by fixing certain vertices to be leaves.

\begin{corollary}
    Let $X \leq K^n$ be generic of codimension $1$ and $S \subseteq [n]$. 
    The bases of $\mathcal{M}(L^S(X))$ correspond to spanning trees of the complete graph $K_n$ on $n$ vertices s.t. any $v \in S$ is a leaf.
\end{corollary}

\subsection{$L^S$-spaces in higher codimension}\label{section: 2 - higher codim}

We now investigate the situation for $L^S$-spaces where $X$ is of dimension $d+1$ possibly smaller than $n-1$.
Our starting point is the algebraic result \cite[Theorem 3]{JMRS}, which we generalize motivated by our results in cocodim 1.

\begin{theorem}\label{theorem: algebraic statement}
    Replace the $q_I$ in the $S$-incidence relations by variables $Q_I$ (see Equation
    \eqref{equation: S-incidence relations}) and assume $S \subseteq C$ for some $C \in \binom{[n]}{d+1}$. Then the Plücker relations are contained in the saturation of the $S$-incidence ideal by $Q_C$. Thus $\mathcal{I}_{pl} \subseteq (\mathcal{I}_{in}^S:\langle Q_C\rangle^\infty)$.
\end{theorem}

The proof of this theorem consists of a double induction using a series of technical lemmata which are direct generalizations of their counterparts in the $S = \emptyset$ case. We carefully restate them in the Appendix, but defer the reader to \cite{JMRS} for proofs, as they only require small adjustments. \\

As a corollary we obtain a partial generalization of Theorem \ref{theorem: L^S(X) = L(X^S) in cocodim 1}. 

\begin{corollary}\label{corollary: L^S(X) is linear}
    Suppose $q$ and $S$ are such that there is $q_C \neq 0$ with $S \subseteq C$. Then $L^S(X)$ is a linear projective variety of dimension at most $d$, cut out by the $S$-incidence relations \eqref{equation: S-incidence relations}, i.e. $L^S(X)$ equals the kernel of $U^S$.
\end{corollary}

\begin{proof}
    Let $N = \binom{n}{d}-1$ and $M = \binom{n}{d+1}-1$. We view both $\mathcal{I}_{in}^S$ and $\mathcal{I}_{pl}$ as ideals in the multi-homogeneous coordinate ring of $\PP^N \times \PP^M$ i.e.\ the polynomial ring generated by the variables $P_I$ and  $Q_J$ with $I \in \binom{[n]}{d}, J \in \binom{[n]}{d+1}$. We write $\mathcal{I}_{in}^S(q)$ to mean the ideal in $K[P_I : I \in \binom{[n]}{d}]$ obtained from $\mathcal{I}_{in}^S$ by specializing the variables $Q_I$ to the Plücker coordinates $q_I$ of $X$.
    
    The geometry behind the containment of ideals in Theorem \ref{theorem: algebraic statement} is 
    $$ \Gr(d,n) \times \PP^M \supseteq \mathcal{V}_{\PP^N \times \PP^M}(\mathcal{I}_{in}^S) \setminus \PP^N \times \mathcal{V}_{\PP^M}(Q_c).$$
    Let $\pi_1, \pi_2$ denote the projections of $\PP^N \times \PP^M$ onto the first and second factor respectively. Then we have 
    $$ \mathcal{V}_{\PP^M}(\mathcal{I}_{in}^S(q)) = \pi_1(\mathcal{V}_{\PP^N \times \PP^M}(\mathcal{I}_{in}^S) \cap \pi_2^{-1}(q) ) \subseteq \Gr(d,n).$$
    Thus $$ L^S(X) = \mathcal{V}_{\PP^M}(\mathcal{I}_{in}^S(q)) \cap \Gr(d,n) = \mathcal{V}_{\PP^M}(\mathcal{I}_{in}^S(q)) $$
    is a linear space. Of course we can rewrite $L^S(X)$ as kernel of the matrix $U^S$. The claim about the dimension follows by identifying a lower triangular matrix submatrix of appropriate size in $U^S$ with diagonal entries $\pm q_C$. This is done exactly as in \cite[Corollary 2]{JMRS}.  
\end{proof}

We follow our approach in cocodim 1 as much as possible. For this, we again write $L^S(X)$ as the row space of a suitable matrix $V^S$. Suppose $X$ satisfies the premise of the last corollary, then it can be written as row space of $W \in K^{d+1 \times n}$ s.t.\ the submatrix of $W$ with columns indexed by $C$ is the identity matrix. We modify this matrix to define $W^S$ row-wise as follows. For $i \in [d+1]$ we set
\begin{equation}\label{definition: W^S}
    W^S_{i,-} = 
    \begin{cases}
        e_i^T   & \text{if } i \in S, \\
        W_{i,-} & \text{otherwise},
    \end{cases}
\end{equation}
where $e_i \in K^n$ is the standard unit vector and $W_{i,-}$ denotes the $i$-th row of $W$ and similar for $W^S$. The submatrix of $W^S$ with columns indexed by $C$ is still the identity matrix, so the row space $X^S$ of $W^S$ is of dimension $d+1$.

\begin{theorem}\label{theorem: L^S(X) = L(X^S)}
    Suppose $X$ and $S$ are such that there is $q_C \neq 0$ with $S \subseteq C$. Then in the above notation $L^S(X) = L(X^S)$. Thus $L^S(X)$ is of dimension $d$ and may be written as the row space of $V^S \in K^{d+1 \times \binom{n}{d}}$ defined via
    \begin{equation*}
        V^S_{i,J} = \det(e_i | (W^S)^J).
    \end{equation*}
\end{theorem}

\begin{proof}
    For dimensional reasons it suffices to prove the inclusion $L^S(X) \supseteq L(X^S)$. Consider the spaces $Y_j$ for $j = 1, \ldots, d+1$ defined as the row space of $W^S$ without its $j$-th row. Recall that $\pr_S$ is the projection vanishing on the coordinates indexed by $S$.
    Now clearly $\pr_S(Y_j) = \langle ~W_{i,-} : i \notin Sj ~\rangle_K \leq X$ so $Y_j \in L^S(X)$. 
    Note that the Plücker vector of $Y_j$ is precisely the $j$-th row of $V^S$, i.e. $Y_1, \ldots, Y_{d+1}$ form a basis of $L(X^S)$ considered as vector space.
\end{proof}

We remark that the formula \eqref{equation: formula V^S in cocodim 1} remains valid. Indeed, it is readily verified that just like in cocodim 1, the Plücker coordinates of $X^S$ are given by $q^S$ which is obtained from $q$ by setting all $q_I$ to zero for which $S \subseteq I$ fails.
Hence Proposition \ref{proposition: matroid of L^S(X) in cocodim 1} immediately generalizes to the case at hand. 
Note however that by far not every element in $\Gr(d+1,n)$ is the $S$-degeneration $X^S$ of some generic $X$. 
Thus we need to take a little care when generalizing Proposition \ref{proposition: matroid of L^S(X) in cocodim 1}:

\begin{proposition}\label{proposition: matroid of L^S(X)}
    Suppose $X$ is of dimension $d+1$ and $S \subseteq [n]$ such that there is $q_C \neq 0$ with $S \subseteq C$. Then the linear matroid $\mathcal{M}(L^S(X))$ decomposes into a direct sum 
    $$\mathcal{M}(L^S(X)) = \mathcal{M}_S \oplus \mathcal{M}_{S,0} \oplus \bigoplus_{a \in S} \mathcal{M}_{S,a} .$$
    Here, $\mathcal{M}_S$ is the restriction of $\mathcal{M}(L(X))$ to those $d$-sets containing all of $S$.
    $\mathcal{M}_{S,0}$ is the rank zero matroid on the $d$-sets missing at least two elements of $S$. 
    Finally, for $a \in S$ we have a rank one matroid $\mathcal{M}_{S,a}$ whose ground set consists of the $d$-sets missing precisely $a$. An element $J$ of $\mathcal{M}_{S,a}$ is independent if and only if $q_{Ja} \neq 0$.
\end{proposition}

When $X$ is generic, the $\mathcal{M}_{S,a}$ above are of course again free. 
In particular, the matroid $\mathcal{M}(L^S(X))$ is constant on the dense open subset of generic spaces in $\Gr(d+1,n)$. \\

Next, we wish to generalize our description of the Plücker coordinates for $L^S$-spaces from cocodim 1. This will be done in several steps, and requires the use of hypergraphs. Recall that by Theorem \ref{theorem: matroid of L(X)}, the matroid of $L(X)$ is the relabeled Dilworth truncation $\tilde{D}_{n-d}(U_{n,n})$ of the uniform rank $n$ matroid $U_{n,n}$ as long as $X$ is generic. We take complements to undo the relabeling
and regard the elements of $\text{D}_{n-d}(U_{n,n})$ as hyperedges of size $n-d$ on the vertices $1, \ldots, n$.
Some notions for hypergraphs directly transfer from ordinary graphs. For example vertices still have degrees so we can define the \emph{leaves} $l(H)$ of a hypergraph $H$ as before. Additionally we make the following definitions.

\begin{definition}\label{definition: orientations and related notions for hypergraphs}
    Let $H$ be a $(n-d)$-uniform hypergraph with vertices $1, \ldots, n$ and hyperedges $h_1, \ldots, h_k$ of size $n-d$ listed in reverse lexicographical order. Any $h_i$ may be directed by choosing a \emph{target} $t(h_i) \in h_i$. Then the \emph{sources} $s(h_i)$ of $h_i$ are defined to be $h_i \setminus t(h_i)$. An \emph{orientation} $\mathcal{O}$ of $H$ arises by directing all its edges, such that their targets are pairwise disjoint. In this situation, we emphasize $\mathcal{O}$ as subscript in sources and targets of hyperedges.
    Moreover we define the \emph{sources} of $\mathcal{O}$ via $s(\mathcal{O}) = \bigcup_{h \in H}s_\mathcal{O}(h)$ and likewise its \emph{targets} $t(\mathcal{O}) = \bigcup_{h \in H}t_\mathcal{O}(h)$. 
    Then the \emph{roots} of $\mathcal{O}$ are defined by $r(\mathcal{O}) = s(\mathcal{O}) \setminus t(\mathcal{O})$.
    The targets of the $h_i$ with respect to $\mathcal{O}$ give rise to the \emph{target map} 
    $$
    \pi_\mathcal{O}:[k] \to [n] \setminus r(\mathcal{O}); i \mapsto t_\mathcal{O}(h_i) .
    $$
    Finally we define the \emph{number of increasing edges} in $\mathcal{O}$ by counting each hyperedge $h_i$ with \emph{multiplicity} $|s_\mathcal{O}(h_i) \cap [t_\mathcal{O}(h_i)]|$.
\end{definition}

Note that we have a well-defined notion of $\sign \pi_\mathcal{O}$ defined by raising $-1$ to the number of inversions of $\pi_\mathcal{O}$. 
This coincides with the sign of the permutation obtained from $\pi_\mathcal{O}$ by identifying its image with $[k]$ in order preserving fashion.

\begin{lemma}\label{lemma: Plücker coords of L(X) with fixed I_0}
    Fix $I_0 \subseteq [n]$ of size $n-d-1$ and suppose $q_{[n]\setminus I_0} \neq 0$. Let $J_1, \ldots, J_{d+1} \in \binom{[n]}{d}$ be distinct and sorted in lexicographically order. Let $H$ be the hypergraph on $[n]$ with edges $h_i = [n] \setminus J_i$. Then the Plücker coordinate of $L(X)$ determined by the $J_i$ may be written as
    \begin{equation}\label{equation: Plücker coords of L^S(X) with fixed I_0}
        p_{J_1 \ldots J_{d+1}} = \sum_\mathcal{O} \varepsilon_\mathcal{O}  \prod_{i=1}^{d+1} q_{[n] \setminus s_{\mathcal{O}}(h_i)} ,  
    \end{equation}
    where $\mathcal{O}$ runs over all orientations of $H$ with roots $I_0$. For any such orientation, $\mathcal{O}$ we set
    $$ 
    \varepsilon_\mathcal{O} = 
    (-1)^{ \text{ \# increasing edges in $\mathcal{O}$}}
    \sign(\pi_\mathcal{O}). 
    $$
\end{lemma}

\begin{proof}
    We assume without loss of generality that $I_0 = [n] \setminus [d+1]$. This forces all occurring target maps to be permutations of $[d+1]$.
    It suffices to prove the claimed formula on the open dense subset of generic $X$. Then linear matroid of $L(X)$ is the Dilworth truncation $\text{D}_{n-d}(U_{n,n})$.
    If $\bigcup_i h_i \neq [n]$, then for cardinality reasons, there are no orientations $\mathcal{O}$ as required. Hence our formula gives $0$. This is correct because the $h_i$ are easily seen to be dependent in $\text{D}_{n-d}(U_{n,n})$.
    Now assume $\bigcup_i h_i = [n]$.
    Recall that $L(X)$ is the row space of the matrix $V$ determined in equation \eqref{equation: formula for V}. Thus
    $$ 
    V_{i, J_l} = \begin{cases}
        (-1)^{|[i] \cap J_l|}q_{J_l i} & i \in h_l \\
        0 & else.
    \end{cases}
    $$
    Now
    $$
    p_{J_1 \ldots J_{d+1}} =
    \det(V_{-,J_1} | \ldots | V_{-,J_{d+1}}) =
    \sum_{\sigma \in S_{d+1}} \sign \sigma \prod_{i=1}^{d+1} V_{\sigma(i), J_i} .
    $$
    The latter product can be nonzero for a given $\sigma$ only when $\sigma(i) \in h_i$ for all $i$. Directing each edge $h_i$ towards $\sigma(i)$ gives an orientation $\mathcal{O}$ with roots $r(\mathcal{O}) = I_0$, since $\bigcup_i h_i = [n]$. Clearly, $\pi_\mathcal{O} = \sigma$ and any $\sigma$ sending each $i$ into $h_i$ arises this way. Now for any orientation $\mathcal{O}$ as above we have 
    $J_i \cup t_\mathcal{O}(h_i) = [n] \setminus s_\mathcal{O}(h_i)$. Hence
    $$
    V_{\pi_\mathcal{O}(i),J_i} = 
    (-1)^{|[t_\mathcal{O}(h_i)] \setminus h_i|} \cdot q_{[n] \setminus s_\mathcal{O}(h_i)}.
    $$
    For the signs note
    $$
    \prod_{i=1}^{d+1} (-1)^{|[t_\mathcal{O}(h_i)] \setminus h_i |} = 
    \prod_{i=1}^{d+1} (-1)^{t_\mathcal{O}(h_i)-1 + |s_\mathcal{O}(h_i) \cap [t_\mathcal{O}(h_i)] |} =
    (-1)^{\frac{d(d+1)}{2}}\prod_{i=1}^{d+1} (-1)^{|s_\mathcal{O}(h_i) \cap [t_\mathcal{O}(h_i)] |},
    $$
    where in the last equality we used that any $v \in [d+1]$ occurs exactly once as target of an hyperedge. Scaling globally by $(-1)^{\frac{d(d+1)}{2}}$ now gives the desired result.
\end{proof}


\begin{example}
    Let $d=2$ and $n=6$. Assume $q_{123} \neq 0$, then $L(X)$ can be written as row space of the matrix 
    $$
    V = \left(\begin{smallmatrix}
        0 & 0 & 0 & 0 & 0 & q_{123} & q_{124} & q_{125} & q_{126} & q_{134} & q_{135} & q_{136} & q_{145} & q_{146} & q_{156} \\
        0 & -q_{123} & -q_{124} & -q_{125} & -q_{126} & 0 & 0 & 0 & 0 & q_{234} & q_{235} & q_{236} & q_{245} & q_{246} & q_{256} \\
        q_{123} & 0 & -q_{134} & -q_{135} & -q_{136} & 0 & -q_{234} & -q_{235} & -q_{236} & 0 & 0 & 0 & q_{345} & q_{346} & q_{356}
    \end{smallmatrix}\right), 
    $$
    where we ordered the columns lexicographically.
    Let us calculate the Plücker coordinate of $L(X)$ indexed by $J_1=12,J_2=16$ and $J_3=56$. Taking complements yields the hyperedges $h_1=3456, h_2=2345$ and $h_3=1234$. The only way to orient the corresponding hypergraph such that it has roots $I_0 = 456$ is shown in Figure \ref{figure: hypergraph giving monomial}. Here we draw the target of each hyperedge inside an accordingly colored circle. The sources of the hyperedges are given by $456, 345$ and $234$ respectively. 
    Our formula gives the monomial $p_{12,16,56} = \varepsilon \cdot q_{123}q_{126}q_{156}$. To calculate $\varepsilon$ we first compute the target map to be the permutation $(1 ~ 3)$. Counting increasing edges with multiplicity gives $0+0+0 = 0$. Hence $\varepsilon = (-1)^0 \sign \hspace{.3mm} (1 ~ 3) = -1$. Computing $p_{12,16,56}$ directly from $V$ gives the same result up to scaling by $(-1)^{\frac{d(d+1)}{2}} = -1$ as expected.
\end{example}

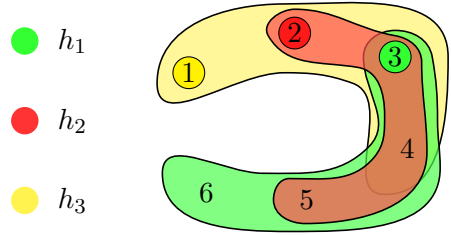
\begin{figure}[ht]
\begin{center}
    \begin{tikzpicture}[scale = 0.7]
        \node (v1) at (0, 2) {};
        \node (v2) at (1.9, 3) {};
        \node (v3) at (4, 2.5) {};
        \node (v6) at (0.1, 0.3) {};
        \node (v5) at (2, -0.2) {};
        \node (v4) at (3.9, 0.4) {};

        \begin{scope}[fill opacity = 0.7, line width=0.6pt]
            \filldraw [fill = yellow!70] ($(v1) + (-0.5, 0)$)
            to [out = 90, in = 180] ($(v2) + (0, 0.5)$)
            to [out = 0,in = 90] ($(v3) + (1, 0)$)
            to [out = 270,in = 0] ($(v4) + (-0.1, -0.5)$)
            to [out = 180,in = -20] ($(v3) + (-0.8, -0.5)$)
            to [out = 160, in = 0] ($(v2) + (0, -0.8)$)
            to [out = 180, in = 270] ($(v1) + (-0.5, 0)$);
        
            \filldraw [fill = green!70] ($(v5) + (-0.5, 0.5)$)
            to [out = 0, in = 225] ($(v3) + (-0.5, -1.7)$)
            to [out = 45, in = 270] ($(v3) + (-0.6, 0)$)
            to [out = 90, in = 180] ($(v3) + (0, 0.5)$)
            to [out = 0, in = 100] ($(v3) + (0.7, 0)$)
            to [out = 280, in = 90] ($(v4) + (0.9, 0.9)$)
            to [out = 270, in = 90] ($(v4) + (0.95, 0)$)
            to [out = 270, in = 0] ($(v5) + (0, -0.6)$)
            to [out = 180, in = 270] ($(v6) + (-0.5, 0)$)
            to [out = 90, in = 180] ($(v5) + (-0.5, 0.5)$);
            
            \filldraw [fill = red!70] ($(v2) + (-0.3, -0.1)$)
            to [out = 120, in = 180] ($(v2) + (0.2, 0.4)$)
            to [out = 0, in = 170] ($(v3) + (0, 0.4)$)
            to [out = -10, in = 90] ($(v3) + (0.5, -0.1)$)
            to [out = 270, in = 90] ($(v4) + (0.7, 0.3)$)
            to [out = 270, in = 0] ($(v5) + (0.3, -0.45)$)
            to [out = 180, in = 270] ($(v5) + (-0.3, 0)$)
            to [out = 90, in = 180] ($(v5) + (0, 0.4)$)
            to [out = 0, in = 230] ($(v4) + (-0.3, 0.2)$)
            to [out = 50, in = 270] ($(v3) + (-0.2, -1.1)$)
            to [out = 90, in = -10] ($(v3) + (-0.7, -0.2)$)
            to [out = 170, in = 300] ($(v2) + (-0.3, -0.1)$)
            ;
        \end{scope}


        \filldraw [fill = yellow] ($(v1) + (0.1,0.2)$) circle (0.3) node {$1$};
        \filldraw [fill = red!90] ($(v2) + (0.2, -0.05)$) circle (0.3) node {$2$};
        \filldraw [fill = green!80] ($(v3) + (0, 0)$) circle (0.3) node {$3$};
        \node at (v4) [above right] {$4$};
        \node at (v5) [right] {$5$};
        \node at (v6) [below right] {$6$};

        \begin{scope}[every node/.style = {fill, shape = circle, node distance = 30pt}]
            \node (h1) [color = green!80, label = right:$h_1$] at (-3, 2.8) {};
            \node (h2) [below of = h1, color = red!80, label = right:$h_2$] {};
            \node (h3) [below of = h2, color = yellow!80, label = right:$h_3$] {};
        \end{scope}
        
    \end{tikzpicture}
    \caption{A directed hypergraph on $6$ vertices with sources $4,5,6$.}\label{figure: hypergraph giving monomial}
\end{center}
\end{figure}
So far we have just reinterpreted determinants in terms of hypergraphs. To get a more powerful formula we need to allow choosing the roots $I_0$ flexibly. This way one can get expressions with fewer terms as when fixing the same roots for the calculation of all Plücker coordinates. \\ 

To state our next result, we use the shorthand $(-1)^A$ for $\prod_{a \in A}(-1)^a$ when $A \subseteq \ZZ$ is finite.

\begin{lemma}\label{lemma: Plücker coords of L(X) changing I_0}
    Suppose that $I_0$ and $I_0'$ are both subsets of $[n]$ of size $n-d-1$ such that $q_{[n] \setminus I_0}$ and $q_{[n] \setminus I_0'}$ are non zero. We write
    $p_{J_1 \ldots J_{d+1}}$ and $p_{J_1 \ldots J_{d+1}}'$ respectively to 
    distinguish the expressions given by Equation $\eqref{equation: Plücker coords of L^S(X) with fixed I_0}$ when fixing the roots as $I_0$ or $I_0'$. Then we have
    $$ 
    (-1)^{I_0'}q_{[n] \setminus I_0'} \cdot p_{J_1 \ldots J_{d+1}} = 
    (-1)^{I_0}q_{[n] \setminus I_0} \cdot p_{J_1 \ldots J_{d+1}}'.
    $$
\end{lemma}

\begin{proof}
    We already know $p_{J_1 \ldots J_{d+1}}$ and $p_{J_1 \ldots J_{d+1}}'$ may differ only by a scalar independent of $J_1, \ldots, J_{d+1}$. 
    To find this scalar it suffices to compare $p_{J_1 \ldots J_{d+1}}$ with $p_{J_1 \ldots J_{d+1}}'$ for indices $J_1, \ldots, J_{d+1}$ where both expressions are nonzero. 
    Here we may assume $X$ to lie in the open dense set where all $q_I$ are nonzero.
    First suppose $I_0$ and $I_0'$ are neighboring i.e. $|I_0 \cap I_0'| = n-d-2$. Let $A$ be their intersection and let $a_1 < \ldots < a_{d+2}$ form the complement of $A$. Then there are $i \neq j \leq d+2$ such that $I_0 = Aa_i$ and $I_0' = Aa_j$. Without loss of generality let us assume $i < j$. For $1 \leq l \leq d+1$ we define the hyperedge
    $h_{d+2-l} = Aa_la_{l+1}$ and take $J_i = [n] \setminus h_i$.
    Note that the $J_i$ are in lexicographical order.
    There is only one way to direct the $h_i$ such they have distinct targets and the resulting orientation $\mathcal{O}(I_0)$ has roots $I_0$. This orientation is determined by the target map
    $$
    \pi_{\mathcal{O}(I_0)} = d+2-l \mapsto \begin{cases}
        a_l & 1 \leq l < i \\
        a_{l+1} & i \leq l \leq d+1.
    \end{cases}
    $$
    Likewise we have a unique orientation for $I_0'$ defined in the same way with $i$ replaced by $j$. 
    
    Up to signs the orientations $\mathcal{O}(I_0)$ and $\mathcal{O}(I_0')$ give the nonzero monomials 
    $$
    q_{[n] \setminus I_0} \cdot \prod_{l =2}^{d+1}q_{[n]\setminus Aa_l} 
    ~~~\text{ and }~~~q_{[n] 
    \setminus I_0'} \cdot \prod_{l =2}^{d+1}q_{[n]\setminus Aa_l},
    $$ 
    respectively.
    It remains to compare the signs.
    Both $\pi_{\mathcal{O}(I_0)}$ and $\pi_{\mathcal{O}(I_0')}$ reverse orders and thus have the same sign.
    As for the numbers of increasing edges note that
    the hyperedge $h_{d+2-l} = Aa_la_{l+1}$ is directed differently in $\mathcal{O}(I_0)$ and $\mathcal{O}(I_0')$ only for $i \leq l < j$. 
    In $\mathcal{O}(I_0)$ it is directed towards $a_{l+1}$, so we count it with multiplicity $|A \cap [a_{l+1}]|+1$. 
    In $\mathcal{O}(I_0')$ it is directed towards $a_{l}$ and hence is counted with multiplicity $|A \cap [a_l]|$. Thus 
    \begin{align*}
        \varepsilon_{\mathcal{O}(I_0)} \varepsilon_{\mathcal{O}(I_0')} 
        =&~ \prod_{l=i}^{j-1} (-1)^{|A \cap [a_{l+1}]|+1} \cdot (-1)^{|A \cap [a_l]|} \\
        =&~ (-1)^{|A \cap [a_i]| + |A \cap [a_j]| + j-i} \\
        =&~ (-1)^{| (A \cup \{a_{i+1}, \ldots, a_j\}) \cap [a_i, a_{j}] |} \\
        =&~ (-1)^{a_j - a_i} \\
        =&~ (-1)^{I_0}(-1)^{I_0'}.
    \end{align*}
    Now the claim follows easily.
    
    If $I_0$ and $I_0'$ are not neighboring, we exploit the fact that $B = \{I \in \binom{[n]}{n-d-1} : q_{[n] \setminus I} \neq 0 \}$ is the set of bases of the matroid dual to the linear matroid of $X$. This guarantees us a sequence $I_0 = K_1, \ldots, K_r = I_0'$ of neighboring sets in $B$ from with the general case is deduced. 
\end{proof}

\begin{theorem}\label{theorem: Plücker coords of L^S(X)}
    Let $X \leq K^n$ be of dimension $d+1$ and $S \subseteq [n]$ as in Theorem \ref{theorem: L^S(X) = L(X^S)}.
    Let $J_1, \ldots, J_{d+1} \in \binom{[n]}{d}$ be distinct and sorted in lexicographic order. Let $H$ be the hypergraph on $[n]$ with hyperedges $h_i = [n] \setminus J_i$. Then the Plücker coordinate $p_{J_1 \ldots J_{d+1}}$ of $L^S(X)$ equals
    %
    \begin{equation}\label{equation: Plücker coords for L^S(X)}
        p_{J_1 \ldots J_{d+1}} = 
        \begin{cases}
            (-1)^{I_0} \sum_{\mathcal{O}} \varepsilon_\mathcal{O}  \prod_{i \neq i_0} q_{[n] \setminus s_{\mathcal{O}}(h_i)} & \text{if $S \subseteq l(H)$} \\
            0 & \text{otherwise}.
        \end{cases}
    \end{equation}
    Here, $h_{i_0} \in H$ and $I_0 \subset h_{i_0}$ of size $n-d-1$ can be chosen freely and the sum runs over all orientations $\mathcal{O}$ of $H$ with roots given by $I_0$. The signs $\varepsilon_\mathcal{O}$ are as in Lemma \ref{lemma: Plücker coords of L(X) with fixed I_0}. 
\end{theorem}

\begin{proof}
    We again only need to consider $X$ for which all $q_I$ are nonzero. The case $S = \emptyset$ is established by applying the previous two results.
    Indeed, adjust \eqref{equation: Plücker coords of L^S(X) with fixed I_0} by multiplying with $(-1)^{I_0}q_{[n] \setminus I_0}^{-1}$, then Lemma \ref{lemma: Plücker coords of L(X) changing I_0} tells us that this yields expressions for the Plücker coordinates of $L(X)$ independent of the choice of $I_0$. Restricting this choice so that $I_0 \subset h_{i_0}$ for some $i_0$ forces $s_{\mathcal{O}}(h_{i_0}) = I_0$ in any orientation of interest. Thus we obtain equation $\eqref{equation: Plücker coords for L^S(X)}$ when $S = \emptyset$.
    
    For nonempty $S$ we invoke Theorem \ref{theorem: L^S(X) = L(X^S)} and use the established case for $L(X^S) = L^S(X)$. 
    Recall that the Plücker vector of $X^S$ is denoted by $q^S$ and arises from $q$ by replacing $q_I$ with zero whenever $S \not \subseteq I$. 
    If $S$ contains a vertex $v$ of degree greater than 1 covered by say $h$ and $h'$, we can choose $h_{i_0} = h$ and $I_0 \subset h$ containing $v$. This forces $v \in s_{\mathcal{O}}(h')$ for any orientation $\mathcal{O}$ with roots $I_0$. Hence all terms in $p_{J_1 \ldots J_{d+1}}$ will vanish.
\end{proof}
Note that even if $S \subseteq l(H)$, equation \eqref{equation: Plücker coords for L^S(X)} may still give $0$. For generic $X$, Corollary \ref{corollary: characterization of D_{n-d}(U_{n,n})} below tells us this happens only when we can choose $I_0 \subset h_{i_0}$ such that $H$ admits no orientation with roots $I_0$.
As a rule of thumb $I_0$ should always be chosen such that it occurs in as many $h_i$ as possible to keep the number of terms low.

\newsavebox{\hypergraphZero}
\sbox{\hypergraphZero}{
\begin{tikzpicture}[scale = 0.7]
    \node (v1) at (0, 2) {};
    \node (v2) at (1.9, 3) {};
    \node (v3) at (4, 2.5) {};
    \node (v6) at (0.1, 0.3) {};
    \node (v5) at (2, -0.2) {};
    \node (v4) at (3.9, 0.4) {};

    \begin{scope}[fill opacity = 0.7, line width=0.6pt]
        \filldraw [fill = yellow!70] ($(v1) + (-0.35, 0.45)$)
        to [out = 75, in = 180] ($(v2) + (0, 0.6)$)
        to [out = 0,in = 95] ($(v3) + (0.9, -0.4)$)
        to [out = 275,in = 55] ($(v4) + (0.8, 0)$)
        to [out = 235,in = -50] ($(v4) + (-0.15, 0)$)
        to [out = 130,in = -20] ($(v3) + (-0.8, -0.5)$)
        to [out = 160, in = 0] ($(v2) + (0, -0.8)$)
        to [out = 180, in = -10] ($(v1) + (0, -0.2)$)
        to [out = 170, in = 255] ($(v1) + (-0.35, 0.45)$)
        ;
    
        \filldraw [fill = green!70] ($(v1) + (-0.2, 0.5)$)
        to [out = 60, in = 170] ($(v2) + (0.55, 0.4)$)
        to [out = 350, in = 140] ($(v3) + (-0.8, -0.8)$)
        to [out = -40, in = 90] ($(v4) + (0.9, 0.35)$)
        to [out = -90, in = -35] ($(v5) + (-0.2, -0.4)$)
        to [out = 145, in = 290] ($(v5) + (-0.3, 1)$)
        to [out = 110, in = 240] ($(v1) + (-0.2, 0.5)$)
        ;
        
        \filldraw [fill = red!70] ($(v2) + (-0.3, -0.1)$)
        to [out = 110, in = 180] ($(v2) + (0.2, 0.4)$)
        to [out = 0, in = 130] ($(v3) + (0.4, 0.3)$)
        to [out = -50, in = 90] ($(v3) + (0.73, -1)$)
        to [out = 270, in = 75] ($(v4) + (0.75, 0)$)
        to [out = 255, in = 0] ($(v5) + (0.3, -0.45)$)
        to [out = 180, in = 270] ($(v5) + (-0.3, 0)$)
        to [out = 90, in = 180] ($(v5) + (0, 0.4)$)
        to [out = 0, in = 230] ($(v4) + (-0.3, 0.2)$)
        to [out = 50, in = 270] ($(v3) + (-0.2, -1.1)$)
        to [out = 90, in = -10] ($(v3) + (-0.7, -0.2)$)
        to [out = 170, in = 290] ($(v2) + (-0.3, -0.1)$)
        ;
    \end{scope}

    \node at ($(v1) + (0.1,0.2)$) {$1$};
    \node at ($(v2) + (0.2, -0.05)$) {$2$};
    \node at ($(v3) + (0, 0)$) {$3$};
    \node at (v4) [above right] {$4$};
    \node at (v5) [right] {$5$};
    \node at (v6) [below right] {$6$};
    
\end{tikzpicture}
}

\newsavebox{\hypergraphMonomOne}
\sbox{\hypergraphMonomOne}{
\begin{tikzpicture}[scale = 0.7]
    \node (v1) at (0, 2) {};
    \node (v2) at (1.9, 3) {};
    \node (v3) at (4, 2.5) {};
    \node (v6) at (0.1, 0.3) {};
    \node (v5) at (2, -0.2) {};
    \node (v4) at (3.9, 0.4) {};

    \begin{scope}[fill opacity = 0.7, line width=0.6pt]
        \filldraw [fill = yellow!70] ($(v1) + (-0.35, 0.45)$)
        to [out = 75, in = 180] ($(v2) + (0, 0.6)$)
        to [out = 0,in = 95] ($(v3) + (0.9, -0.4)$)
        to [out = 275,in = 55] ($(v4) + (0.8, 0)$)
        to [out = 235,in = -50] ($(v4) + (-0.15, 0)$)
        to [out = 130,in = -20] ($(v3) + (-0.8, -0.5)$)
        to [out = 160, in = 0] ($(v2) + (0, -0.8)$)
        to [out = 180, in = -10] ($(v1) + (0, -0.2)$)
        to [out = 170, in = 255] ($(v1) + (-0.35, 0.45)$)
        ;
    
        \filldraw [fill = green!70] ($(v5) + (-0.5, 0.58)$)
        to [out = 0, in = 225] ($(v3) + (-0.5, -1.7)$)
        to [out = 45, in = 270] ($(v3) + (-0.6, 0)$)
        to [out = 90, in = 180] ($(v3) + (0, 0.5)$)
        to [out = 0, in = 100] ($(v3) + (0.7, 0)$)
        to [out = 280, in = 90] ($(v4) + (1, 0)$)
        to [out = 270, in = 0] ($(v5) + (0, -0.6)$)
        to [out = 180, in = 270] ($(v6) + (-0.2, -0.3)$)
        to [out = 90, in = 180] ($(v5) + (-0.5, 0.58)$);
        
        \filldraw [fill = red!70] ($(v1) + (0.1, 0.65)$)
        to [out = 0, in = 92] ($(v1) + (0.5, 0)$)
        to [out = 272, in = 90] ($(v6) + (0.3, 0.65)$)
        to [out = -90, in = 180] ($(v5) + (0.4, 0.4)$)
        to [out = 0, in = 250] ($(v4) + (-0.1, 0.4)$)
        to [out = 70, in = 180] ($(v4) + (0.35, 0.8)$)
        to [out = 0, in = 90] ($(v4) + (0.75, 0.2)$)
        to [out = 270, in = 3] ($(v5) + (0.4, -0.4)$)
        to [out = 183, in = -10] ($(v6) + (0.6, -0.8)$)
        to [out = 170, in = -50] ($(v6) + (-0.2, -0.4)$)
        to [out = 130, in = 270] ($(v1) + (-0.4, -0.1)$)
        to [out = 90, in = 180] ($(v1) + (0.1, 0.65)$)
        ;
    \end{scope}

    \node at ($(v1) + (0.1,0.2)$) {$1$};
    \node at ($(v2) + (0.2, -0.05)$) {$2$};
    \node at ($(v3) + (0, 0)$) {$3$};
    \node at (v4) [above right] {$4$};
    \node at (v5) [right] {$5$};
    \node at (v6) [below right] {$6$};
    
\end{tikzpicture}
}

\newsavebox{\hypergraphMonomTwo}
\sbox{\hypergraphMonomTwo}{
\begin{tikzpicture}[scale = 0.7]
    \node (v1) at (0, 2) {};
    \node (v2) at (1.9, 3) {};
    \node (v3) at (4, 2.5) {};
    \node (v6) at (0.1, 0.3) {};
    \node (v5) at (2, -0.2) {};
    \node (v4) at (3.9, 0.4) {};

    \begin{scope}[fill opacity = 0.7, line width=0.6pt]
        \filldraw [fill = yellow!70] ($(v1) + (-0.35, 0.45)$)
        to [out = 75, in = 180] ($(v2) + (0, 0.6)$)
        to [out = 0,in = 95] ($(v3) + (0.9, -0.4)$)
        to [out = 275,in = 55] ($(v4) + (0.8, 0)$)
        to [out = 235,in = -50] ($(v4) + (-0.15, 0)$)
        to [out = 130,in = -20] ($(v3) + (-0.8, -0.5)$)
        to [out = 160, in = 0] ($(v2) + (0, -0.8)$)
        to [out = 180, in = -10] ($(v1) + (0, -0.2)$)
        to [out = 170, in = 255] ($(v1) + (-0.35, 0.45)$)
        ;
    
        \filldraw [fill = green!70] ($(v5) + (-0.5, 0.5)$)
        to [out = 0, in = 225] ($(v3) + (-0.5, -1.7)$)
        to [out = 45, in = 270] ($(v3) + (-0.6, 0)$)
        to [out = 90, in = 180] ($(v3) + (0, 0.5)$)
        to [out = 0, in = 100] ($(v3) + (0.7, 0)$)
        to [out = 280, in = 90] ($(v4) + (0.9, 0.9)$)
        to [out = 270, in = 90] ($(v4) + (0.95, 0)$)
        to [out = 270, in = 0] ($(v5) + (0, -0.6)$)
        to [out = 180, in = 270] ($(v6) + (-0.2, -0.3)$)
        to [out = 90, in = 180] ($(v5) + (-0.5, 0.5)$);
        
        \filldraw [fill = red!70] ($(v2) + (-0.3, -0.1)$)
        to [out = 110, in = 180] ($(v2) + (0.2, 0.4)$)
        to [out = 0, in = 130] ($(v3) + (0.4, 0.3)$)
        to [out = -50, in = 90] ($(v3) + (0.73, -1)$)
        to [out = 270, in = 80] ($(v4) + (0.75, 0)$)
        to [out = 260, in = 0] ($(v5) + (0.3, -0.45)$)
        to [out = 180, in = 270] ($(v5) + (-0.3, 0)$)
        to [out = 90, in = 180] ($(v5) + (0, 0.4)$)
        to [out = 0, in = 230] ($(v4) + (-0.3, 0.2)$)
        to [out = 50, in = 270] ($(v3) + (-0.2, -1.1)$)
        to [out = 90, in = -10] ($(v3) + (-0.7, -0.2)$)
        to [out = 170, in = 290] ($(v2) + (-0.3, -0.1)$)
        ;
    \end{scope}

    \node at ($(v1) + (0.1,0.2)$) {$1$};
    \node at ($(v2) + (0.2, -0.05)$) {$2$};
    \node at ($(v3) + (0, 0)$) {$3$};
    \node at (v4) [above right] {$4$};
    \node at (v5) [right] {$5$};
    \node at (v6) [below right] {$6$};
    
\end{tikzpicture}
}

\newsavebox{\hypergraphMonomThree}
\sbox{\hypergraphMonomThree}{
\begin{tikzpicture}[scale = 0.7]
    \node (v1) at (0, 2) {};
    \node (v2) at (1.9, 3) {};
    \node (v3) at (4, 2.5) {};
    \node (v6) at (0.1, 0.3) {};
    \node (v5) at (2, -0.2) {};
    \node (v4) at (3.9, 0.4) {};

    \begin{scope}[fill opacity = 0.7, line width=0.6pt]
        \filldraw [fill = yellow!70] ($(v2) + (0.2, 0.5)$)
        to [out = 0, in = 120] ($(v2) + (1.1, -1.3)$)
        to [out = -60, in = 170] ($(v4) + (0.4, 0.9)$)
        to [out = -10, in = 90] ($(v4) + (1, 0)$)
        to [out = 270, in = 0] ($(v5) + (0, -0.75)$)
        to [out = 180, in = 275] ($(v6) + (-0.63, -0.1)$)
        to [out = 95, in = 240] ($(v2) + (-0.7, -1.3)$)
        to [out = 60, in = 180] ($(v2) + (0.2, 0.5)$)
        ;
    
        \filldraw [fill = green!70] ($(v5) + (-0.5, 0.58)$)
        to [out = 0, in = 225] ($(v3) + (-0.5, -1.7)$)
        to [out = 45, in = 270] ($(v3) + (-0.6, 0)$)
        to [out = 90, in = 180] ($(v3) + (0, 0.5)$)
        to [out = 0, in = 100] ($(v3) + (0.7, 0)$)
        to [out = 280, in = 90] ($(v4) + (0.9, 0)$)
        to [out = 270, in = 0] ($(v5) + (0, -0.6)$)
        to [out = 180, in = 270] ($(v6) + (-0.2, -0.3)$)
        to [out = 90, in = 180] ($(v5) + (-0.5, 0.58)$);
        
        \filldraw [fill = red!70] ($(v1) + (0.1, 0.65)$)
        to [out = 0, in = 92] ($(v1) + (0.5, 0)$)
        to [out = 272, in = 90] ($(v6) + (0.3, 0.65)$)
        to [out = -90, in = 180] ($(v5) + (0.4, 0.4)$)
        to [out = 0, in = 250] ($(v4) + (-0.1, 0.4)$)
        to [out = 70, in = 180] ($(v4) + (0.35, 0.8)$)
        to [out = 0, in = 90] ($(v4) + (0.75, 0.2)$)
        to [out = 270, in = 3] ($(v5) + (0.4, -0.4)$)
        to [out = 183, in = -10] ($(v6) + (0.6, -0.8)$)
        to [out = 170, in = -50] ($(v6) + (-0.2, -0.4)$)
        to [out = 130, in = 270] ($(v1) + (-0.4, -0.1)$)
        to [out = 90, in = 180] ($(v1) + (0.1, 0.65)$)
        ;
    \end{scope}

    \node at ($(v1) + (0.1,0.2)$) {$1$};
    \node at ($(v2) + (0.2, -0.05)$) {$2$};
    \node at ($(v3) + (0, 0)$) {$3$};
    \node at (v4) [above right] {$4$};
    \node at (v5) [right] {$5$};
    \node at (v6) [below right] {$6$};
    
\end{tikzpicture}
}

\newsavebox{\hypergraphBinom}
\sbox{\hypergraphBinom}{
\begin{tikzpicture}[scale = 0.7]
    \node (v1) at (0, 2) {};
    \node (v2) at (1.9, 3) {};
    \node (v3) at (4, 2.5) {};
    \node (v6) at (0.1, 0.3) {};
    \node (v5) at (2, -0.2) {};
    \node (v4) at (3.9, 0.4) {};

    \begin{scope}[fill opacity = 0.7, line width=0.6pt]
        \filldraw [fill = yellow!70] ($(v1) + (-0.35, 0.45)$)
        to [out = 75, in = 180] ($(v2) + (0, 0.6)$)
        to [out = 0,in = 95] ($(v3) + (0.9, -0.4)$)
        to [out = 275,in = 55] ($(v4) + (0.8, 0)$)
        to [out = 235,in = -50] ($(v4) + (-0.15, 0)$)
        to [out = 130,in = -20] ($(v3) + (-0.8, -0.5)$)
        to [out = 160, in = 3] ($(v2) + (0, -0.85)$)
        to [out = 183, in = -10] ($(v1) + (0, -0.2)$)
        to [out = 170, in = 255] ($(v1) + (-0.35, 0.45)$)
        ;
    
        \filldraw [fill = green!70] ($(v5) + (-0.5, 0.5)$)
        to [out = 0, in = 225] ($(v3) + (-0.5, -1.7)$)
        to [out = 45, in = 270] ($(v3) + (-0.6, 0)$)
        to [out = 90, in = 180] ($(v3) + (0, 0.5)$)
        to [out = 0, in = 100] ($(v3) + (0.7, 0)$)
        to [out = 280, in = 90] ($(v4) + (0.9, 0.9)$)
        to [out = 270, in = 57] ($(v4) + (0.73, -0.2)$)
        to [out = 237, in = 0] ($(v5) + (0, -0.6)$)
        to [out = 180, in = 270] ($(v6) + (-0.2, -0.3)$)
        to [out = 90, in = 180] ($(v5) + (-0.5, 0.5)$);
        
        \filldraw [fill = red!70] ($(v2) + (0.1, 0.35)$)
        to [out = 180, in = 45] ($(v1) + (-0.1, 0.6)$)
        to [out = 225, in = 93] ($(v1) + (-0.53, -0.95)$)
        to [out = 273, in = 150] ($(v6) + (0.1, -0.77)$)
        to [out = -30, in = 180] ($(v5) + (0.2, -0.5)$)
        to [out = 0, in = 270] ($(v5) + (0.8, 0)$)
        to [out = 90, in = 0] ($(v5) + (0.25, 0.43)$)
        to [out = 180, in = -70] ($(v6) + (0.5, 0.45)$)
        to [out = 110, in = -118] ($(v1) + (0.6, -0.25)$)
        to [out = 58, in = -100] ($(v2) + (0.7, -0.2)$)
        to [out = 80, in = 0] ($(v2) + (0.1, 0.35)$)
        ;
    \end{scope}

    \node at ($(v1) + (0.1,0.2)$) {$1$};
    \node at ($(v2) + (0.2, -0.05)$) {$2$};
    \node at ($(v3) + (0, 0)$) {$3$};
    \node at (v4) [above right] {$4$};
    \node at (v5) [right] {$5$};
    \node at (v6) [below right] {$6$};
    
\end{tikzpicture}
}

\begin{example}\label{example: plücker coords of L^S(X) for n=6,d=2}
    We again consider the case $d=2$ and $n=6$. In \cite[Example 11]{JMRS} it is shown that $60$ of the $455$ Plücker coordinates of $L(X)$ always vanish due to Plücker relations. Of the remaining $395$ coordinates $15$ are binomials in the $q_I$ when reduced modulo the Plücker relations while $380$ even reduce to monomials. We can use Proposition \ref{theorem: Plücker coords of L^S(X)} to understand this in terms of hypergraphs. The $60$ vanishing Plücker coordinates correspond to hypergraphs which arise from the upper left hypergraph in figure \ref{figure: hypergraphs giving plücker coords} by permuting vertices. Indeed, the edges of any such hypergraph do not cover enough vertices to accommodate an orientation with $3$ roots and distinct targets.
    The upper right hypergraph gives rise to $15$ distinct hypergraphs when permuting vertices. Each of them has two distinct orientations when we choose the roots inside a hyperedge, thus explaining the $15$ Plücker coordinates reducing to binomials. Permuting vertices of course affects the signs in the binomials. Nevertheless, they are of the form 
    $$ \pm q_{bcd}q_{aef} \pm q_{acd}q_{bef} ,$$
    with $\{a,b,c,d,e,f\} = [6]$.
    Lastly, the three hypergraphs in the bottom row of Figure \ref{figure: hypergraphs giving plücker coords} each have a unique orientation with distinct targets and roots $456$, hence they give monomial Plücker coordinates. When permuting vertices, the first two of these hypergraphs have orbits of size 180 each, while the third one only gives rise to $20$ different hypergraphs. 
    These hypergraphs yield monomials of the form 
    $$
    \pm q_{abc}q_{bef}, ~~ \pm q_{abf}q_{aef} ~~\text{and}~~ \pm q_{abc}^2.
    $$
    Now consider $L^S(X)$ for nonempty $S \subset [6]$. Immediately, all binomials must vanish, as they come from hypergraphs without leaves. Increasing the size of $S$ step by step such that $L^S(X)$ remains a linear space, the monomials coming from the hypergraphs in the bottom of figure \ref{figure: hypergraphs giving plücker coords} will start vanishing from left to right where for $|S| = 3$ only a single nonzero coordinate is left. Recall that for $|S| > 3$, $L^S(X)$ will in general not be linear anymore due to Lemma \ref{lemma: need pl.rels if |S| >= d+2}.
\end{example}

\begin{figure}[ht]
\begin{center}
    \begin{tikzpicture}
    
    \node at (0,-3.8)   (HM1) {\usebox{\hypergraphMonomOne}};
    \node at (4.5,-3.8) (HM2) {\usebox{\hypergraphMonomTwo}};
    \node at (9,-3.8)   (HM3) {\usebox{\hypergraphMonomThree}};
    
    \node at (2,0)  (H0) {\usebox{\hypergraphZero}};
    \node at (7,0)	(HB) {\usebox{\hypergraphBinom}};
        
    \end{tikzpicture}
    \caption{The hypergraphs corresponding to Plücker coordinates of $L(X)$ up to symmetry, when $d=2, n=6$.}\label{figure: hypergraphs giving plücker coords}
    
\end{center}
\end{figure}
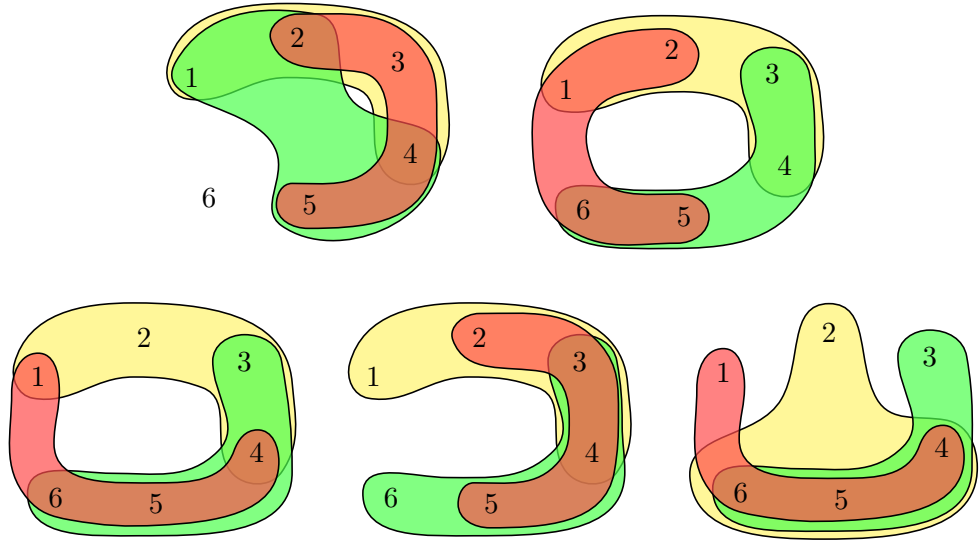

As corollary to our description of Plücker coordinates we see that bases of $\text{D}_{n-d}(U_{n,n})$ correspond to hypergraphs $H$ which are orientable enough:

\begin{corollary}\label{corollary: characterization of D_{n-d}(U_{n,n})}
    Let $h_1, \ldots, h_{d+1} \in \text{D}_{n-d}(U_{n,n})$ and denote the hypergraph with edges given by the $h_i$ as $H$. Then the following are equivalent.
    \begin{itemize}
        \item[(i)] The $h_i$ form a basis of $\text{D}_{n-d}(U_{n,n})$.
        \item[(ii)] For any $I_0 \subseteq [n]$ of size $n-d-1$ there is an orientation $\mathcal{O}$ of $H$ with roots $r(\mathcal{O}) = I_0$.
        \item[(iii)] For any $i=1, \ldots, d+1$, we can choose $I_0 \subset h_i$ of size $n-d-1$ such that there is an orientation $\mathcal{O}$ of $H$ with roots $r(\mathcal{O}) = I_0$.
        \item[(iv)] For any $i = 1, \ldots, d+1$, we find $d$ distinct elements $t_{i,j} \in h_j \setminus h_i$, where $j \neq i$.
    \end{itemize}
\end{corollary}

\begin{proof}
    If the $h_i$ form a basis of $\text{D}_{n-d}(U_{n,n})$, then for any generic $X \leq K^n$ of dimension $d+1$, the Plücker coordinate $p_{J_1 \ldots J_{d+1}}$ of $L(X)$ with $J_i = [n] \setminus h_i$ is nonzero. By Lemma \ref{lemma: Plücker coords of L(X) with fixed I_0}, $H$ must admit orientations for any choice of $n-d-1$ roots. The implications $(ii) \Rightarrow (iii) \Rightarrow (iv)$ are clear. To see $(iv) \Rightarrow (i)$ suppose we are given $h_{i_1}, \ldots, h_{i_s}$. Then $\bigcup_{l = 1}^s h_{i_l}$ contains $h_{i_1}$ as well as 
    $t_{i_1,i_2}, \ldots, t_{i_1,i_s}$. Thus $|\bigcup_{l = 1}^s h_{i_l}| \geq n-d+s-1$.
\end{proof}

A version of this corollary for independent sets is easily obtained by weakening all equalities $r(\mathcal{O}) = I_0$ to inclusions $r(\mathcal{O}) \supseteq I_0$.

\begin{corollary}
    Let $X \leq K^n$ be generic of dimension $d+1$. The bases of $\mathcal{M}(L^S(X))$ correspond to those bases of $\text{D}_{n-d}(U_{n,n})$ yielding hypergraphs with leaves containing $S$. 
\end{corollary}

\section{Tropicalizations}\label{section: 3 - Tropicalizations}

\subsection{Preliminaries on tropical geometry}\label{section: 3 - Preliminaries on tropical geometry}

There are several approaches to define tropicalizations of (multi-projective) varieties. We review the ones relevant to us, following essentially \cite{MS15}. 

We will assume our ground field $K$ to be algebraically closed together with a non trivial non-Archimedean valuation $\val: K \to \RR \cup \infty$. It is instructive to think of $K = \puiseux$ being the field of Puiseux series in one indeterminate $t$, together with the lowest order term valuation.

In the coming definitions we need the \emph{tropical semiring}, which is the set $\barRR = \RR \cup \infty$ equipped with the operations $\oplus = \min$ and $\odot = +$. 
The ordering inherited form the reals extends in the obvious way and gives rise to the order topology on $\barRR$. One may think of $\barRR^n$ as a tropical analogue of affine space and of $\RR^n$ as the tropical equivalent of a torus. 
We just write $\infty$ again for the additive neutral element of $\barRR^n$.
The tropical analogue of projective space is 
$$
\TT\PP^n = \left( \barRR^{n+1}\setminus \infty \right)/ \hspace{.5ex} \RR ,
$$ 
where $\RR$ acts by tropical scaling, i.e. $\lambda \odot w = (\lambda, \ldots, \lambda)+w$ for $w \in \barRR^{n+1}$ and $\lambda \in \RR$. We call $\TT\PP^n$ the \emph{tropical projective space} and products of such spaces \emph{tropical multi-projective space}. 

\begin{definition}\label{definition: trop f}
    Lef $f = \sum c_ux^u \in K[x_1^{\pm 1}, \ldots, x_n^{\pm 1}]$, where $u$ runs over a finite subset of $\ZZ^n$. The \emph{tropicalization} of $\trop(f)$ of $f$ is obtained by replacing each coefficient by its valuation, and each occurrence of $+$ or $\cdot$ in $f$ with $\oplus$ or $\odot$ respectively. 
    That is, $\trop(f) = \bigoplus \val(c_u) \odot x^{\odot u}$. 
\end{definition}

Note that $\trop(f)$ induces a function $\RR^n \to \RR$.
If $f \in K[x_1, \ldots, x_n]$, that is $f$ has only nonnegative exponents, we obtain a function $\barRR^n \to \barRR$. Here, we follow the conventions that $0 \cdot \infty = 0$ and $u \cdot \infty = \infty$ for $u > 0$. Moreover, $a + \infty = \infty$ for any $a \in \RR$.

\begin{definition}\label{definition: tropical hypersurfaces}
    Let $f \in K[x_1^{\pm 1}, \ldots, x_n^{\pm 1}]$. We define the \emph{tropical hypersurface} of $f$ as the set 
    $$ \trop(\mathcal{V}(f)) = \{ w \in \RR^n : \text{ the minimum in } \trop(f)(w) \text{ is attained at least twice} \} .$$
    %
    For $f \in K[x_1, \ldots, x_n]$ we define its \emph{affine tropical hypersurface}
    $\trop(\mathcal{V}(f))$ as
    $$ \{ w \in \barRR^n : \trop(f)(w) = \infty, \text{ or the minimum in } \trop(f)(w) \text{ is attained at least twice} \} .$$ 
    Finally, if $f$ is multi-homogeneous we can replace $\barRR^n$ in the last definition by tropical multi-projective space to define its \emph{multi-projective tropical hypersurface.}
\end{definition}


\begin{definition}\label{definition: tropicalization of varieties}
    Let $Z$ be a \emph{very affine} variety, i.e.\ a subvariety of the torus $(K^*)^n$ defined by an ideal $I \trianglelefteq K[x_1^{\pm 1}, \ldots, x_n^{\pm 1}]$. Then we define the \emph{tropicalization} of $Z$ as 
    $$ \trop(Z) = \bigcap_{f \in I} \trop(\mathcal{V}(f)).$$
    Likewise, we define the tropicalization of affine or multi-projective varieties using the corresponding version of tropical hypersurfaces, where in the latter case we 
    only run over the multi-homogeneous elements of the ideal.
\end{definition} 
It can be shown that each ideal $I \trianglelefteq K[x_1^{\pm 1}, \ldots, x_n^{\pm 1}]$ (respectively $I \trianglelefteq K[x_1, \ldots, x_n]$) possesses a \emph{tropical basis}, that is a finite generating set $\mathcal{T}$, such that $\trop(Z)$ is the intersection of the tropical hypersurfaces defined by $f \in \mathcal{T}$. Moreover, if $I$ is multi-homogeneous, then the elements of $\mathcal{T}$ can be chosen multi-homogeneous as well.

Since we assume $K$ to be algebraically closed and $\val$ as nontrivial, the various versions of the Fundamental Theorem (see \cite{MS15}) tell us that we can also compute tropicalizations essentially by taking pointwise valuations. Indeed, many authors \emph{define} $\trop(Z)$ as the closure of $\val(Z)$, where the closure is taken in the appropriate ambient space. 
Note that the varieties we are interested in all embed into toric varieties. 
These are naturally stratified by their torus orbits, which themselves can be identified with tori. 
For example, $\mathbb{P}^{n-1}$ is stratified by 
$$
\mathcal{O}_\sigma = 
\{ \mu \in \mathbb{P}^{n-1} : \mu_i = 0 \text{ if and only if } i \in \sigma \} \cong 
(K^*)^{n-|\sigma|}/(K^*) ,
$$ 
where $\sigma \subset [n]$, and $K^* = K \setminus 0$.
%
Such stratifications carry over to the tropical world.
\begin{theorem}\label{theorem: stratification of trop X}
    Assume that $Z$ is an affine, or multi-projective variety. Then stratifying the ambient variety of $Z$ into torus orbits $\mathcal{O}_\sigma$ as above leads to the decomposition
    $$ \trop(Z) = \bigcup_{\sigma} \trop(Z \cap \mathcal{O}_\sigma) \times \infty^\sigma ,$$
    where $Z \cap \mathcal{O}_\sigma$ is viewed as subvariety of the appropriate torus isomorphic to $\mathcal{O}_\sigma$.
\end{theorem}

For details on tropicalizing subvarieties of more general toric varieties, see \cite[chapter 6]{MS15}.

\begin{example}
    Consider $Z = \mathcal{V}(x_1+x_2+x_3) \subseteq \PP^2$ over $K = \puiseux$. The polynomial $x_1+x_2+x_3$ is already a tropical basis, so 
    $$
    \trop(Z) = 
    \{\mu \in \TT\PP^2 : \text{the minimum in } \mu_1 \oplus \mu_2 \oplus \mu_3 \text{ is attained at least twice}\}.
    $$ 
    This is the standard tropical line (see Figure \ref{fig:tropline}) consisting of $3$ rays spanned by the images of unit vectors in $\RR^3$, together with $3$ extra points, each of which corresponds to some $\mu_i$ being infinite. If we want to invoke Theorem \ref{theorem: stratification of trop X} instead, we first write $$Z = \{[a,b,-a-b] \in \PP^2 : (a,b) \in K^2 \setminus (0,0) \}.$$ The part of $Z$ where no coordinate vanishes tropicalizes to
    $$
    \trop(Z \cap (K^*)^3/K^*) = 
    \overline{\{ [\val(a), \val(b), \val(a+b)] : a,b \in K^* \text{ s.t. }a+b \neq 0\}} \subseteq \RR^3/\RR.
    $$ 
    This is readily verified to be the union of the $3$ rays from above. Setting any coordinate to zero forces the valuations of the remaining two coordinates to coincide. Thus we see the $3$ extra points.
\end{example}

 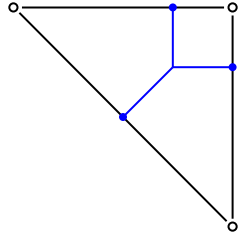
\begin{figure}[ht]
     \begin{center}

    \tikzset{every picture/.style={line width=0.75pt}} 
    \begin{tikzpicture}[x=0.75pt,y=0.75pt,yscale=-1,xscale=1]
    
    \draw [color=blue  ,draw opacity=1 ]   (310,140) -- (310,170) ;
    \draw [color=blue  ,draw opacity=1 ]   (310,170) -- (340,170) ;
    \draw [color=blue  ,draw opacity=1 ]   (285,195) -- (310,170) ;
    \draw (227.95,140) .. controls (227.95,138.87) and (228.87,137.95) .. (230,137.95) .. controls (231.13,137.95) and (232.05,138.87) .. (232.05,140) .. controls (232.05,141.13) and (231.13,142.05) .. (230,142.05) .. controls (228.87,142.05) and (227.95,141.13) .. (227.95,140) -- cycle ;
    \draw (337.95,250) .. controls (337.95,248.87) and (338.87,247.95) .. (340,247.95) .. controls (341.13,247.95) and (342.05,248.87) .. (342.05,250) .. controls (342.05,251.13) and (341.13,252.05) .. (340,252.05) .. controls (338.87,252.05) and (337.95,251.13) .. (337.95,250) -- cycle ;
    \draw (337.95,140) .. controls (337.95,138.87) and (338.87,137.95) .. (340,137.95) .. controls (341.13,137.95) and (342.05,138.87) .. (342.05,140) .. controls (342.05,141.13) and (341.13,142.05) .. (340,142.05) .. controls (338.87,142.05) and (337.95,141.13) .. (337.95,140) -- cycle ;
    \draw    (234.29,140.04) -- (335.71,140.04) ;
    \draw    (233,142.71) -- (337,247.29) ;
    \draw    (340,144.04) -- (340,246.04) ;
    \draw [color=blue  ,draw opacity=1 ][fill=blue  ,fill opacity=1 ] (308.5,140) .. controls (308.5,139.17) and (309.17,138.5) .. (310,138.5) .. controls (310.83,138.5) and (311.5,139.17) .. (311.5,140) .. controls (311.5,140.83) and (310.83,141.5) .. (310,141.5) .. controls (309.17,141.5) and (308.5,140.83) .. (308.5,140) -- cycle ;
    \draw [color=blue  ,draw opacity=1 ][fill=blue  ,fill opacity=1 ] (338.5,170) .. controls (338.5,169.17) and (339.17,168.5) .. (340,168.5) .. controls (340.83,168.5) and (341.5,169.17) .. (341.5,170) .. controls (341.5,170.83) and (340.83,171.5) .. (340,171.5) .. controls (339.17,171.5) and (338.5,170.83) .. (338.5,170) -- cycle ;
    \draw [color=blue  ,draw opacity=1 ][fill=blue  ,fill opacity=1 ] (283.5,195) .. controls (283.5,194.17) and (284.17,193.5) .. (285,193.5) .. controls (285.83,193.5) and (286.5,194.17) .. (286.5,195) .. controls (286.5,195.83) and (285.83,196.5) .. (285,196.5) .. controls (284.17,196.5) and (283.5,195.83) .. (283.5,195) -- cycle ;
    
    \end{tikzpicture}
    \end{center}
     
     \caption{A standard tropical line.}\label{fig:tropline}
 \end{figure}

Note that in the last example $Z$ was a linear space. We explore this case in more detail.
Assume $Z \leq K^n$ is the row space of a full rank matrix $A \in K^{r \times n}$. For $x \in K^n$ let $A_x$ be be matrix obtained from $A$ by appending $x$ as additional row. Now $x \in Z$ if and only if all maximal minors of $A_x$ vanish. In other words $Z$ is defined by the vanishing of the linear forms
\begin{equation}\label{equation: circuits of a linear space}
    c_I = \sum_{j \in I}(-1)^{|[j] \cap I|}q_{I\setminus j}x_j
\end{equation}
where $I \in \binom{[n]}{r+1}$, and the $q_J$ are the Plücker coordinates of $Z$. These linear forms are also called the \emph{circuits} of $Z$. 
They tropicalize to the \emph{valuated circuits}
\begin{equation}\label{equation: tropicalized circuits}
    \trop(c_I) = \bigoplus_{j \in I} \val(q_{I \setminus j}) \odot x_j
\end{equation}
and it turns out that the nonzero circuits $c_I$ give a tropical basis for the prime ideal they generate. Thus 
$$\trop(Z) =  \bigcap_I \trop(c_I).$$
We often identify a linear space $Z$ with its projectivization, in which case we also view its tropicalization in tropical projective space. In either case, its tropicalization is called a \emph{tropicalized linear space}.

This generalizes to the concept of tropical linear spaces.
We call $\mu \in \TT\PP^{\binom{n}{r}-1}$ a \emph{tropical Plücker vector} if it satisfies the \emph{tropical Plücker relations}: for $A \in \binom{[n]}{r-1}, B \in \binom{[n]}{r+1}$ the minimum $$ \min_{i \in B \setminus A}(\mu_{B\setminus i}+\mu_{Ai})$$ is attained at least twice (or equals $\infty$). Note that this is the same as saying $\mu$ is a valuated matroid  \cite{DW92, MR20}.
The \emph{tropical linear space} $\mathcal{L}(\mu)$ corresponding to $\mu$ is the intersection of the tropical hypersurfaces defined by the tropical polynomials \eqref{equation: tropicalized circuits} where $\val(q_{I \setminus j})$ needs to be replaced by $\mu_{I \setminus j}$. 
The tropical linear space $\mathcal{L}(\mu)$ is a pure $d$-dimensional polyhedral complex, and it determines the tropical Plücker vector $\mu$ up to tropical multiplication by a scalar. The collection of all $d$-dimensional tropical linear spaces is parametrized by the Dressian $\text{Dr}(d, n)$, which is the subset of $\TT\PP^{\binom{n}{d}-1}$ consisting of the tropical Plücker vectors \cite{HJJS09, HJS14, SW21, BS22}.
Clearly $\trop(Z) = \mathcal{L}(\val(q))$. Tropical Plücker vectors arising as valuations of classical Plücker vectors are called \emph{realizable}. They are precisely the elements of $\trop(\Gr(d,n)) \subset \text{Dr}(d, n)$. \\

The tropicalization of a vector space $Z \leq K^n$ is particularly nice, if $Z$ has only Plücker coordinates with ``trivial'' valuations, i.e.\ up to global tropical scaling we have $\val(q_I) \in \{0 ,\infty\}$ for all $I \in \binom{r}{n}$. In this case, we just say \emph{$Z$ has trivial valuations}. Indeed, in this case its tropicalization is determined solely by its associated linear matroid $\mathcal{M} = \mathcal{M}(Z)$. If $Z$ is the row space of a matrix $A$, this matroid can be described as the matroid of dependencies among the columns of $A$. 

Recall that a $k$-flat of $\mathcal{M}$ is a maximal subset of rank $k$. The set of all flats $\mathcal{L}(\mathcal{M})$ of $\mathcal{M}$ forms a (geometric) lattice called the \emph{lattice of flats}. For each chain of flats of the form
\begin{equation}\label{equation: chain of flats}
    F_{\bullet} :~~ \emptyset = F_0 \subset F_1 \subset \ldots \subset F_{s} \subset [n],
\end{equation}
in $\mathcal{L}(\mathcal{M})$ we define the convex cone 
\begin{equation}\label{equation: cone for chain of flats}
    C_{F_\bullet} = \cone(e_{F_1}, \ldots, e_{F_s}) + \RR e_{[n]} \subseteq \RR^n/\RR,
\end{equation}
where $e_I = \sum_{i \in I}e_i$ for $I \subseteq [n]$.
Note $F_\bullet^1 \subseteq F_\bullet^2$, if and only if $C_{F_\bullet^1} \subseteq C_{F_\bullet^2}$, so that 
$F_\bullet \mapsto C_{F_\bullet}$ is a bijection. One can show that the collection of cones as in \eqref{equation: cone for chain of flats} forms a fan called the \emph{Bergman fan} of $\mathcal{M}$ (see \cite{AK06} and \cite{MS15}).

\begin{theorem}\label{theorem: very affine trop X is bergman fan}
    Let $Z \leq K^n$ be a linear space with trivial valuations. Then the tropicalization of $Z \cap (K^*)^n/K^*$ equals the Bergman fan of $\mathcal{M}(Z)$, i.e.\ the fan given by the collection of cones as in equation \eqref{equation: cone for chain of flats}.
\end{theorem}

Note that in general $\emptyset$ is not a flat of $\mathcal{M}$. If it is not, we simply have no chains of the form \eqref{equation: chain of flats}. This fits well as in this case $\mathcal{M}$ has loops which means $Z$ is contained in a coordinate hyperplane. Then both $Z \cap (K^*)^n/K^*$ and its tropicalization are empty. \\

What about the tropicalization of $Z \cap \mathcal{O}_\sigma$ for other torus orbits of $\PP^{n-1}$? It turns out, we only need to extend our definition of cones. For $\sigma \subset [n]$, the Zarisky closure of $\mathcal{O}_\sigma$ is the vanishing locus of the coordinates indexed by $\sigma$. Now the matroid associated to
$Z_\sigma := Z \cap \overline{\mathcal{O}_\sigma}$ is given by the contraction $\mathcal{M}/\sigma$ of $\mathcal{M}$ by $\sigma$ \cite{Shaw13}. 
To see this, extend a basis $Z_\sigma$ to one of $Z$. Collect the basis elements as rows of a matrix $A$, then the rows spanning $Z_\sigma$ restricted to their coordinates indexed by $\sigma$ are zero. From here the claim easily follows. Moreover if $Z$ has trivial valuations then so does $Z_\sigma$, hence Theorem \ref{theorem: very affine trop X is bergman fan} applies to $Z \cap \mathcal{O}_\sigma = Z_\sigma \cap \mathcal{O}_\sigma$. Note that flats of $\mathcal{M}/\sigma$ correspond to flats of $\mathcal{M}$ containing $\sigma$. In particular $\mathcal{M}/\sigma$ is loopless iff $\sigma$ is a flat of $\mathcal{M}$. Using Theorem \ref{theorem: stratification of trop X} we have
$$
\trop(Z) = \bigcup_{\sigma \in \mathcal{L}(\mathcal{M})} \trop(Z \cap \mathcal{O}_\sigma) \times \infty^\sigma.
$$
See \cite{BEZ21} for an analogous statement in the context of valuated matroids.
The cones of the Bergman fans inside the various strata can all be written in a unified way as
\begin{equation}\label{equation: cones for chaines above flats}
    C_{F_\bullet} = \infty \cdot e_{F_0} + \cone(e_{F_1}, \ldots, e_{F_s}) + \RR e_{[n]} \subseteq (\RR^{[n]\setminus F_0}/\RR) \times \infty^{F_0},
\end{equation}
where $F_{\bullet}$ is chain of flats in $\mathcal{L}(\mathcal{M})$ for which we explicitly drop the assumption $F_0 = \emptyset$. 

\begin{corollary}\label{theorem: proj trop X is extended bergman fan}
    Let $Z \leq K^n$ be a linear space with trivial valuations. Then the tropicalization of (the projectivization of) $Z$ equals the \emph{extended Bergman fan} of $\mathcal{M}(Z)$, i.e. the collection of cones as in equation \eqref{equation: cones for chaines above flats}.
\end{corollary}

The name extended Bergman fan is justified by the observation that passing to closures of cones in $\TT\PP^{n-1}$ gives rise to a global (abstract) fan structure. Faces correspond to subchains, while intersections correspond to taking the greatest common subchain.

\subsection{Tropicalizing $L^S$-spaces}\label{section: 3 - trop L^S(X)}

We again fix a linear subspace $X \leq K^n$ of dimension $d+1$, with Plücker vector $q$, and apply our results to the $L^S$-spaces defined earlier. 

\subsubsection{Trivial valuation case} \mbox{ } \\
We first consider the case when $L^S(X)$ has trivial valuations. This is in general not implied by $X$ having trivial valuations. While we know the Plücker coordinates of $L^S(X)$ are polynomials in the Plücker coordinates $q_I$ of $X$, Example \ref{example: plücker coords of L^S(X) for n=6,d=2} shows that they need not be monomial. 
A sufficient criterion is to require that $X$ has \emph{constant coefficients} i.e. $X$ is defined over a subfield of $K$ on which the whole valuation is trivial. 
Indeed, if $X$ has constant coefficients then so does $L^S(X)$. 

\begin{corollary}\label{corollary: trop L^S(X) is extended Bergman fan}
    Suppose $X$ is of dimension $d+1$ and $S \subseteq [n]$ such that there is $q_C \neq 0$ with $S \subseteq C$. Moreover assume that $L^S(X)$ has trivial valuations (or more strongly, $X$ has constant coefficients). Then $\trop( L^S(X))$ is the extended Bergman fan of its associated matroid given by Proposition \ref{proposition: matroid of L^S(X)}.
\end{corollary}
    
\begin{example}\label{example: trop L^S(X) for n=4}
    Assume $n=4$, $d=2$. Fix any $3$-dimensional generic space $X \leq K^n$ with trivial valuations. Then $\mathcal{X} = \trop X$ is the extended Bergman fan of $U_{3,4}$. Its intersection with $\mathbb{R}^4/\RR$ is the standard tropical plane (drawn blue in Figures \ref{figure: projected trop lines S=1} and \ref{figure: projected trop lines S=1}).
    We describe $L_{\trop}^S(\mathcal{X}) = \trop(L^S(X))$ for $S = [k]$ with increasing $k \geq 0$. 
    In each case we focus on the largest non empty stratum of $L_{\trop}^S(\mathcal{X})$.

    \mbox{ } \\
    \textbf{Without degeneration i.e. $S = \emptyset$}\label{subsec-withoutdeg}:

    This example was already discussed in \cite{JMRS}. We repeat it here before we study the degenerate versions below. In particular, Figures \ref{fig-DW} and \ref{fig-Bergman1} are taken from \cite{JMRS}. Moreover Figures \ref{figure: projected trop lines S=1} and \ref{figure: projected trop lines S=12} are adaptations of the corresponding illustrations in \cite{JMRS}.
    
    We identify the linear matroid $\mathcal{M}$ of $L(X)$ with $\mathcal{M}(K_4)$.
    %
    The largest nonempty stratum of $L_{\trop}(\mathcal{X})$ lives inside $\RR^{\binom{[4]}{2}}/\RR \cong \RR^6/\RR$ and is given as the Bergman fan of $\mathcal{M}$.
    
    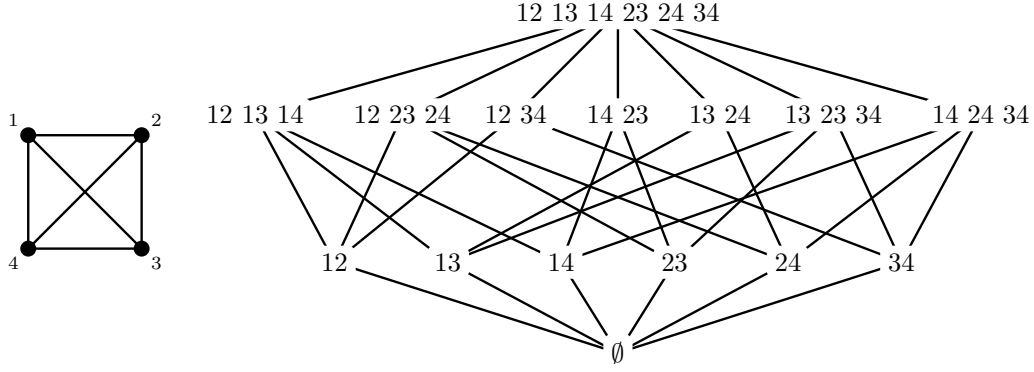
\begin{figure}[ht]
    \begin{center}

    \begin{tikzpicture}[scale = 1.5,
                        edge/.style={black, line width=0.9pt, line cap=round},
                        color = {black}]
    
    
      
    
    


    
      \tikzstyle{vertex}=[draw=black, fill=white, line width=1.4pt]
      \coordinate (trans) at (-1.2, .72);
      
      \coordinate (v1) at (trans);
      \coordinate (v2) at ($(1,0)+(trans)$);
      \coordinate (v3) at ($(1,-1)+(trans)$);
      \coordinate (v4) at ($(0,-1)+(trans)$);
      
      \node at ($(v1)+(-.13,.13)$) {\tiny $1$};
      \node at ($(v2)+(.13,.13)$) {\tiny $2$};
      \node at ($(v3)+(.13,-.13)$) {\tiny $3$};
      \node at ($(v4)+(-.13,-.13)$) {\tiny $4$};
    
      \draw[edge] (v1) -- (v2);
      \draw[edge] (v1) -- (v3);
      \draw[edge] (v1) -- (v4);
      \draw[edge] (v2) -- (v3);
      \draw[edge] (v2) -- (v4);
      \draw[edge] (v3) -- (v4);
    
      \fill (v1) circle (2pt);
      \fill (v2) circle (2pt);
      \fill (v3) circle (2pt);
      \fill (v4) circle (2pt);

    
        \tikzstyle{node}=[text=black, inner sep=3pt, rectangle, rounded corners=3pt,fill=white, draw=none]
        
        \coordinate (H1) at (0.8,.9);
        \coordinate (H2) at (2.1,.9);
        \coordinate (H5) at (3.1,.9);
        \coordinate (H7) at (4,.9);
        \coordinate (H6) at (4.9,.9);
        \coordinate (H3) at (5.9,.9);
        \coordinate (H4) at (7.2,.9);
        \coordinate (v12) at (1.5,-.4);
        \coordinate (v13) at (2.5,-.4);
        \coordinate (v14) at (3.5,-.4);
        \coordinate (v23) at (4.5,-.4);
        \coordinate (v24) at (5.5,-.4);
        \coordinate (v34) at (6.5,-.4);
        \coordinate (b) at (4,-1.2);
        \coordinate (t) at (4,1.8);
        
        \foreach \i/\k in {1/12,1/13,1/14,2/12,2/23,2/24,3/13,3/23,3/34,4/14,4/24,4/34,5/12,5/34,6/13,6/24,7/14,7/23} {
           \draw[edge] (H\i) -- (v\k);
           }
        
        \foreach \i in {1,2,3,4,5,6,7} {
           \draw[edge] (H\i) -- (t);
          }
          
        \foreach \k in {12,13,14,23,24,34} {
           \draw[edge] (v\k) -- (b);
          }
        
        \node[node]  at (b) {\small $\emptyset$};
        
         \node[node] at (H1) {\small $12$ $13$ $14$};
         \node[node] at (H2) {\small $12$ $23$ $24$};
         \node[node] at (H3) {\small $13$ $23$ $34$};
         \node[node] at (H4) {\small $14$ $24$ $34$};
         
         \node[node] at (H5) {\small $12$ $34$};
         \node[node] at (H6) {\small $13$ $24$};
         \node[node] at (H7) {\small $14$ $23$};
          
        \foreach \k in {12,13,14,23,24,34} {
           \node[node] at (v\k) {\small $\k$};
          }
        
        \node[node] at (t) {\small $12$ $13$ $14$ $23$ $24$ $34$};
     
    \end{tikzpicture}

    \caption{The complete graph on $4$ vertices on the left, the lattice of flats of $\mathcal{M}(L(X))$ on the right. }\label{fig-DW}
    
    \end{center}
    \end{figure}
    
    In its fine subdivision, the Bergman fan of this matroid has 13 rays corresponding to the 13 flats.
    For the lattice of flats, see the right part of Figure \ref{fig-DW}.
    In this fan, there is a $2$-dimensional cone spanned by two rays if and only if the two rays correspond to an edge $e$ and a subgraph of $K_4$ strictly containing $e$ whose connected components are cliques.
    The link of this Bergman fan is the Petersen graph, with three additional vertices corresponding to the three disconnected subgraphs of $K_4$ with $2$ edges
    (see Figure \ref{fig-Bergman1}).
    We obtain the coarse subdivision (cf. \cite{AK06}) by dropping those three extra vertices.

    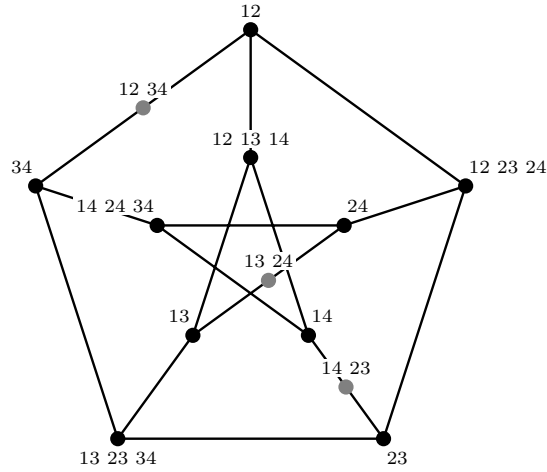
\begin{figure}[ht]
    \begin{center}

    \begin{tikzpicture}[scale = 1.3,
                        edge/.style={black, line width=0.9pt, line cap=round},
                        vertex/.style = {text=black, inner sep=2pt, white, draw=none}
                        ]
    
      \coordinate (p1) at (90:1); 
      \coordinate (p2) at (18:2.3); 
      \coordinate (p3) at (234:2.3); 
      \coordinate (p4) at (162:1); 
    
      \coordinate (p12) at (90:2.3); 
      \coordinate (p13) at (234:1); 
      \coordinate (p14) at (306:1); 
      \coordinate (p23) at (306:2.3); 
      \coordinate (p24) at (18:1); 
      \coordinate (p34) at (162:2.3);

      \draw[edge] (p12) -- (p34) -- (p3) -- (p23) -- (p2) -- cycle; 
      \draw[edge] (p1) -- (p13) -- (p24) -- (p4) -- (p14) -- cycle; 
      \draw[edge] (p1) -- (p12); 
      \draw[edge] (p2) -- (p24); 
      \draw[edge] (p3) -- (p13); 
      \draw[edge] (p4) -- (p34); 
      \draw[edge] (p14) -- (p23); 
    
      \fill (p1) circle [radius=2.2pt];
      \fill (p2) circle [radius=2.2pt];
      \fill (p3) circle [radius=2.2pt];
      \fill (p4) circle [radius=2.2pt];
    
      \fill (p12) circle [radius=2.2pt];
      \fill (p13) circle [radius=2.2pt];
      \fill (p14) circle [radius=2.2pt];
      \fill (p23) circle [radius=2.2pt];
      \fill (p24) circle [radius=2.2pt];
      \fill (p34) circle [radius=2.2pt];
      \fill[gray] ($.5*(p12)+.5*(p34)$) circle [radius=2.2pt];
      \fill[gray] ($.5*(p13)+.5*(p24)$) circle [radius=2.2pt];
      \fill[gray] ($.5*(p14)+.5*(p23)$) circle [radius=2.2pt];
      
      \node[above, fill=white, inner sep=1pt] at ($(p1)+(0,0.1)$) {\tiny $12$ $13$ $14$};
      \node[above right, fill=white, inner sep=1pt] at ($(p2)+(0,0.1)$) {\tiny $12$ $23$ $24$};
      \node[below, fill=white, inner sep=1pt] at ($(p3)-(0,0.1)$) {\tiny $13$ $23$ $34$};
      \node[above left, fill=white, inner sep=1pt] at ($(p4)+(0,0.1)$) {\tiny $14$ $24$ $34$};

      \node[above, fill=white, inner sep=1pt] at ($(p12)+(0,0.1)$){\tiny $12$};
      \node[above left, fill=white, inner sep=1pt] at ($(p13)+(0,0.1)$){\tiny $13$}; 
      \node[above right, fill=white, inner sep=1pt] at ($(p14)+(0,0.1)$){\tiny $14$};
      \node[below right, fill=white, inner sep=1pt] at ($(p23)-(0,0.1)$){\tiny $23$};    
      \node[above right, fill=white, inner sep=1pt] at ($(p24)+(0,0.1)$){\tiny $24$};        
      \node[above left, fill=white, inner sep=1pt] at ($(p34)+(0,0.1)$){\tiny $34$};    
      
      \node[above, fill=white, inner sep=1pt] at ($.5*(p12)+.5*(p34)+(0,0.1)$) {\tiny $12$ $34$};
      \node[above, fill=white, inner sep=1pt] at ($.5*(p13)+.5*(p24)+(0,0.1)$) {\tiny $13$ $24$};
      \node[above, fill=white, inner sep=1pt] at ($.5*(p14)+.5*(p23)+(0,0.1)$) {\tiny $14$ $23$};

    \end{tikzpicture}
    
    \caption{The link of $L_{\trop}(\mathcal{X})$ is the Petersen graph with $3$ extra vertices (gray).}\label{fig-Bergman1}
    \end{center}
    \end{figure}
    
    We refer to \cite[Example 25]{JMRS} for a detailed description of how points in the various cones of $L_{\trop}(\mathcal{X})$ parameterize tropical lines inside $\mathcal X$.

    \mbox{ } \\
    \textbf{Projection with $S=\{1\}$}\label{subsec-S=1}:
    
    By Proposition \ref{proposition: matroid of L^S(X) in cocodim 1}, the linear matroid $\mathcal{M}^1$ of $L^1(X)$ is isomorphic to $\mathcal{M}(K_{\{2,3,4\}}) \oplus U_{1,3}$. 
    We show its lattice of flats in Figure \ref{figure: generic lattice for L^1(X)}. Note that compared to the previous case $23,24$ and $34$ are now parallel.
    The largest non empty stratum of $L_{\trop}^S(\mathcal{X})$ again lives inside $\RR^6/\RR$ and is the Bergman fan of $\mathcal{M}^1$.
    %
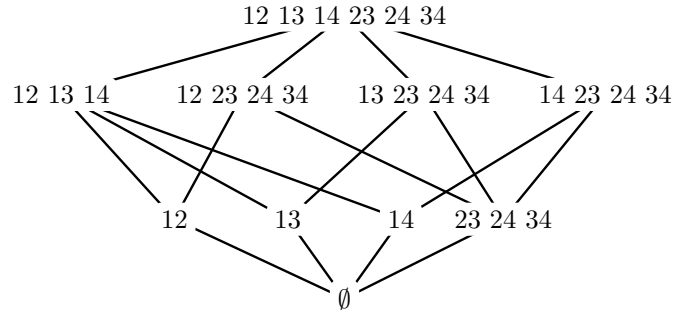
\begin{figure}[ht]
\begin{center}
    
    \begin{tikzpicture}[scale = 1.5,
                        edge/.style={black, line width=0.9pt, line cap=round},
                        color = {black}]
      
    
    \tikzstyle{node}=[text=black, inner sep=3pt, rectangle, rounded corners=3pt,fill=white, draw=none]
    
    \coordinate (H123)  at (0, .9);
    \coordinate (H1456) at (1.6, .9);
    \coordinate (H2456) at (3.2, .9);
    \coordinate (H3456) at (4.8, .9);
    
    \coordinate (v1) at (1,-.2);
    \coordinate (v2) at (2,-.2);
    \coordinate (v3) at (3,-.2);
    \coordinate (v456) at (3.9,-.2);
    \coordinate (b) at (2.5,-0.9);
    \coordinate (t) at (2.5,1.6);
    
    \foreach \i/\k in {123/1, 123/2,123/3,1456/1,1456/456,2456/2,2456/456,3456/3,3456/456} {
       \draw[edge] (H\i) -- (v\k);
       }
    
    \foreach \i in {123, 1456,2456,3456} {
       \draw[edge] (H\i) -- (t);
      }
      
    \foreach \k in {1,2,3,456} {
       \draw[edge] (v\k) -- (b);
      }
    
    \node[node]  at (b) {\small $\emptyset$};
    
    \node[node] at (H123) {\small $12~13~14$};
    \node[node] at (H1456) {\small $12~23~24~34$};
    \node[node] at (H2456) {\small $13~23~24~34$};
    \node[node] at (H3456) {\small $14~23~24~34$};

    \node[node] at (v1) {\small $12$};
    \node[node] at (v2) {\small $13$};
    \node[node] at (v3) {\small $14$};
    \node[node] at (v456) {\small $23~24~34$};
    
    \node[node] at (t) {\small $$12~13~14~23~24~34$$};
    
    \end{tikzpicture}
    
    \caption{The lattice of flats $\mathcal{L}(\mathcal{M}(L^1(X)))$ for a generic space $X \leq K^4$.}\label{figure: generic lattice for L^1(X)}

\end{center}
\end{figure}
    %
    %
    The link of this Bergman fan is depicted in Figure \ref{fig-Bergman-S=1}.
    
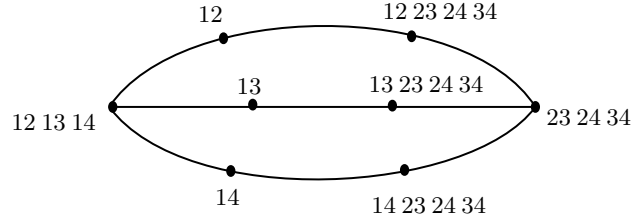
\begin{figure}[ht]
\centering
        
    \tikzset{every picture/.style={line width=0.75pt}} 
    \begin{tikzpicture}[x=0.75pt,y=0.75pt,yscale=-0.9,xscale=0.9]
    
    \draw  [fill={rgb, 255:red, 0; green, 0; blue, 0 }  ,fill opacity=1 ] (429.78,259.43) .. controls (429.78,257.92) and (428.88,256.7) .. (427.77,256.7) .. controls (426.66,256.7) and (425.76,257.92) .. (425.76,259.43) .. controls (425.76,260.94) and (426.66,262.16) .. (427.77,262.16) .. controls (428.88,262.16) and (429.78,260.94) .. (429.78,259.43) -- cycle ;
    \draw    (190.07,259.43) .. controls (230.64,199.65) and (389.1,198.85) .. (427.77,259.43) ;
    \draw    (190.07,259.43) .. controls (233.81,317.62) and (391,308.85) .. (427.77,259.43) ;
    \draw  [fill={rgb, 255:red, 0; green, 0; blue, 0 }  ,fill opacity=1 ] (259.93,295.15) .. controls (259.93,293.64) and (259.03,292.42) .. (257.91,292.42) .. controls (256.8,292.42) and (255.9,293.64) .. (255.9,295.15) .. controls (255.9,296.66) and (256.8,297.88) .. (257.91,297.88) .. controls (259.03,297.88) and (259.93,296.66) .. (259.93,295.15) -- cycle ;
    \draw  [fill={rgb, 255:red, 0; green, 0; blue, 0 }  ,fill opacity=1 ] (255.85,221.13) .. controls (255.85,219.62) and (254.95,218.4) .. (253.84,218.4) .. controls (252.73,218.4) and (251.83,219.62) .. (251.83,221.13) .. controls (251.83,222.64) and (252.73,223.86) .. (253.84,223.86) .. controls (254.95,223.86) and (255.85,222.64) .. (255.85,221.13) -- cycle ;
    \draw  [fill={rgb, 255:red, 0; green, 0; blue, 0 }  ,fill opacity=1 ] (360.76,220.03) .. controls (360.76,218.52) and (359.86,217.3) .. (358.75,217.3) .. controls (357.64,217.3) and (356.73,218.52) .. (356.73,220.03) .. controls (356.73,221.54) and (357.64,222.76) .. (358.75,222.76) .. controls (359.86,222.76) and (360.76,221.54) .. (360.76,220.03) -- cycle ;
    \draw  [fill={rgb, 255:red, 0; green, 0; blue, 0 }  ,fill opacity=1 ] (194.1,259.43) .. controls (194.1,257.92) and (193.19,256.7) .. (192.08,256.7) .. controls (190.97,256.7) and (190.07,257.92) .. (190.07,259.43) .. controls (190.07,260.94) and (190.97,262.16) .. (192.08,262.16) .. controls (193.19,262.16) and (194.1,260.94) .. (194.1,259.43) -- cycle ;
    \draw  [fill={rgb, 255:red, 0; green, 0; blue, 0 }  ,fill opacity=1 ] (356.82,294.58) .. controls (356.82,293.07) and (355.92,291.85) .. (354.8,291.85) .. controls (353.69,291.85) and (352.79,293.07) .. (352.79,294.58) .. controls (352.79,296.09) and (353.69,297.31) .. (354.8,297.31) .. controls (355.92,297.31) and (356.82,296.09) .. (356.82,294.58) -- cycle ;
    \draw    (192.08,259.43) -- (427.77,259.43) ;
    \draw  [fill={rgb, 255:red, 0; green, 0; blue, 0 }  ,fill opacity=1 ] (272.15,258.14) .. controls (272.15,256.63) and (271.25,255.41) .. (270.14,255.41) .. controls (269.03,255.41) and (268.13,256.63) .. (268.13,258.14) .. controls (268.13,259.65) and (269.03,260.87) .. (270.14,260.87) .. controls (271.25,260.87) and (272.15,259.65) .. (272.15,258.14) -- cycle ;
    \draw  [fill={rgb, 255:red, 0; green, 0; blue, 0 }  ,fill opacity=1 ] (350.03,258.71) .. controls (350.03,257.2) and (349.13,255.98) .. (348.01,255.98) .. controls (346.9,255.98) and (346,257.2) .. (346,258.71) .. controls (346,260.22) and (346.9,261.44) .. (348.01,261.44) .. controls (349.13,261.44) and (350.03,260.22) .. (350.03,258.71) -- cycle ;
    
    \draw (349.34,246.82) node  [font=\footnotesize] [align=left] {\begin{minipage}[lt]{19.56pt}\setlength\topsep{0pt}
    $\displaystyle 13~23~24~34$
    \end{minipage}};
    \draw (350.88,314.08) node  [font=\footnotesize] [align=left] {\begin{minipage}[lt]{19.56pt}\setlength\topsep{0pt}
    $\displaystyle 14~23~24~34$
    \end{minipage}};
    \draw (448.53,265.11) node  [font=\footnotesize] [align=left] {\begin{minipage}[lt]{19.56pt}\setlength\topsep{0pt}
    $\displaystyle 23~24~34$
    \end{minipage}};
    \draw (357.25,206.26) node  [font=\footnotesize] [align=left] {\begin{minipage}[lt]{19.56pt}\setlength\topsep{0pt}
    $\displaystyle 12~23~24~34$
    \end{minipage}};
    \draw (254.27,207.86) node  [font=\footnotesize] [align=left] {\begin{minipage}[lt]{19.56pt}\setlength\topsep{0pt}
    $\displaystyle 12$
    \end{minipage}};
    \draw (263.74,309.82) node  [font=\footnotesize] [align=left] {\begin{minipage}[lt]{19.56pt}\setlength\topsep{0pt}
    $\displaystyle 14$
    \end{minipage}};
    \draw (275.81,247.87) node  [font=\footnotesize] [align=left] {\begin{minipage}[lt]{19.56pt}\setlength\topsep{0pt}
    $\displaystyle 13$
    \end{minipage}};
    \draw (150.04,268.06) node  [font=\footnotesize] [align=left] {\begin{minipage}[lt]{19.56pt}\setlength\topsep{0pt}
    $\displaystyle 12~13~14$
    \end{minipage}};
    
    \end{tikzpicture}
    
    \caption{The link of the Bergman fan of $L^1(X)$.}
    \label{fig-Bergman-S=1}
\end{figure}
    
    We can coarsen the fan structure by dropping the rays which correspond to $2$-valent vertices in Figure \ref{fig-Bergman-S=1}. 
    Then we have two rays which are opposite to each other, and three $2$-dimensional cones. Up to lineality, the fan thus looks like a tropical line in the plane. 
    The three $2$-dimensional cones correspond to the three combinatorial types for tropical (non-degenerate) lines in $\mathbb{R}^3$. 
    Since the projection of our line is required to be contained in the tropical plane, i.e.\ in the $x_1 = \infty$-part of its boundary, one vertex of our tropical line (namely the one not adjacent to the end of direction $e_1$) has to be contained in the span of the ray in direction $e_1$. 
    The two opposite rays $\{12,13,14\}$ resp. $\{23,24,34\}$ correspond to tropical lines for which this vertex is, resp.\ is not contained in the ray of $e_1$. 
    For each of the three $2$-dimensional cones in the coarse subdivision, the refinement into three two-dimensional cones distinguishes the following different relative positions of the tropical line and the plane: 
    the cone closest to $12,13,14$ parametrizes tropical lines for which both vertices are contained in the plane, the next cone parametrizes lines for which one vertex is contained in the plane but the other one (the one not adjacent to the end of direction $e_1$) is not, and the last parametrizes tropical lines for which none of the vertices is contained in the tropical plane. 
    Their projection of course always is, as discussed before. 
    We illustrate a few examples in Figure \ref{figure: projected trop lines S=1}
    
    \newsavebox{\tropLineOne}
    \sbox{\tropLineOne}{
    \begin{tikzpicture}[
    	x  = {(-0.9cm,0.076cm)},
            y  = {(0.0cm,-1cm)},
            z  = {(.6cm,.4cm)},
      edge/.style={line width=.8pt, line cap=round, teal},
      edge2/.style={line width=.8pt, line cap=round, blue},
      edge4/.style={line width=.8pt, line cap=round, blue, dotted},
      face/.style={draw=none, opacity=0.2, fill=blue},
      vertex/.style={draw = red, fill = red, line width=1.5pt},
      edge3/.style={line width=1.3pt, line cap=round, red, dotted},
      scale=1.2
      ]
    
      \coordinate (p0) at (0,0,0); 
      \coordinate (p1) at (-1,1,1); 
      \coordinate (p2) at (1,-1,1); 
      \coordinate (p3) at (1,1,-1); 
      \coordinate (p4) at (-1,-1,-1); 
    
      \coordinate (p12) at ($(p1)+(p2)$);
      \coordinate (p13) at ($(p1)+(p3)$);
      \coordinate (p14) at ($(p1)+(p4)$);
      \coordinate (p23) at ($(p2)+(p3)$);
      \coordinate (p24) at ($(p2)+(p4)$);
      \coordinate (p34) at ($(p3)+(p4)$);
    
      \node[label = right:{\tiny $1$}]  at (p1) {};
      \node[label = above:{\tiny $2$}]  at (p2) {};
      \node[label = left:{\tiny $3$}]  at (p3) {};
      \node[label = above:{\tiny $4$}]  at (p4) {};
    
      \node[label = right:{\tiny $12$}]  at (p12) {};
      \node[label = below:{\tiny $13$}]  at (p13) {};
      \node[label = right:{\tiny $14$}]  at (p14) {};
      \node[label = left:{\tiny $23$}]  at (p23) {};
      \node[label = above:{\tiny $24$}]  at (p24) {};
      \node[label = left:{\tiny $34$}]  at (p34) {};
    
      \draw[face] (p0) -- (p2) -- (p24) -- (p4) -- cycle;
      \draw[edge] (p0) -- (p24); 
      \draw[face] (p0) -- (p1) -- (p14) -- (p4) -- cycle;
      \draw[edge] (p0) -- (p14); 
      \draw[face] (p0) -- (p3) -- (p34) -- (p4) -- cycle;
      \draw[edge] (p0) -- (p34); 
      \draw[edge2] (p0) -- (p4); 
      \draw[edge4] (p0) -- ($-.85*(p1)$);

      \draw[edge3] ($-.7*(p1)$) -- ($(p4)-.7*(p1)$);
      \draw[edge3] ($-.7*(p1)$) -- ($(p0)$); 
      \draw[edge3] ($-.7*(p1)+0.7*(p3)$) -- ($-.7*(p1)$); 
      \draw[edge3] ($-.7*(p1)$) -- ($(p2)-.7*(p1)$);

      \draw[vertex] ($-.7*(p1)$) circle (.6pt);
    
      \draw[face] (p0) -- (p1) -- (p12) -- (p2) -- cycle; 
      \draw[edge] (p0) -- (p12);
      \draw[edge2] (p0) -- (p1);
    
      \draw[face] (p0) -- (p1) -- (p13) -- (p3) -- cycle; 
      \draw[edge] (p0) -- (p13);
    
      \draw[face] (p0) -- (p2) -- (p23) -- (p3) -- cycle; 
      \draw[edge] (p0) -- (p23);
      \draw[edge2] (p0) -- (p2); 
      \draw[edge2] (p0) -- (p3); 
    
      \draw[edge3] ($(p0)$) -- ($(p1)$); 
      
    \end{tikzpicture}}

    \newsavebox{\tropLineTwo}
    \sbox{\tropLineTwo}{
    \begin{tikzpicture}[
    	x  = {(-0.9cm,0.076cm)},
            y  = {(0.0cm,-1cm)},
            z  = {(.6cm,.4cm)},
      edge/.style={line width=.8pt, line cap=round, teal},
      edge2/.style={line width=.8pt, line cap=round, blue},
      edge4/.style={line width=.8pt, line cap=round, blue, dotted},
      face/.style={draw=none, opacity=0.2, fill=blue},
      vertex/.style={draw = red, fill = red, line width=1.5pt},
      edge3/.style={line width=1.3pt, line cap=round, red, dotted},
      scale=1.2
      ]
    
      \coordinate (p0) at (0,0,0); 
      \coordinate (p1) at (-1,1,1); 
      \coordinate (p2) at (1,-1,1); 
      \coordinate (p3) at (1,1,-1); 
      \coordinate (p4) at (-1,-1,-1); 
    
      \coordinate (p12) at ($(p1)+(p2)$);
      \coordinate (p13) at ($(p1)+(p3)$);
      \coordinate (p14) at ($(p1)+(p4)$);
      \coordinate (p23) at ($(p2)+(p3)$);
      \coordinate (p24) at ($(p2)+(p4)$);
      \coordinate (p34) at ($(p3)+(p4)$);
    
      \node[label = right:{\tiny $1$}]  at (p1) {};
      \node[label = above:{\tiny $2$}]  at (p2) {};
      \node[label = left:{\tiny $3$}]  at (p3) {};
      \node[label = above:{\tiny $4$}]  at (p4) {};
    
      \node[label = right:{\tiny $12$}]  at (p12) {};
      \node[label = below:{\tiny $13$}]  at (p13) {};
      \node[label = right:{\tiny $14$}]  at (p14) {};
      \node[label = left:{\tiny $23$}]  at (p23) {};
      \node[label = above:{\tiny $24$}]  at (p24) {};
      \node[label = left:{\tiny $34$}]  at (p34) {};
    
      \draw[face] (p0) -- (p2) -- (p24) -- (p4) -- cycle;
      \draw[edge] (p0) -- (p24); 
      \draw[face] (p0) -- (p1) -- (p14) -- (p4) -- cycle;
      \draw[edge] (p0) -- (p14); 
      \draw[face] (p0) -- (p3) -- (p34) -- (p4) -- cycle;
      \draw[edge] (p0) -- (p34); 
      \draw[edge2] (p0) -- (p4); 
      \draw[edge4] (p0) -- ($-.85*(p1)$);

      \draw[edge3] ($-.7*(p1)$) -- ($(p4)-.7*(p1)$);
      \draw[edge3] ($-.4*(p1)+.3*(p3)$) -- ($0.3*(p3)$); 
      \draw[edge3] ($-.4*(p1)+(p3)$) -- ($0.3*(p3)-.4*(p1)$); 
      \draw[edge3] ($-.7*(p1)$) -- ($(p2)-.7*(p1)$);
      \draw[edge3] ($-.4*(p1)+0.3*(p3)$) -- ($-.7*(p1)$);
      \draw[vertex] ($-.4*(p1)+0.3*(p3)$) circle (.6pt);
      \draw[vertex] ($-.7*(p1)$) circle (.6pt);
    
      \draw[face] (p0) -- (p1) -- (p12) -- (p2) -- cycle; 
      \draw[edge] (p0) -- (p12);
      \draw[edge2] (p0) -- (p1);
    
      \draw[face] (p0) -- (p1) -- (p13) -- (p3) -- cycle; 
      \draw[edge] (p0) -- (p13);
    
      \draw[face] (p0) -- (p2) -- (p23) -- (p3) -- cycle; 
      \draw[edge] (p0) -- (p23);
      \draw[edge2] (p0) -- (p2); 
      \draw[edge2] (p0) -- (p3); 
    
      \draw[edge3] ($.3*(p3)$) -- ($(p1)+0.3*(p3)$); 

    \end{tikzpicture}}
    
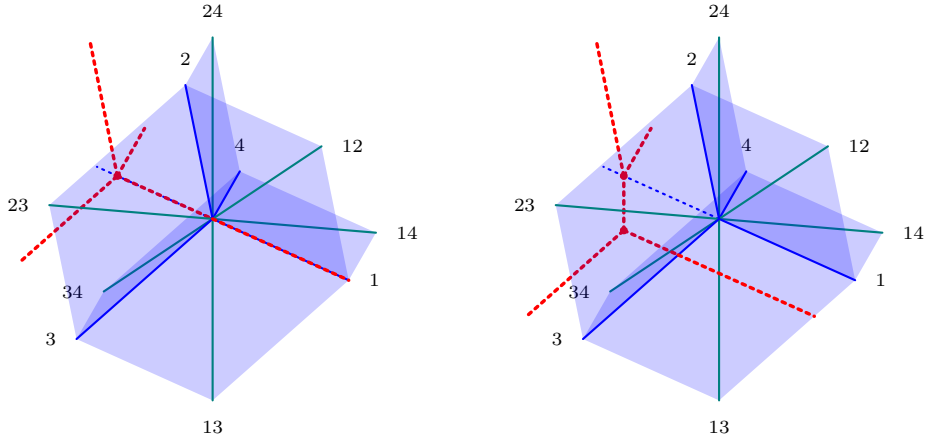
\begin{figure}[ht]
\begin{center}
    \begin{tikzpicture}
        \node at (0,0)   (A) {\usebox{\tropLineOne}};
        \node at (6.7,0)   (B) {\usebox{\tropLineTwo}};
    \end{tikzpicture}
    
    \caption{The combinatorial type of the tropical lines parametrized by the ray $23~24~34$ on the left, and by the $2$-dimensional cone spanned by $23~24~34$ and $13~23~24~34$ on the right in the moduli space $L^1_{\trop}(\mathcal{X})$ in its fine subdivision (see Figure \ref{fig-Bergman-S=1}).}\label{figure: projected trop lines S=1}
    
\end{center}
\end{figure}

    \mbox{ } \\
    \textbf{Projection with $S=\{1,2\}$}\label{subsec-S=12}:
    
    The linear matroid $\mathcal{M}^{12}$ of $L^{12}(X)$ is isomorphic to $\mathcal{M}(K_{\{3,4\}}) \oplus U_{1,2}^2 \oplus U_{0,1}$. Compared to the previous case, $34$ becomes a loop and $13$ is now parallel to $14$. 
    %
    %
    The lattice of flats of $\mathcal{M}^{12}$ is depicted in Figure \ref{figure: generic lattice for L^12(X)}.
    In the language of Section \ref{section: 3 - Preliminaries on tropical geometry}, the largest non empty stratum is the union of cones $C_{F_\bullet}$ as in equation \eqref{equation: cones for chaines above flats} with fixed $F_0 = \{34\}$. 
    After projecting away the infinite coordinate, this is the same as the Bergman fan of $\mathcal{M}^{12}/\small{34}$.
    The link of this fan is depicted in Figure \ref{fig-Bergman-S=12}, where we dropped the loop $34$.

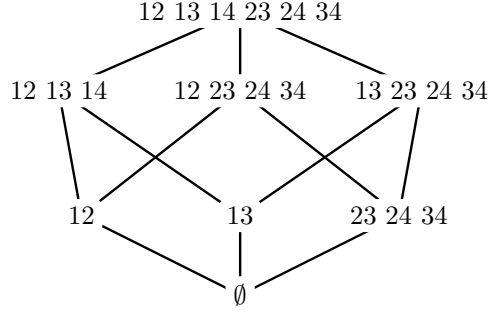
\begin{figure}[ht]
\begin{center}
    
    \begin{tikzpicture}[scale = 1.5,
                        edge/.style={black, line width=0.9pt, line cap=round},
                        color = {black}]
      
    
    \tikzstyle{node}=[text=black, inner sep=3pt, rectangle, rounded corners=3pt,fill=white, draw=none]
    
    \coordinate (H1236)  at (0, .9);
    \coordinate (H1456) at (1.6, .9);
    \coordinate (H23456) at (3.2, .9);
    
    \coordinate (v16) at (0.2,-.2);
    \coordinate (v236) at (1.6,-.2);
    \coordinate (v456) at (3,-.2);
    \coordinate (b) at (1.6,-0.9);
    \coordinate (t) at (1.6,1.6);
    
    \foreach \i/\k in {1236/16, 1236/236,1456/16,1456/456,23456/236,23456/456} {
       \draw[edge] (H\i) -- (v\k);
       }
    
    \foreach \i in {1236, 1456,23456} {
       \draw[edge] (H\i) -- (t);
      }
      
    \foreach \k in {16,236,456} {
       \draw[edge] (v\k) -- (b);
      }
    
    \node[node]  at (b) {\small $\emptyset$};
    
    \node[node] at (H1236) {\small $12~13~14$};
    \node[node] at (H1456) {\small $12~23~24~34$};
    \node[node] at (H23456) {\small $13~23~24~34$};

    \node[node] at (v16) {\small $12$};
    \node[node] at (v236) {\small $13$};
    \node[node] at (v456) {\small $23~24~34$};
    
    \node[node] at (t) {\small $$12~13~14~23~24~34$$};

    \end{tikzpicture}
    
    \caption{The lattice of flats $\mathcal{L}(\mathcal{M}(L^{12}(X)))$ for a generic space $X \leq K^4$.}\label{figure: generic lattice for L^12(X)}

\end{center}
\end{figure}
    
\begin{figure}[ht]
    \centering
    
    \tikzset{every picture/.style={line width=0.75pt}} 
    
    \begin{tikzpicture}[x=0.75pt,y=0.75pt,yscale=-0.9,xscale=0.9]
    
    \draw  [fill={rgb, 255:red, 0; green, 0; blue, 0 }  ,fill opacity=1 ] (429.78,259.43) .. controls (429.78,257.92) and (428.88,256.7) .. (427.77,256.7) .. controls (426.66,256.7) and (425.76,257.92) .. (425.76,259.43) .. controls (425.76,260.94) and (426.66,262.16) .. (427.77,262.16) .. controls (428.88,262.16) and (429.78,260.94) .. (429.78,259.43) -- cycle ;
    \draw    (190.07,259.43) .. controls (230.64,199.65) and (389.1,198.85) .. (427.77,259.43) ;
    \draw    (190.07,259.43) .. controls (233.81,317.62) and (391,308.85) .. (427.77,259.43) ;
    \draw  [fill={rgb, 255:red, 0; green, 0; blue, 0 }  ,fill opacity=1 ] (259.93,295.15) .. controls (259.93,293.64) and (259.03,292.42) .. (257.91,292.42) .. controls (256.8,292.42) and (255.9,293.64) .. (255.9,295.15) .. controls (255.9,296.66) and (256.8,297.88) .. (257.91,297.88) .. controls (259.03,297.88) and (259.93,296.66) .. (259.93,295.15) -- cycle ;
    \draw  [fill={rgb, 255:red, 0; green, 0; blue, 0 }  ,fill opacity=1 ] (255.85,221.13) .. controls (255.85,219.62) and (254.95,218.4) .. (253.84,218.4) .. controls (252.73,218.4) and (251.83,219.62) .. (251.83,221.13) .. controls (251.83,222.64) and (252.73,223.86) .. (253.84,223.86) .. controls (254.95,223.86) and (255.85,222.64) .. (255.85,221.13) -- cycle ;
    \draw  [fill={rgb, 255:red, 0; green, 0; blue, 0 }  ,fill opacity=1 ] (360.76,220.03) .. controls (360.76,218.52) and (359.86,217.3) .. (358.75,217.3) .. controls (357.64,217.3) and (356.73,218.52) .. (356.73,220.03) .. controls (356.73,221.54) and (357.64,222.76) .. (358.75,222.76) .. controls (359.86,222.76) and (360.76,221.54) .. (360.76,220.03) -- cycle ;
    \draw  [fill={rgb, 255:red, 0; green, 0; blue, 0 }  ,fill opacity=1 ] (194.1,259.43) .. controls (194.1,257.92) and (193.19,256.7) .. (192.08,256.7) .. controls (190.97,256.7) and (190.07,257.92) .. (190.07,259.43) .. controls (190.07,260.94) and (190.97,262.16) .. (192.08,262.16) .. controls (193.19,262.16) and (194.1,260.94) .. (194.1,259.43) -- cycle ;
    \draw  [fill={rgb, 255:red, 0; green, 0; blue, 0 }  ,fill opacity=1 ] (356.82,294.58) .. controls (356.82,293.07) and (355.92,291.85) .. (354.8,291.85) .. controls (353.69,291.85) and (352.79,293.07) .. (352.79,294.58) .. controls (352.79,296.09) and (353.69,297.31) .. (354.8,297.31) .. controls (355.92,297.31) and (356.82,296.09) .. (356.82,294.58) -- cycle ;
    
    \draw (350.88,314.08) node  [font=\footnotesize] [align=left] {\begin{minipage}[lt]{19.56pt}\setlength\topsep{0pt}
    $\displaystyle 13~14~23~24$
    \end{minipage}};
    \draw (448.53,265.11) node  [font=\footnotesize] [align=left] {\begin{minipage}[lt]{19.56pt}\setlength\topsep{0pt}
    $\displaystyle 23~24$
    \end{minipage}};
    \draw (357.25,206.26) node  [font=\footnotesize] [align=left] {\begin{minipage}[lt]{19.56pt}\setlength\topsep{0pt}
    $\displaystyle 12~23~24$
    \end{minipage}};
    \draw (254.27,207.86) node  [font=\footnotesize] [align=left] {\begin{minipage}[lt]{19.56pt}\setlength\topsep{0pt}
    $\displaystyle 12$
    \end{minipage}};
    \draw (263.74,309.82) node  [font=\footnotesize] [align=left] {\begin{minipage}[lt]{19.56pt}\setlength\topsep{0pt}
    $\displaystyle 13~14$
    \end{minipage}};
    \draw (150.04,268.06) node  [font=\footnotesize] [align=left] {\begin{minipage}[lt]{19.56pt}\setlength\topsep{0pt}
    $\displaystyle 12~13~14$
    \end{minipage}};
    
    \end{tikzpicture}
    
    \caption{The link of the Bergman fan of $L^{12}(X)$.}
    \label{fig-Bergman-S=12}
\end{figure}
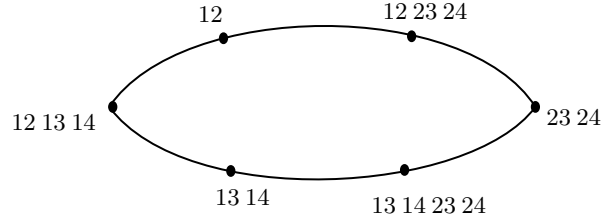
    
    We now describe how this Bergman fan parametrizes lines whose projection with $S=\{1,2\}$ lies in the tropical plane (see Figure \ref{figure: projected trop lines S=12} for an example).
    The tropical plane meets the boundary $\{x_1 = \infty,x_2 = \infty\}$ in one point, namely the $0$-point. 
    Any tropical line projecting to this point must be contained in the plane $x_3 = 0$, hence it must be a degenerate line with three rays of direction $e_1$, $e_2$ and $-(e_1+e_2)$ (this corresponds to the fact that the Pl\"ucker coordinate $p_{3,4}$ of the line is $\infty$). 
    Up to lineality, the Bergman fan is just a point. The lineality space parametrizes the position of the vertex of the tropical line in the plane $x_3=0$. 
    The subdivision into six cones corresponds to the subdivision of the plane $x_3=0$ into six cones as follows: take the four orthants, and subdivide the positive and negative orthant further with the diagonal. \\

    \newsavebox{\tropLineThree}
    \sbox{\tropLineThree}{
    \begin{tikzpicture}[
    	x  = {(-0.9cm,0.076cm)},
            y  = {(0.0cm,-1cm)},
            z  = {(.6cm,.4cm)},
      edge/.style={line width=.8pt, line cap=round, teal},
      edge2/.style={line width=.8pt, line cap=round, blue},
      edge4/.style={line width=.8pt, line cap=round, blue, dotted},
      face/.style={draw=none, opacity=0.2, fill=blue},
      vertex/.style={draw = red, fill = red, line width=1.5pt},
      edge3/.style={line width=1.3pt, line cap=round, red, dotted},
      scale=1.2
      ]
    
      \coordinate (p0) at (0,0,0); 
      \coordinate (p1) at (-1,1,1); 
      \coordinate (p2) at (1,-1,1); 
      \coordinate (p3) at (1,1,-1); 
      \coordinate (p4) at (-1,-1,-1); 
    
      \coordinate (p12) at ($(p1)+(p2)$);
      \coordinate (p13) at ($(p1)+(p3)$);
      \coordinate (p14) at ($(p1)+(p4)$);
      \coordinate (p23) at ($(p2)+(p3)$);
      \coordinate (p24) at ($(p2)+(p4)$);
      \coordinate (p34) at ($(p3)+(p4)$);

      \coordinate (t) at ($(p1)+.2*(p2)$);
    
      \node[label = right:{\tiny $1$}]  at (p1) {};
      \node[label = above:{\tiny $2$}]  at (p2) {};
      \node[label = left:{\tiny $3$}]  at (p3) {};
      \node[label = above:{\tiny $4$}]  at (p4) {};
    
      \node[label = right:{\tiny $12$}]  at (p12) {};
      \node[label = below:{\tiny $13$}]  at (p13) {};
      \node[label = right:{\tiny $14$}]  at (p14) {};
      \node[label = left:{\tiny $23$}]  at (p23) {};
      \node[label = above:{\tiny $24$}]  at (p24) {};
      \node[label = left:{\tiny $34$}]  at (p34) {};
    
      \draw[face] (p0) -- (p2) -- (p24) -- (p4) -- cycle;
      \draw[edge] (p0) -- (p24); 
      \draw[face] (p0) -- (p1) -- (p14) -- (p4) -- cycle;
      \draw[edge] (p0) -- (p14); 
      \draw[face] (p0) -- (p3) -- (p34) -- (p4) -- cycle;
      \draw[edge] (p0) -- (p34); 
      \draw[edge2] (p0) -- (p4); 

      \draw[edge3] ($-2.15*(p1)-1.45*(p2)+(t)$) -- ($-.9*(p1)-.2*(p2)+(t)$);
      \draw[vertex] ($-.7*(p1)+(t)$) circle (.6pt);
    
      \draw[face] (p0) -- (p1) -- (p12) -- (p2) -- cycle; 
      \draw[edge] (p0) -- (p12);
      \draw[edge2] (p0) -- (p1);
    
      \draw[face] (p0) -- (p1) -- (p13) -- (p3) -- cycle; 
      \draw[edge] (p0) -- (p13);
    
      \draw[face] (p0) -- (p2) -- (p23) -- (p3) -- cycle; 
      \draw[edge] (p0) -- (p23);
      \draw[edge2] (p0) -- (p2); 
      \draw[edge2] (p0) -- (p3); 

      \draw[edge3] ($-.7*(p1)+(t)$) -- ($(p2)-.7*(p1)+(t)$);
      \draw[edge3] ($-.7*(p1)+(t)$) -- ($(p0)+(t)$); 
      \draw[edge3] ($-.9*(p1)-.2*(p2)+(t)$) -- ($-.7*(p1)+(t)$);

    \end{tikzpicture}}  
    
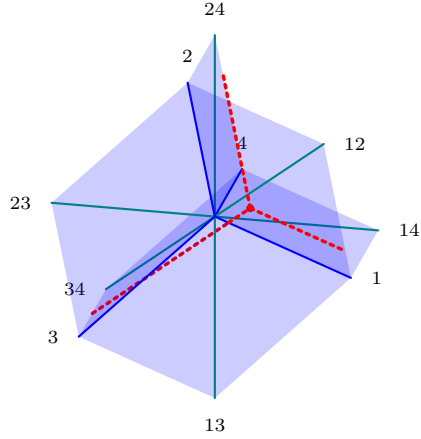
\begin{figure}[ht]
\begin{center}
    \begin{tikzpicture}
        \node at (0,0)   (A) {\usebox{\tropLineThree}};
    \end{tikzpicture}
    
    \caption{The combinatorial type of the tropical lines parametrized by the $2$-dimensional cone spanned by the rays $12$ and $12~13~14$ in the moduli space $L^{12}_{\trop}(\mathcal{X})$ in its fine subdivision.}\label{figure: projected trop lines S=12}
\end{center}
\end{figure}

    \mbox{ } \\
    \textbf{Projection with $S=\{1,2,3\}$}\label{subsec-S=123}:

    Since $\mathcal{M}(K_{\{4\}}) = \emptyset$, the linear matroid $\mathcal{M}^{123}$ of $L^{123}(X)$ is isomorphic to $U_{1,1}^3 \oplus U_{0,3}$. Now $14, 24$ and $34$ are loops, and the unique basis is given by $\{12,13,23\}$. The corresponding lattice of flats is shown in Figure \ref{figure: generic lattice for L^123(X)}. 
    The largest non empty stratum of $L_{\trop}^{123}(\mathcal{X})$ is given by cones $C_{F_\bullet}$ corresponding to chains of flats starting with $F_0 = \{14,24,34\}$.
    Projecting away the infinite coordinates, this is the Bergman fan of $\mathcal{M}^{123}/\small{F_0}$.

    \begin{figure}[ht]
\begin{center}
    
    \begin{tikzpicture}[scale = 1.5,
                        edge/.style={black, line width=0.9pt, line cap=round},
                        color = {black}]
      
    
    \tikzstyle{node}=[text=black, inner sep=3pt, rectangle, rounded corners=3pt,fill=white, draw=none]
    
    \coordinate (H1236)  at (0, .9);
    \coordinate (H1456) at (2, .9);
    \coordinate (H23456) at (4, .9);
    
    \coordinate (v16) at (0.3,-.2);
    \coordinate (v236) at (2,-.2);
    \coordinate (v456) at (3.7,-.2);
    \coordinate (b) at (2,-0.9);
    \coordinate (t) at (2,1.6);
    
    \foreach \i/\k in {1236/16, 1236/236,1456/16,1456/456,23456/236,23456/456} {
       \draw[edge] (H\i) -- (v\k);
       }
    
    \foreach \i in {1236, 1456,23456} {
       \draw[edge] (H\i) -- (t);
      }
      
    \foreach \k in {16,236,456} {
       \draw[edge] (v\k) -- (b);
      }
    
    \node[node]  at (b) {\small $\emptyset$};
    
    \node[node] at (H1236) {\small  $12~13~14~24~34$};
    \node[node] at (H1456) {\small  $12~23~14~24~34$};
    \node[node] at (H23456) {\small $13~23~14~24~34$};

    \node[node] at (v16) {\small $12~14~24~34$};
    \node[node] at (v236) {\small $13~14~24~34$};
    \node[node] at (v456) {\small $23~14~24~34$};
    
    \node[node] at (t) {\small $$12~13~14~23~24~34$$};

    \end{tikzpicture}
    
    \caption{The lattice of flats $\mathcal{L}(\mathcal{M}(L^{123}(X)))$ for a generic space $X \leq K^4$.}\label{figure: generic lattice for L^123(X)}

\end{center}
\end{figure}
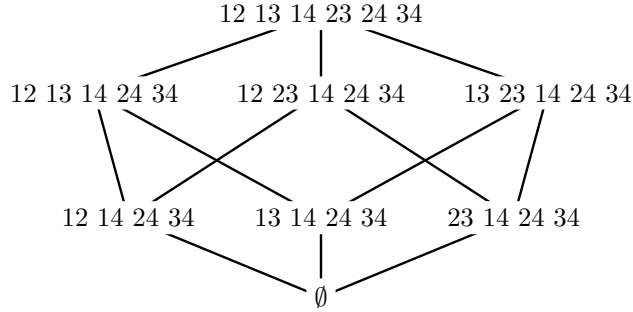
    
    
    The link of this fan looks as in Figure \ref{fig-Bergman-S=12} (except that the loops need to be removed in the labels of the rays). 
    The fan itself is just a plane $\mathbb{R}^2$, in the boundary of 
    $\barRR^6/\RR$ where $p_{14} = \infty$, $p_{24} = \infty$ and $p_{34} = \infty$. 
    Any tropical line parametrized by a Pl\"ucker vector in this boundary satisfies $x_4 = \infty$. 
    The Bergman fan parametrizes the position of the vertex of the tropical line in this boundary plane. 
    When we project with $S=\{1,2,3\}$, all coordinates are $\infty$, so the projection is contained in the (affine version) of the tropical plane.
\end{example}

\subsubsection{General case} \mbox{ } \\

Even if $L^S(X)$ does not have trivial coefficients, we can still describe its tropicalization in terms of its associated valuated matroid. 
Indeed, its tropical Plücker vector is obtained as valuation of its classical Plücker vector, which we already described in Section 3.
We next describe its valuated circuits (see Equation \eqref{equation: tropicalized circuits}).
%

We write $x_h$ for the coordinate $x_{[n] \setminus h}$ of $\TT\PP^{\binom{n}{d}-1}$, when $h$ is a hyperedge of size $n-d$.
\begin{proposition}\label{proposition: valuated circuits of L^S(X)}
    Let $X \leq K^n$ be of dimension $d+1$. 
    Then each valuated circuit of $L(X)$ comes from a circuit $C$ of $\text{D}_{n-d}(U_{n,n})$ via the formula
    \begin{equation}\label{equation: valuated circuits of L(X)}
        l_C = \bigoplus_{h \in C} ~ \val \big(
            \sum_{\mathcal{O}} \varepsilon_{\mathcal{O}}\prod_{ \substack{e \in C \setminus h \\ e \neq e_h}} q_{[n] \setminus s_{\mathcal{O}}(e)}
        \big) \odot x_h.
    \end{equation}
    Here the sums run over orientations $\mathcal{O}$ of the hypergraph with edges $C\setminus h$ and roots $I_h$, where $e_h \in C \setminus h$ and $I_h \subset e_h$ of size $n-d-1$ can be chosen arbitrary for each $h \in C$.
\end{proposition}

We emphasize that the coefficients in \eqref{equation: valuated circuits of L(X)} need not be finite, unless $X$ is generic.

\begin{proof}
    Let $w = (w_H)_H$, with $H \subseteq \binom{[n]}{n-d}$ of size $d+1$, denote the tropical Plücker vector of $L(X)$. Furthermore take a subset $I \subseteq \binom{[n]}{n-d}$ consisting of $d+2$ hyperedges and consider the valuated circuit
    \begin{equation*}
        c_I = \bigoplus_{h \in I} w_{I \setminus h} \odot x_h.
    \end{equation*}
    The coefficient $w_{I \setminus h}$ has a chance at being finite only if $I \setminus h$ is a basis of $\text{D}_{n-d}(U_{n,n})$, as the latter has the maximal amount of bases among all $\mathcal{M}(L(X))$. Thus we need $I$ to contain exactly one circuit $C$ of $\text{D}_{n-d}(U_{n,n})$ or $c_I$ will be identically $\infty$.
    With this assumption in place $I \setminus h$ is basis exactly when $h \in C$ and $c_I$ reduces to
    \begin{equation*}
        c_I = \bigoplus_{h \in C} w_{I \setminus h} \odot x_h.
    \end{equation*}
    Note $|C| \geq 3$. Given $h \in C$ we apply Proposition \ref{theorem: Plücker coords of L^S(X)} to write 
    $$ 
    w_{I \setminus h} = \val \big(
            \sum_{\mathcal{O}} \varepsilon_{\mathcal{O}}\prod_{ \substack{e \in I \setminus h \\ e \neq e_h}} 
            q_{[n] \setminus s_{\mathcal{O}}(e)} \big), 
    $$
    where $e_h \in C \setminus h$, $I_h \subset e_h$ is of size $n-d-1$ and the sum runs over all orientations of $I \setminus h$ with roots $I_h$. It is an easy exercise to show that $C$ being a circuit of $\text{D}_{n-d}(U_{n,n})$ implies
    $$ |\bigcup_{e \in C} e \hspace{.5mm}| = |\bigcup_{e \in C \setminus h} e \hspace{.5mm}| = n-d-1+|C \setminus h|.$$ 
    Essentially for cardinality reasons, any orientation $\mathcal{O}$ of $I \setminus h$ with roots $I_h$ restricts to two smaller orientations $\mathcal{O}_1$ of $C \setminus h$ and $\mathcal{O}_2$ of $I \setminus C$ 
    s.t. $r(\mathcal{O}_1) = I_h$ and $r(\mathcal{O}_2) = ( \bigcup_{e \in C} e ) \cap ( \bigcup_{e \in I \setminus C} e )$. 
    On the other hand, any two such orientations combine to an orientation $\mathcal{O}$ of $I \setminus h$ with roots $I_h$.
    Clearly the numbers of increasing edges just sum up when combining $\mathcal{O}_1$ and $\mathcal{O}_2$.
    To relate the different target maps, list $I \setminus h = \{h_1, \ldots, h_{d+1}\}$ in reverse lexicographical order. 
    Likewise, write 
    $C \setminus h = \{h_1', \ldots, h_r'\}$ and $I \setminus C = \{h_{1}'', \ldots, h_{d+1-r}''\}$ both in reverse lexicographical order. 
    Then there is a permutation $\sigma \in S_{d+1}$ such that $h_i' = h_{\sigma(i)}$ for $i \in [r]$ and $h_j'' = h_{\sigma(j+r)}$ for $j \in [d+1-r]$. In other words,
    $$
    (\pi_\mathcal{O} \circ \sigma)(i) = \begin{cases}
        t_\mathcal{O}(h_i') = \pi_{\mathcal{O}_1}(i) & i \leq r \\
        t_\mathcal{O}(h_{i-r}'') = \pi_{\mathcal{O}_2}(i-r)& i > r.
    \end{cases}
    $$ 
    We must have 
    $ \textnormal{im } \pi_{\mathcal{O}_1} = (\bigcup_{e \in C} e) \setminus I_h$ and 
    $ \textnormal{im } \pi_{\mathcal{O}_2} = [n] \setminus (\bigcup_{e \in C} e)$ regardless of how exactly the orientations look like.
    Let $\tau$ be the unique permutation of $[n] \setminus I_h$ such that $\tau(x) < \tau(y)$ whenever $x \in \textnormal{im } \pi_{\mathcal{O}_1}, y \in \textnormal{im } \pi_{\mathcal{O}_2}$ and whose restriction to both images is order preserving.
    With this setup one can count inversions to see
    $$ 
    \sign(\tau \circ \pi_{\mathcal{O}} \circ \sigma) = 
    \sign(\pi_{\mathcal{O}_1})\sign(\pi_{\mathcal{O}_2}).
    $$
    Putting everything together we get 
    \begin{align*}
        w_{I \setminus h} 
        =&  \val \big(
            \sum_{\mathcal{O}} \varepsilon_{\mathcal{O}} \prod_{ \substack{e \in I \setminus h \\ e \neq e_h}} 
            q_{[n] \setminus s_{\mathcal{O}}(e)} 
            \big) \\
        =&  \val \big(
            \sum_{\mathcal{O}_1} \varepsilon_{\mathcal{O}_1} \prod_{ \substack{e \in C \setminus h \\ e \neq e_h}} 
            q_{[n] \setminus s_{\mathcal{O}_1}(e)} 
            \big) 
            \odot
            \val \big(
            \sum_{\mathcal{O}_2} \varepsilon_{\mathcal{O}_2} \prod_{e \in I \setminus C} 
            q_{[n] \setminus s_{\mathcal{O}_2}(e)} 
            \big), \label{equation: }
    \end{align*}
    where in the second equality we made use of $\val(a) = \val(-a)$ for any $a \in K$. Note the second factor does not depend on $h \in C$, so $c_I$ is a tropical multiple of the tropical polynomial in \eqref{equation: valuated circuits of L(X)}. On the other hand, suppose we start with a circuit $C$ of $\text{D}_{n-d}(U_{n,n})$ such that \eqref{equation: valuated circuits of L(X)} is not identically $\infty$. This implies $q_{[n] \setminus A} \neq 0$ for some $A \subseteq \bigcup_{e \in C} e$ of size $n-d-1$. Take such $A$ and consider $I = C \cup \{Av : v \in [n] \setminus (\bigcup_{e \in C} e)\}$. Then $c_I$ arises from \eqref{equation: valuated circuits of L(X)} by scaling with a nonnegative power of $\val( q_{[n] \setminus A})$, which is finite. Hence all valuated circuits of $L(X)$ are of the claimed form after some scaling.
\end{proof}

%

To obtain the valuated circuits of $L^S(X)$, we use Proposition \ref{proposition: matroid of L^S(X)} and replace $\text{D}_{n-d}(U_{n,n})$ with $\mathcal{M} = \mathcal{M}(L^S(X))$ for generic $X$. 
The circuit $C$ in the above proof must fall into a summand of $\mathcal{M}$. 
So $C$ is either a circuit in $\text{D}_{n-d}(U_{n,n})_{|E_S}$,
a loop, or lives in one of the rank one summands of $\mathcal{M}$. 
The latter case gives valuated circuits of the form
$$ \val(q_{Ja})  \odot x_J \oplus \val(q_{J'a}) \odot x_{J'} $$ where $J,J' \in E_{S,a}, a \in S$ such that $q_{Ja}, q_{J'a} \neq 0$, and it is clear what to do with the loops.
Note that $D_{n-d}(U_{[n] \setminus S})_{|E_S} = D_{n-d}(U_{[n] \setminus S})$, where $U_{[n] \setminus S}$ denotes the free matroid on $[n]\setminus S$.
Circuits of $D_{n-d}(U_{[n] \setminus S})$ can be interpreted as hypergraphs on $[n]\setminus S$ and give valuated circuits according to \eqref{equation: valuated circuits of L(X)}. 

In cocodim 1, any circuit $C$ is given by a cycle in $K_n$. After taking away an edge $e$ there is only one orientation of $C\setminus e$ with a given root. Choosing this root to be an endpoint of $e$ easily leads to the following corollary:

\begin{corollary}\label{corollary: valuated circuits of L^S(X) in cocodim 1}
    Let $X \leq K^n$ be of codimension $1$. Then up to tropical scaling the valuated circuits of $L(X)$ are given by
    \begin{equation}\label{equation: valuated circuits of L(X) in cocodim 1}
        l_C = \bigoplus_{e \in E(C)} \bigodot_{v \in V(C) \setminus e} \val(q_{[n] \setminus v}) \odot x_e,
    \end{equation}
    where $C$ is a (vertex disjoint) cycle in $K_n$ with vertices $V(C)$ and edges $E(C)$.
\end{corollary}

\begin{remark}
    We note that in cocodim 1, we do not need to consider arbitrary cycles in order to cut out $\trop L(X)$. 
    Indeed, any cycle $C$ of length greater than $3$ may be divided into two smaller cycles $C_1$ and $C_2$ by adding a singe new edge. 
    In this case, one can show $\trop (l_{C_1}) \cap \trop (l_{C_2}) \subseteq \trop (l_C)$. 
    Thus it is enough to consider triangles. 
    
    With some effort it can be shown that this argument generalizes: 
    If a circuit $C$ of $\text{D}_{n-d}(U_{n,n})$ can be subdivided into smaller circuits $C_1, C_2$ it remains true that $\trop (l_{C_1}) \cap \trop (l_{C_2}) \subseteq \trop (l_C)$. 
    However it is no longer clear what the "indecomposable" circuits look like.
\end{remark}

\section{The linear degenerate short flag variety}\label{section: 4 - short flags}
In this section, we use our results on $L^S$-spaces to describe tropicalizations of linear degenerate short flag varieties.
\subsection{Covering by $L^S$-spaces in cocodim 1}\label{section: 4 - covering short flags}
We first introduce linear degenerate short flag varieties and discuss how they can be covered by $L^S$-spaces.
\begin{definition}\label{definition: flag varieties}
    Fix an increasing sequence of integers $\underline{d} = 0 < d_1 < d_2 < \ldots < d_r  < n$. The \emph{partial flag variety} $\flag_{\underline{d}}$ is the set of all \emph{partial flags} 
    $$V_{d_1} \leq \ldots \leq V_{d_r} ,$$ 
    where each $V_{d_i} \leq K^n$ is a subspace of dimension $d_i$.
    For $0 < k < n-1$, the special case of $ \underline{d} = (k,k+1)$ is denoted by $\shortFlag(k)$ and called a \emph{short flag variety}. In case $k = n-2$ we drop $k$ from our notation.
    
    Replace the conditions $V_{d_i} \leq V_{d_{i+1}}$  by $f_i(V_{d_i}) \leq V_{d_{i+1}}$, where $f_i: K^n \to K^n$ is a linear map for any $1 \leq i \leq r-1$. Doing so leads to the \emph{linear degenerate partial flag variety} $\flag_{\underline{d}}^f$ \cite{CFFFM19}, where we write $f = (f_1, \ldots, f_{r-1})$. We are especially interested in \emph{$R$-linearly degenerate short flag varieties}, in which the single map $f=f_1$ is the projection $\text{pr}_R$, with $R \subseteq [n]$.
    This case is denoted by $\shortFlag^R(k)$ or just $\shortFlag^R$ when $k = n-2$.
\end{definition}

We always view linear degenerate flag varieties as subvarieties in a suitable product of Grassmannians. 
Clearly any linearly degenerated short flag variety is isomorphic to $\shortFlag^R(k)$, for some $R \subseteq [n]$ and suitable $k$. \\

The goal of this section is to cover the linear degenerate short flag varieties $\shortFlag^R$ by $L^S$-spaces, so that we can exploit our results from earlier sections to describe the cones contained in $\trop \shortFlag^R$.
In what follows we will usually suppress the notation of residue classes. It should always be clear where elements live. \\

The natural action of the torus $(K^*)^n$ on $K^n$ induces an action of $T =(K^*)^n/K^*$ on $\Gr(r,n)$ for any $r$. Given $Y \in \Gr(r,n)$ and $t \in T$, we can express the action of $t$ on $Y$ easily in terms of its Plücker coordinates. For any $J \in \binom{[n]}{r}$ we have $p_J(tY) = t_Jp_J(Y)$ where $t_J = \prod_{j \in J}t_j$. Note that the map
\begin{equation}\label{equation: torus action}
    \phi_{n-r}: T \to (K^*)^{\binom{[n]}{r}}/(K^*) ~;~ t \mapsto (t_J)_J
\end{equation}
is a monomorphism of algebraic groups giving rise to an action of $T$ on the whole projective space $\PP^{\binom{[n]}{r}-1}$. 
%
It is sometimes useful to scale the $t_J$ by $t_{[n]}^{-1}$ so that $t_J = \prod_{j \notin J}t_j^{-1}$. This in particular shows that $\phi_1$ is an isomorphism. 

Take again $X \leq K^n$ a subspace of dimension $d+1$ and let $t \in T$. 
Clearly we have $L(tX) = tL(X)$. More generally, for any $S \subseteq [n]$ one obtains $L^S(tX) = tL^S(X)$, because projections commute with the torus action on $K^n$.

We once again restrict our attention to cocodim 1 and choose $X$ generic, with trivial coefficients. For example $X$ with constant Plücker vector $q \equiv 1$ works. Note that any other subspace $X' \in \Gr(n-1,n)$ is of the form $tX^S$ for some $t \in T$ and $S \subset [n]$.
Thus we can cover the short flag variety using $L^S$-spaces:
\begin{equation*}
    \shortFlag = \bigcup_{S \subset [n]} \bigcup_{t \in T} tL^S(X) \times tq^S.
\end{equation*}
Recall that $q^S_I = q_I$ if $S \subseteq I$ and $q^S_I = 0$ otherwise. This covering is a somewhat wasteful in the sense that for nonempty $S$ the stabilizer of $q^S$ is non trivial as well. The quotient of such a stabilizer is isomorphic to the smaller torus
\begin{equation}\label{equation: torus subgroups}
    T_S = \{t \in T : t_i = t_j ~\text{ for all }~ i,j \in S\},
\end{equation}
which acts freely on the orbit $Tq^S$. Then
\begin{equation}\label{equation: covering short flag}
    \shortFlag = \bigcup_{S \subset [n]} \bigcup_{t \in T_S} tL^S(X) \times tq^S.
\end{equation}
Now let $R \subseteq [n]$ and consider the $R$-linearly degenerate short flag variety $\shortFlag^R$.
We can cover it just as the usual short flag variety leading to
\begin{equation*}
    \shortFlag^R = \bigcup_{S \subset [n]} \bigcup_{t \in T_S} tL^{R \cup S}(X) \times tq^S.
\end{equation*}
Applying Theorem \ref{theorem: L^S(X) = L(X^S) in cocodim 1}, we know $L^{R \cup S}(X)$ is linear if $R\cup S \subset [n]$ and $L^{R \cup S}(X) = \Gr(n-2,n)$ whenever $R\cup S = [n]$. Thus
\begin{equation}\label{equation: covering degenerate short flag}
    \shortFlag^R = 
    \bigcup_{R^C \subseteq S \subset [n]} ~ \bigcup_{t \in T_S} \Gr(n-2,n) \times tq^S \cup 
    \bigcup_{R^C \nsubseteq S \subset [n]} ~ \bigcup_{t \in T_S} tL^{R \cup S}(X) \times tq^S.
\end{equation}

\subsection{Tropicalizing the linear degenerate short flag variety in cocodim 1}\label{section: 4 - trop short flags}
Just as in the previous subsection, we suppress the notation of residue classes. \\

On the tropical side, the torus actions in the previous subsection can be seen as translation. More precisely the embedding $\phi_{n-r}:T \hookrightarrow (K^*)^{\binom{[n]}{r}}/K^*$ from equation \eqref{equation: torus action} tropicalizes to
\begin{equation}
    \psi_{n-r}: \RR^n/\RR \hookrightarrow \RR^{\binom{[n]}{r}}/\RR; s \mapsto (~ \sum_{j \in J}s_j~)_J,
\end{equation}
thus giving us an action of $\RR^n/\RR$ on the tropical projective space $\TT\PP^{\binom{[n]}{r}-1}$. Clearly we have $$\val \circ~ \phi_{n-r} = \psi_{n-r} \circ \val.$$ 
Now for $S \subset [n]$ the closed subgroup $T_S$ tropicalizes to the subspace $\langle e_i ~:~ i \notin S \rangle_\RR$ of $\RR^n/\RR$, and thanks to \cite[Corollary 3.2.13]{MS15} 
$$
\trop ((\phi_2 \times \phi_1)(T_S)) = (\psi_2 \times \psi_1)( \trop( T_S)).
$$
Explicitly $\trop ((\phi_2 \times \phi_1)(T_S))$ is the image of the subspace
$$
A_S = \langle a_i \times e_{[n] \setminus i} ~:~ i \notin S \rangle_\RR 
$$ 
in $\RR^{\binom{[n]}{n-2}}/\RR \times \RR^{\binom{[n]}{n-1}}/\RR$, which we just denote as $A_S$ again. 
Here $(a_i)_J = |\{i\} \setminus J| = |\{i\} \cap J^C|$
for $J \in \binom{[n]}{n-2}$, and the $e_{[n] \setminus i}$ denote standard basis vectors in $\RR^{\binom{[n]}{n-1}}$.

\begin{theorem}\label{theorem: trop (degen) short flag}
    Let $X \in \Gr(n-1,n)$ be generic and with trivial valuations.
    Then the tropicalization of the short flag variety can be covered using tropicalizations of $L^S$-spaces of $X$. More precisely, we have the disjoint union
    \begin{equation*}
        \trop( \shortFlag) = \bigcup_{S \subset [n]} A_S +  \trop( L^S(X)) \times w^S,
    \end{equation*}
    where $w^S = \val(q^S) = \infty \cdot \sum_{i \in S} e_{[n] \setminus i}$. 
    Similarly, $\shortFlag^R$ tropicalizes to the disjoint union \vspace{.5ex}
    \begin{align*}
        \trop( \shortFlag^R) = 
        &\bigcup_{R^C \subseteq S \subset [n]} ~ \trop (\Gr(n-2,n)) \times (\RR^{\binom{[n]}{n-1}}/\RR + w^S) ~~ \cup \\
        &\bigcup_{R^C \nsubseteq S \subset [n]} ~ A_S + \trop(L^{R \cup S}(X)) \times w^S.
    \end{align*}
    
\end{theorem}

\begin{proof}
    We want to apply Theorem \ref{theorem: stratification of trop X}, so we need to consider the intersections of $\shortFlag$ with torus orbits of $\PP^{\binom{[n]}{n-2}-1} \times \PP^{\binom{[n]}{n-1}-1}$. Such an intersection can only be nonempty for torus orbits of the form $\mathcal{O}_F \times \mathcal{O}_S$, where
    $$
    \mathcal{O}_S = \{\mu \in \mathbb{P}^{\binom{[n]}{n-1}} : \mu_{[n] \setminus i} = 0 \text{ if and only if } i \in S\}, 
    $$ 
    with $S \subset [n]$ and
    $$
    \mathcal{O}_F = \{\nu \in \mathbb{P}^{\binom{[n]}{n-2}} : \nu_{I} = 0 \text{ if and only if } I \in F\}
    $$ 
    for a proper flat $F$ in the linear matroid $\mathcal{M}(L^S(X))$ (see the discussion following Theorem \ref{theorem: very affine trop X is bergman fan}).
    Using \eqref{equation: covering short flag} we have
    $$ 
    \shortFlag \cap \mathcal{O}_F \times \mathcal{O}_S = \bigcup_{t \in T_S} tL^S(X)\cap \mathcal{O_F} \times tq^S = T_S( L^S(X)\cap \mathcal{O_F} \times q^S ) ,
    $$
    where $T_S$ acts via $\phi_2 \times \phi_1$ as in \eqref{equation: torus action}.
    To apply Theorem \ref{theorem: stratification of trop X} we need to project away zero coordinates of $ \shortFlag \cap \mathcal{O}_F \times \mathcal{O}_S$ and view $\mathcal{O}_F \times \mathcal{O}_S$ as a torus. Let $\pi$ be this projection, then 
    \begin{equation*}
        \shortFlag \cap \mathcal{O}_F \times \mathcal{O}_S = 
        \pi(\phi_2 \times \phi_1(T_S))((L^S(X)\cap \mathcal{O_F}) \times \pi(q^S)).
    \end{equation*} 
    We claim that taking valuations and passing to the closure on both sides gives
    $$ 
    \trop(\shortFlag \cap \mathcal{O}_F \times \mathcal{O}_S) = 
    \pi(A_S)+\trop(L^S(X)\cap \mathcal{O_F}) \times \pi(w^S),
    $$ 
    where we slightly abuse notation and just denote the tropicalization of $\pi$ by $\pi$ again. For the right hand side
    the inclusion "$\supseteq$" easily follows, because the closure of a Minkowski sum of two sets contains the Minkowski sum of their closures. 
    To achieve equality we need to see that the right hand side is already closed. However exploiting the Bergman fan structure of $\trop(L^S(X)\cap \mathcal{O_F})$ the right hand side is a finite union of cones and thus closed. 
    
    Putting all pieces $\trop(\shortFlag \cap \mathcal{O}_F \times \mathcal{O}_S)$ together we conclude that $\trop(\shortFlag)$ is as claimed. The proof for the degenerate case is similar.
\end{proof}

Our next goal is to derive a Bergman fan like structure on $\trop( \shortFlag)$. By Theorem \ref{theorem: trop (degen) short flag} $\trop(\shortFlag)$ is the union of cones $A_S + C \times w^S$, where $S$ varies over proper subsets of $[n]$ and $C$ is a cone in the extended Bergman fan of $\mathcal{M}(L^S(X))$. That is, $C$ is of the form
\begin{equation*}
    C_{F_\bullet} = \infty \cdot e_{F_0} + \cone(e_{F_1}, \ldots, e_{F_s}),
\end{equation*} 
living inside $\RR^{\binom{[n]}{n-2}}/\RR + \infty \cdot e_{F_0}$,
where $F_\bullet$ is a chain of flats in $\mathcal{L}(\mathcal{M}(L^S(X)))$. However, different choices of $F_\bullet$ may lead to the same result.

\begin{example}\label{example: different chains give same cone}
    Suppose $n = 4$ and $X \leq K^n$ as in Theorem \ref{theorem: trop (degen) short flag}. We identify $ \binom{[4]}{2}$ with $[6]$, and $\binom{[4]}{3}$ with $[4]$, using the lexicographical order. Let $S = \{1\}$. Taking the identifications into account, the lattice of flats for $\mathcal{M}(L^S(X))$ is depicted in Figure \ref{figure: generic lattice for L^1(X) relabeled}. Now consider the cone 
    $$
    C = A_1 + (\infty\cdot e_\emptyset + \cone(e_1)) \times w^1 = A_1 + \cone(e_1) \times (0,0,0, \infty)
    $$ 
    coming from the chain $\emptyset \subset 1$. This cone lives in the stratum 
    $$ 
    \RR^5 \times \RR^2 \cong \RR^6/\RR \times (\RR^3\times \infty)/\RR \subset \TT\PP^5 \times \TT\PP^3.
    $$
    %
    Note that $C$ in particular contains $\psi_2 \times \psi_1(-e_2-e_3-e_4) + 0 \times w^1$ which up to tropical scaling (in each factor) is just $e_{123} \times w^1$. This shows that we do not get a larger cone if we pass to the longer chain $\emptyset \subset 1 \subset 123$.
\end{example}

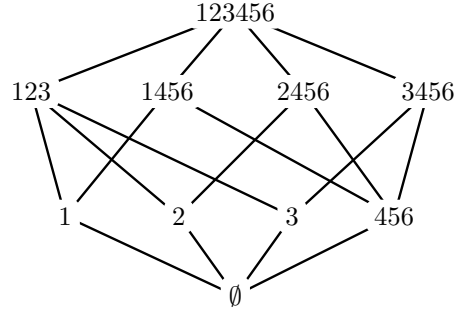
\begin{figure}[ht]
\begin{center}

\begin{tikzpicture}[scale = 1.5,
                    edge/.style={black, line width=0.9pt, line cap=round},
                    color = {black}]
  

\tikzstyle{node}=[text=black, inner sep=3pt, rectangle, rounded corners=3pt,fill=white, draw=none]

\coordinate (H123) at (0.7,.9);
\coordinate (H1456) at (1.9,.9);
\coordinate (H2456) at (3.1,.9);
\coordinate (H3456) at (4.2,.9);

\coordinate (v1) at (1,-.2);
\coordinate (v2) at (2,-.2);
\coordinate (v3) at (3,-.2);
\coordinate (v456) at (3.9,-.2);
\coordinate (b) at (2.5,-0.9);
\coordinate (t) at (2.5,1.6);

\foreach \i/\k in {123/1, 123/2,123/3,1456/1,1456/456,2456/2,2456/456,3456/3,3456/456} {
   \draw[edge] (H\i) -- (v\k);
   }

\foreach \i in {123, 1456,2456,3456} {
   \draw[edge] (H\i) -- (t);
  }
  
\foreach \k in {1,2,3,456} {
   \draw[edge] (v\k) -- (b);
  }

\node[node]  at (b) {\small $\emptyset$};

 \node[node] at (H123) {\small $123$};
 \node[node] at (H1456) {\small $1456$};
 \node[node] at (H2456) {\small $2456$};
 \node[node] at (H3456) {\small $3456$};
 
\foreach \k in {1, 2, 3, 456} {
   \node[node] at (v\k) {\small $\k$};
  }

\node[node] at (t) {\small $123456$};

\end{tikzpicture}

\caption{Relabeled lattice of flats $\mathcal{L}(\mathcal{M}(L^1(X)))$ for a generic space $X \leq K^4$.}\label{figure: generic lattice for L^1(X) relabeled}

\end{center}
\end{figure}

Recall that by Proposition \ref{proposition: matroid of L^S(X) in cocodim 1} we have the decomposition 
$$
\mathcal{M}(L^S(X)) = \mathcal{M}_S \oplus \mathcal{M}_{S,0} \oplus \bigoplus_{a \in S} \mathcal{M}_{S,a}.
$$
The summands are matroids on the subsets of $\binom{[n]}{n-2}$ defined by $E_S = \{ J : S \subseteq J \}, E_{S,0} = \{J : | S \setminus J| = 2\}$ and $E_{S,a} = \{J : S \setminus J = a \}$ for $a \in S$. In the above example, we identified $\{1,2,3\}$ with $\{12, 13, 14\}$ which is precisely $E_S$ when $S = \{1\}$. 

\begin{lemma}\label{lemma: A_S up to infinities}
    Let $S \subset [n]$ and $E_{S,0} \subseteq F \subset \binom{[n]}{n-2}$. Then $A_S + (\infty \cdot e_F) \times w^S$ contains $(e_{E_S} + \infty \cdot e_F) \times w^S$ and $(e_{\widehat{E}_S} + \infty \cdot e_F) \times w^S$, where $\widehat{E}_S = \bigcup_{a \in S} E_{S,a}$.
\end{lemma}

\begin{proof}
    $A_S + (\infty \cdot e_F) \times w^S$ certainly contains the element 
    $$
    \sum_{i \notin S} (a_i \times e_{[n] \setminus i}) + (\infty \cdot e_F) \times w^S =  \psi_2 \times \psi_1(-\sum_{i \notin S} e_i) + (\infty \cdot e_F) \times w^S, 
    $$
    where $a_i = (-|\{i\} \cap J|)_J$, with $J$ running over $\binom{[n]}{n-2}$.
    Clearly 
    $$
    \psi_1(-\sum_{i \notin S} e_i)+w^S = (1, \ldots, 1) + w^S = w^S.
    $$ 
    Moreover 
    $$ 
    \psi_2(-\sum_{i \notin S} e_i) = \psi_2(\sum_{i \in S} e_i) = -\sum_{i \in S} a_i.
    $$  
    Now for $i \in S$
    $$ -a_i + \infty \cdot e_F = e_{E_S}+ \sum_{a \neq i} e_{E_{S,a}} + \infty \cdot e_F,$$
    hence
    \begin{align*}
        \psi_2(-\sum_{i \notin S} e_i) + \infty \cdot e_F 
        &= |S|e_{E_S} + (|S|-1)\sum_{a \in S}e_{E_{S,a}} +\infty \cdot e_F\\
        &= e_{E_S} +\infty \cdot e_F.
    \end{align*}
    This proves the first claim from which the second immediately follows.
\end{proof}

We now solve the problem illustrated in Example \ref{example: different chains give same cone} by restricting our attention to a smaller subset of cones. The cones in the strata corresponding to fixed $S$ will no longer be controlled by an arbitrary chain of flats in the matroid of $L^S(X)$. Instead they come from two smaller chains of flats in the matroids $\mathcal{M}_S$ and $\widehat{ \mathcal{M}}_S = \bigoplus_{a \in S} \mathcal{M}_{S,a}$. Then $\widehat{E}_S = \bigcup_{a \in S}E_{S,a} = \{J : |S \setminus J|=1 \}$ is the ground set of $\widehat{\mathcal{M}}_S$.

Given two chains of flats
\begin{align*}
    F_\bullet :&~ F_0 \subset \ldots \subset F_r \subset E_S, \\
    \widehat{F}_\bullet :&~ \widehat{F}_0 \subset \ldots \subset \widehat{F}_s \subset \widehat{E}_S
\end{align*}
in $\mathcal{M}_S$ and $\widehat{ \mathcal{M}}_S$ respectively, we define the cone
$$
    C_{F_\bullet,\widehat{F}_\bullet} = \infty \cdot (e_{E_{S,0}}+e_{F_0}+e_{\widehat{F}_0}) + \cone(e_{F_1}, \ldots, e_{F_r}, e_{\widehat{F}_1}, \ldots, e_{\widehat{F}_s}) + \RR(1, \ldots, 1).
$$
For technical reasons we also need to allow the \emph{exceptional} chains $F_\bullet: F_0 = E_S$ and $\widehat{F}_\bullet: \widehat{F}_0 = \widehat{E}_S$, but always require at least one of $F_\bullet$, $\widehat{F}_\bullet$ to be not of this form. 

We define the \emph{length} of $F_\bullet$ to be the number of flats occurring in the chain minus $2$. In particular the length of exceptional chains is $-1$. 
The latter is important to make dimensions of cones well behaved.

\begin{theorem}\label{theorem: bergman fan for trop short flags}
    The tropicalization of the short flag variety $\trop(\shortFlag)$ equals the union of cones 
    $$ C_{F_\bullet,\widehat{F}_\bullet}^S = A_S + C_{F_\bullet,\widehat{F}_\bullet} \times w^S ,$$ where $F_\bullet$ and $\widehat{F}_\bullet$ are (not both exceptional) chains of flats in $\mathcal{M_S}$ and $\widehat{\mathcal{M}}_S$ respectively and $S$ varies over proper subsets of $[n]$. These cones are in one to one correspondence to pairs of chains as above and the cones in a non empty stratum of $\trop(\shortFlag)$ form a fan.
\end{theorem}

\begin{proof}
    Fix an arbitrary proper subset $S \subset [n]$. Given chains of flats \begin{align*}
    F_\bullet :&~ F_0 \subset \ldots \subset F_r \subset E_S \\
    \widehat{F}_\bullet :&~ \widehat{F}_0 \subset \ldots \subset \widehat{F}_s \subset \widehat{E}_S,
\end{align*} in $\mathcal{M}_S$ and $\widehat{ \mathcal{M}}_S$, we define a third chain $G_\bullet$ via 
$$
G_i = \begin{cases}
    E_{S,0} \cup F_i \cup \widehat{F}_0 & \text{ if } i \leq r, \\
    E_{S,0} \cup E_S \cup \widehat{F}_{i-r} & \text{ if } r < i \leq r+s.
\end{cases}
$$
Note that $G_\bullet$ is a chain of flats in $\mathcal{M}(L^S(X))$. Using Lemma \ref{lemma: A_S up to infinities} it is readily verified that 
$$ 
A_S + C_{G_\bullet} \times w^S = A_S + C_{F_\bullet,\widehat{F}_\bullet} \times w^S .
$$ 
The special case where $F_\bullet$ is exceptional is handled the same way. If $\widehat{F}_\bullet$ is exceptional, just swap the roles of $F_\bullet$, $\widehat{F}_\bullet$ and exchange $E_S$ with $\widehat{E}_S$ when defining $G_\bullet$.
This shows the cones we are interested in lie inside $\trop(\shortFlag)$. For the other inclusion take any chain of flats $G_\bullet = G_0 \subset \ldots \subset G_k \subset [n]$ inside $\mathcal{M}(L^S(X))$. We look for chains of flats $F_\bullet$ and $\widehat{F}_\bullet$ in $\mathcal{M}_S$ and $\widehat{ \mathcal{M}}_S$ such that 
$$
A_S + C_{G_\bullet} \times w^S \subseteq A_S + C_{F_\bullet, \widehat{F}_\bullet} \times w^S.
$$
Because of Proposition \ref{proposition: matroid of L^S(X) in cocodim 1}, each $G_i$ is of the form $E_{S,0} \cup (G_i\cap E_S) \cup (G_i \cap \widehat{E}_S)$ and $G_i\cap E_S$ is a flat in $\mathcal{M}_S$ while $G_i \cap \widehat{E}_S$ is a flat in $\widehat{\mathcal{M}}_S$. Collecting all $G_i\cap E_S$ and $G_i \cap \widehat{E}_S$ gives chains of flats $F_\bullet$, $\widehat{F}_\bullet$ in $\mathcal{M}_S$ and $\widehat{ \mathcal{M}}_S$, of which at least one is not exceptional. Clearly $C_{G_\bullet} \subseteq C_{F_\bullet,\widehat{F}_\bullet}$ from which the desired inclusion follows. Hence $\trop(\shortFlag)$ is a union of cones as claimed.

The non empty strata of $\trop(\shortFlag)$ arise by fixing $S, F_0$ and $\widehat{F}_0$ which effectively determines the infinite coordinates. 
Given $F^1_\bullet$ and $\widehat{F}^1_\bullet$ as well as $F^2_\bullet$ and $\widehat{F}^2_\bullet$ compatible with these choices we claim 
$$ 
A_S + C_{F^1_\bullet, \widehat{F}^1_\bullet} \times w^S ~\subseteq~
A_S + C_{F^2_\bullet, \widehat{F}^2_\bullet} \times w^S
$$ 
if and only if $F^1_\bullet$ is a subchain of $F^2_\bullet$ and $\widehat{F}^1_\bullet$ is a subchain of $\widehat{F}^2_\bullet$. In particular we have a one to one correspondence between pairs of chains and cones. The if implication is trivial. For the only if direction we pass to the preimage of the given stratum in $(\barRR^{\binom{[n]}{n-2}} \setminus \infty) \times (\barRR^{\binom{[n]}{n-1}} \setminus \infty)$. This preimage can be written as 
$$
\RR^{\binom{[n]}{n-2}} \times \RR^{\binom{[n]}{n-1}} + v \times w^S, \text{ where }~ v = \infty \cdot (e_{E_{S,0}} + e_{F_0} + e_{\widehat{F}_0}).
$$
Put $\1_1 = (1, \ldots, 1) \times (0, \ldots, 0)$ and $\1_2 = (0, \ldots, 0) \times (1, \ldots, 1)$. Lemma \ref{lemma: A_S up to infinities} implies 
$$
A_S + \RR\1_1 + \RR\1_2 + v \times w^S = A_S + \RR e_{E_S} \times 0 + \RR e_{\widehat{E}_S} \times 0 + v \times w^S.
$$
Hence our hypothesis is equivalent to
$$
A_S + \cone(e_{F_1^1}, \ldots, e_{F_r^1}, e_{\widehat{F}^1_1}, \ldots, e_{\widehat{F}^1_s}) \times 0 + \RR e_{E_S} \times 0 + \RR e_{\widehat{E}_S} \times 0 + v \times w^S
$$ 
being contained in 
$$
A_S + \cone(e_{F_1^2}, \ldots, e_{F_k^2}, e_{\widehat{F}^2_1}, \ldots, e_{\widehat{F}^2_l}) \times 0 + \RR e_{E_S} \times 0 + \RR e_{\widehat{E}_S} \times 0 + v \times w^S.
$$
Since $A_S = \langle a_i \times e_{[n] \setminus i} : i \notin S \rangle_\RR$ and 
$w^S = \infty \cdot \sum_{i \in S}e_{[n] \setminus i}$, this implies 
$$
\cone(e_{F_1^1}, \ldots, e_{F_r^1}, e_{\widehat{F}^1_1}, \ldots, e_{\widehat{F}^1_s}) + \RR e_{E_S} + \RR e_{\widehat{E}_S} + v
$$
is a subset of
$$
\cone(e_{F_1^2}, \ldots, e_{F_k^2}, e_{\widehat{F}^2_1}, \ldots, e_{\widehat{F}^2_l}) + \RR e_{E_S} + \RR e_{\widehat{E}_S} + v.
$$
Now the claim follows just like for usual Bergman fans.

We leave it to the reader to check that the collection of all cones in our fixed stratum indeed forms a fan. More precisely faces of cones correspond to pairs of subchains while intersections of cones correspond to taking pairs of greatest common subchains. The lineality space of this fan is $A_S + v \times w^S$.
\end{proof}

The dimension of a cone $C_{F_\bullet,\widehat{F}_\bullet}^S$ coming from a pair of chains $F_\bullet$, $\widehat{F}_\bullet$ in $\mathcal{M}_S$ and $\widehat{\mathcal{M}}_S$ can be read off as the dimension of $A_S$ plus the lengths of $F_\bullet$ and $\widehat{F}_\bullet$. Indeed, passing to the preimage in 
$$
\RR^{\binom{[n]}{n-2}} \times \RR^{\binom{[n]}{n-1}} + v \times w^S, \text{ with }~ v = \infty \cdot (e_{E_{S,0}} + e_{F_0} + e_{\widehat{F}_0}),
$$
the cone becomes 
$$
A_S + \cone(e_{F_1}, \ldots, e_{F_r}, e_{\widehat{F}_1}, \ldots, e_{\widehat{F}_s}) \times 0 + \RR e_{E_S} \times 0 + \RR e_{\widehat{E}_S} \times 0 + v \times w^S.
$$
If one of the chains is exceptional, then either $e_{E_S}$ or $e_{\widehat{E}_S}$ is completely covered by infinities. 
Other than that the generators are linearly independent. Accounting for the two projective degrees of freedom, and recalling that exceptional chains are of of length $-1$, we obtain 
$$
\dim C_{F_\bullet,\widehat{F}_\bullet}^S = n-|S| + \text{length}(F_\bullet) + \text{length}(\widehat{F}_\bullet).
$$

As expected, the cones of maximal dimension sit in the big open stratum where all coordinates are finite. Note $\mathcal{M}_\emptyset \cong \mathcal{M}(K_n)$ and $\widehat{\mathcal{M}}_\emptyset = \emptyset$ by Theorem \ref{theorem: matroid of L(X)}. The only chain in $\widehat{\mathcal{M}}_\emptyset$ is the exceptional one, while a maximal chain in $\mathcal{M}_\emptyset$ has length $n-2$. We conclude the dimension of a maximal cone in $\trop(\shortFlag)$ is $2n-3$. This of course coincides with the dimension of $\shortFlag$ as variety.

\begin{remark}
    Replacing cones with their closure in $\TT \PP^{\binom{[n]}{n-2}} \times \TT \PP^{\binom{[n]}{n-1}}$ 
yields abstract fan structures inside the collection of strata for fixed $S$, but otherwise unrestricted chains.
That is inside the collection of strata $\trop (\mathcal{O}_F \times \mathcal{O}_S)$, where $\mathcal{O}_F$ are $\mathcal{O}_S$ defined as in Theorem \ref{theorem: trop (degen) short flag} and $F$ runs over the flats in $\mathcal{M}(L^S(X))$. \\
\end{remark}

Next, we have a closer look at $\trop( \shortFlag^R)$. Recall that by Theorem \ref{theorem: trop (degen) short flag} 
\begin{align*}
    \trop( \shortFlag^R) = 
    &\bigcup_{R^C \subseteq S \subset [n]} ~ \trop (\Gr(n-2,n)) \times (\RR^{\binom{[n]}{n-1}}/\RR + w^S) ~~ \cup \\
    &\bigcup_{R^C \nsubseteq S \subset [n]} ~ A_S + \trop(L^{R \cup S}(X)) \times w^S.
\end{align*}

Let $\Gr^0(n-2,n) \cong \Gr^0(2,n)$ denote the open locus where none of the Plücker coordinates vanish. It is well known, that the fan structure of $\trop \Gr^0(n-2,n)$ is controlled by the combinatorics of phylogenetic trees \cite[4.3]{MS15}, \cite{PS05}. These are trees with $n$ labeled leaves and no vertices of degree $2$. 
For any rank $2$ matroid $M$ on $[n]$, let $\Gr_M$ denote the \emph{realization space} of $M$, i.e. the locus of $X \in \Gr(2,n)$ such that the the linear matroid associated to $X$ is $M$. In this language $\Gr^0(n-2,n) = \Gr_{U_{2,n}}$.
The strata of $\trop \Gr(2,n)$ are naturally of the form $\trop \Gr_M$ and in \cite{Cueto20}, it is shown that the cones in $\trop \Gr_M$ are described by phylogenetic trees as well. Here the main difference is that now leaves need to be multilabeled. More precisely, after removing the loops of $M$, each leaf is labeled by a flat of rank one, essentially to keep track of parallels.
The main ingredient needed to understand these results appears as Proposition C.3. in \cite{Cor21} which is contributed by the author of \cite{Cueto20}.

Now let us focus on the parts $A_S + \trop(L^{R \cup S}(X)) \times w^S$ where $R^C \nsubseteq S$. We can describe the cones in this part completely analogous to the non degenerate case. For this we only need adjust the matroids $\mathcal{M}$ and $\widehat{\mathcal{M}}$. Let
\begin{align*}
    \mathcal{M}_S^R &= \mathcal{M}_{R \cup S} \oplus 
\bigoplus_{a \in R \setminus S} \mathcal{M}_{R \cup S,a} , \\ 
\widehat{\mathcal{M}}_S^R &= \bigoplus_{a \in S} \mathcal{M}_{R \cup S,a}.
\end{align*}
Denote their ground sets by $E_S^R$ and $\widehat{E}_S^R$ and observe that we have the inclusions $E_S^R \subseteq E_S, \widehat{E}_S^R \subseteq \widehat{E}_S$ and $E_{R \cup S, 0} \supseteq E_{S,0}$. Thus for any $E_{R \cup S,0} \subseteq F \subset \binom{[n]}{n-2}$ we obtain 
$$
e_{E_S} + \infty \cdot e_F = e_{E_S^R} + \infty \cdot e_F 
~\text{ and }~
e_{\widehat{E}_S} + \infty \cdot e_F = e_{\widehat{E}_S^R} + \infty \cdot e_F.
$$
Thus, in Lemma \ref{lemma: A_S up to infinities}, we can replace $E_S$ with $E_S^R$, $\widehat{E}_S$ with $\widehat{E}_S^R$ and $E_{S,0}$ with $E_{R \cup S,0}$. With the obvious modifications, the argument in the proof of Theorem \ref{theorem: bergman fan for trop short flags} immediately generalizes showing that $A_S + \trop(L^{R \cup S}(X)) \times w^S$ consists of cones 
$$
C_{F_\bullet,\widehat{F}_\bullet}^{R,S} = A_S + C_{F_\bullet,\widehat{F}_\bullet}\times w^S,
$$ 
parameterized by chains of flats $F_\bullet$ in $\mathcal{M}_S^R$ and $\widehat{F}_\bullet$ in $\widehat{\mathcal{M}}_S^R$ respectively. 

In any case every stratum of $\trop( \shortFlag^R) $ has the structure of a fan and the maximum over the dimensions of occurring cones must coincide with $\dim \shortFlag^R$.

\begin{corollary}
    Suppose $R \subseteq [n]$. The dimension of $\shortFlag^R$ is $\max(2(n-2)+ |R|-1, ~ 2n-3)$.
\end{corollary}

\begin{proof}
    We may assume $\emptyset \neq R$, as the nondegenerate case is clear.
    It is easy to see that whenever $R^C \not \subseteq S$, the parts $A_S + \trop(L^{R \cup S}(X)) \times w^S$ contain cones of maximal dimension $n-|S|+n-3$. Indeed, as in the non degenerate case, one only needs to add up the lengths of two maximal chains in $\mathcal{M}_S^R$ and $\widehat{\mathcal{M}}_S^R$ respectively. 
    As for cones coming from $\trop (\Gr(n-2,n)) \times (\RR^{\binom{[n]}{n-1}}/\RR + w^S)$ with $R^C \subseteq S$, we use that tropicalization preserves dimension. Then the maximal dimension of such cones is $2(n-2) + n - |S|-1$. Now for each case choose $S$ minimal to arrive at the conclusion.
\end{proof}

Consider the case $n=4$. Here the dimension of $\shortFlag^R$ is $5,5,5,6,7$, when $R$ increases in cardinality from $0$ to $4$.
More generally the dimension of $\shortFlag^R$ remains $2n-3$ for $|R| \leq 2 $, and then starts increasing. 
These jumps in dimension are to be expected: Indeed in the language of \cite{CFFFM19} the linear map $\pr_R$ is part of the flat locus iff $|R| \leq 2$, and all elements in the flat locus must have fibers of the same dimension. In particular $\shortFlag^R$ and $\shortFlag$ are of the same dimension. For $|R| > 2$ we leave the flat locus thus allowing the dimension to increase. 

\appendix
\setcounter{secnumdepth}{0}

\section{Appendix}
Let $d \leq n-2$. We state adjusted versions of \cite[Lemma 30, 31 and 32]{JMRS}, which are needed to prove Theorem \ref{theorem: algebraic statement}. Their proof as well as their use is completely analogous as in \cite{JMRS}.

\begin{lemma}
    Let $A,B,C \subseteq [n]$ with $|A| = d-1$ and $|B| = |C| = d+1$ and $S \subset [n]$. Suppose $a \in A \setminus (C \cup S)$. For $i \in B \setminus (A \setminus a)$ and $j \in (C \setminus A)a$ with $j \neq i$, denote
    \begin{align*}
        \varphi_i &= |A \cap [i]| + |B \cap [i]| - (~|Ai \setminus a \cap [a]| + |Ca \cap [a]|~) \\
        \psi_j &= \varphi_i + | (A \setminus a)i \cap [j]| + |Ca \cap[j]| - (~ |(A \setminus a)j \cap [i] | + |B\cap [i]| ~).
    \end{align*}
    Then $(-1)^\psi_j$ is independent of $i$, and
    $$
    \sum_{i \in B\setminus (A \setminus a)}(-1)^{\varphi_i}P_{B \setminus i} \cdot I^S_{(A\setminus a)i, Ca} =
    R_{A,B} \cdot Q_C + \sum_{j \in C \setminus (A \cup S)} (-1)^{\psi_j} R_{(A \setminus a)j,B} \cdot Q_{Ca \setminus j}.
    $$
\end{lemma}

\begin{lemma}
    Let $A, C \subseteq [n]$ with $|A| = d-1$ and $|C| = d+1$. Furthermore assume $S \subseteq C$. Then we have $R_{A,C} \in (\mathcal{I}_{in}^S : \langle Q_C \rangle^\infty)$.
\end{lemma}

\begin{lemma}
    Let $A,B,C \subseteq [n]$ with $|A| = d-1$, $|B| = |C| = d+1$, and $A \subset C$. Moreover let $S \subset [n]$. Let $b \in B \setminus (C \cup S)$ and $\beta = |[b] \cap B| + |[b] \cap C| + 1$. Furthermore for $i \in B \setminus A$, and $j \in C \setminus B$, denote $\varphi_i = |[i] \cap A| + |[i] \cap B| + |[b] \cap B \setminus i| + |[b] \cap C|$ and
    $\psi_j = |[j] \cap B| + |[b] \cap B| + |[j] \cap C| + |[b] \cap C|$.
    Then
    \begin{align*}
        R_{A,B} \cdot Q_C + \sum_{j \in C \setminus (B \cup S)} (-1)^{\psi_j} R_{A,(B \setminus b)j} \cdot Q_{Cb \setminus j} =& \\
        (-1)^\beta P_{B \setminus b} \cdot I_{A,Cb}^S &+ \sum_{\substack{i \in B \setminus A \\ i \neq b}}(-1)^{\varphi_i}P_{Ai} \cdot I_{B \setminus ib, Cb} ~. 
    \end{align*}
\end{lemma}

\bibliographystyle{plain}
\bibliography{shortflags.bib}

\end{document}